\documentclass[10pt]{article}

\usepackage[top=1.0in, left=1.0in, right=1.0in]{geometry}
\usepackage{graphicx}
\usepackage[all,cmtip]{xy}
\usepackage{xypic}
\usepackage{tikz,tikz-cd}

\usepackage{hyperref}
\usepackage{graphicx}
\usepackage{color}
\usepackage{amsmath}
\usepackage{amsthm}
\usepackage{amsfonts}
\usepackage{amssymb}
\usepackage{epsfig}
\usepackage{rotating}
\usepackage{graphics}

\usepackage{tikz}
\usetikzlibrary{positioning}
\usetikzlibrary{arrows}
\usetikzlibrary{shapes.multipart}

\definecolor{darkred}{rgb}{0.8,0,0}

\newtheorem{theorem}{Theorem}
\newtheorem{lemma}{Lemma}
\newtheorem{proposition}{Proposition}
\newtheorem{corollary}{Corollary}

\newtheorem{definition}{Definition}
\newtheorem{example}{Example}

\newtheorem{remark}{Remark}

\newcommand{\be}{\begin{equation}}

\newcommand{\ee}{\end{equation}}

\newcommand{\al}{\alpha}

\newcommand{\gam}{\gamma}

\newcommand{\M}{{\cal M}}

\newcommand{\U}{{\cal U}}

\newcommand{\dz}{\wedge}

\newcommand{\ba}{\begin{array}}

\newcommand{\ea}{\end{array}}

\newcommand{\beq}{\begin{eqnarray}}

\newcommand{\eeq}{\end{eqnarray}}

\newcommand{\bd}{\begin{deff}}

\newcommand{\ed}{\end{deff}}

\newcommand{\bl}{\begin{lm}}

\newcommand{\el}{\end{lm}}

\newcommand{\bp}{\begin{proo}}

\newcommand{\ep}{\end{proo}}

\newcommand{\bt}{\begin{thee}}

\newcommand{\et}{\end{thee}}

\newcommand{\bc}{\begin{co}}

\newcommand{\ec}{\end{co}}

\newcommand{\brm}{\begin{rem}}

\newcommand{\erm}{\end{rem}}

\newcommand{\der}{{\rm d}}

\usepackage{t1enc}

\newcommand{\newc}{\newcommand}

\let\ccdot\cdot
\def\cdot{\hbox to 2.5pt{\hss$\ccdot$\hss}}

\newc{\aR}{\mbox{\boldmath{$ R$}}}
\newc{\aS}{\mbox{\boldmath{$ S$}}}
\newc{\aT}{\mbox{\boldmath{$ T$}}}
\newc{\aW}{\mbox{\boldmath{$ W$}}}

\newc{\aK}{\mbox{\boldmath{$ K$}}}
\newc{\aL}{\mbox{\boldmath{$ L$}}}

\usepackage{amssymb}
\usepackage{amscd}

\newcommand{\bma}{\begin{pmatrix}}
\newcommand{\ema}{\end{pmatrix}}

\newcommand{\bet}{\beta}

\newc{\obstrn}[2]{B^{#1}_{#2}}

\newcommand{\rpl}                         
{\mbox{$
\begin{picture}(12.7,8)(-.5,-1)
\put(0,0.2){$+$}
\put(4.2,2.8){\oval(8,8)[r]}
\end{picture}$}}

\newcommand{\lpl}                         
{\mbox{$
\begin{picture}(12.7,8)(-.5,-1)
\put(2,0.2){$+$}
\put(6.2,2.8){\oval(8,8)[l]}
\end{picture}$}}

\usepackage{ifthen}

\newc{\tensor}[1]{#1}
\newc{\Mvariable}[1]{\mbox{#1}}
\newc{\down}[1]{{}_{#1}}
\newc{\up}[1]{{}^{#1}}

\newc{\JulyStrut}{\rule{0mm}{6mm}}
\newc{\midtenPan}{\mbox{\sf S}}
\newc{\midten}{\mbox{\sf T}}
\newc{\midtenEi}{\mbox{\sf U}}
\newc{\ATen}{\mbox{\sf E}}
\newc{\BTen}{\mbox{\sf F}}
\newc{\CTen}{\mbox{\sf G}}

\def\sideremark#1{\ifvmode\leavevmode\fi\vadjust{\vbox to0pt{\vss
 \hbox to 0pt{\hskip\hsize\hskip1em
 \vbox{\hsize3cm\tiny\raggedright\pretolerance10000
 \noindent #1\hfill}\hss}\vbox to8pt{\vfil}\vss}}}%

\numberwithin{equation}{section}

\newcounter{romenumi}
\newcommand{\labelromenumi}{(\roman{romenumi})}

\newcommand{\D}{\mathcal{D}}

\begin{document}

\title{The geometry of marked contact Engel structures}

\vskip 1.truecm

\author{\textsc{Gianni Manno}\\
    Dipartimento di Scienze Matematiche (DISMA),\\
    Politecnico di Torino,\\
    Corso Duca degli Abruzzi, 24,
    10129 Torino, Italy\\
    \textsf{giovanni.manno@polito.it}\\[0.5cm]
    \textsc{Pawe\l\ Nurowski}\\
    Center for Theoretical Physics PAS\\
     Al. Lotnik\'ow 32/46,
     02-668 Warsaw, Poland\\
    \textsf{nurowski@cft.edu.pl}\\[0.5cm]
    \textsc{Katja Sagerschnig}\\
     Center for Theoretical Physics PAS\\
     Al. Lotnik\'ow 32/46,
     02-668 Warsaw, Poland\\
    \textsf{katja@cft.edu.pl}}

\date{\today}

\maketitle

\begin{abstract}
A \emph{contact twisted cubic structure} $(\M,\mathcal{C},\gamma)$ is a $5$-dimensional manifold $\M$ together with a contact distribution $\mathcal{C}$ and a bundle of twisted cubics $\gamma\subset\mathbb{P}(\mathcal{C})$ compatible with the conformal symplectic form on $\mathcal{C}$.
In Engel's classical work, the Lie algebra of the exceptional Lie group $\mathrm{G}_2$ was realized as the symmetry algebra of the most symmetrical contact twisted cubic structure; we thus refer to this one as the \emph{contact Engel structure}.
In the present paper we equip the contact Engel structure with a smooth section $\sigma: \M\to\gamma$, which ``marks'' a point in each fibre $\gamma_x$.
 We study the local geometry of the resulting structures $(\M,\mathcal{C},\gamma, \sigma)$, which we call  \emph{marked contact Engel structures}.   Equivalently, our study can be viewed as a study of foliations of $\M$ by curves whose tangent directions are everywhere contained in  $\gamma$.
  We provide a complete set of local invariants of marked contact Engel structures, we classify all homogeneous models with symmetry groups of dimension $\geq 6$, and we prove an analogue of the classical Kerr theorem from relativity.
  \end{abstract}

\newpage
\tableofcontents
\newpage
\section{Introduction}

Consider a smooth $5$-dimensional  manifold $\M^5$ together with a contact distribution, i.e., a rank $4$ subbundle $\mathcal{C}\subset T\M^5$ such that the Levi bracket
\begin{equation}\label{Levi}\mathcal{L}:\Lambda^2\mathcal{C}\to T\M^5/\mathcal{C},\quad \xi_x\wedge\eta_x\mapsto [\xi,\eta]_x \ \mathrm{mod}\ \mathcal{C}_x\end{equation}
is non-degenerate at each point $x\in \M$. Then $\mathcal{L}_x$ endows each fibre $\mathcal{C}_x$ with the structure of a conformal symplectic vector space.  Locally, $\mathcal{C}$ is the kernel of a contact form, i.e.,  $\mathcal{C}=\mathrm{ker}(\alpha)$, where $\alpha\in\Omega^1(\M^5)$ satisfies $\der\alpha\wedge\der\alpha\wedge\alpha\neq 0$, and the conformal symplectic structure  on $\mathcal{C}_x$ is generated by  $\der\alpha\vert_{{\mathcal{C}}_x}$.

Now suppose  that each contact plane $\mathcal{C}_x$ is equipped with a cone $\hat{\gamma}_x\subset\mathcal{C}_x$ whose projectivization $\gamma_x\subset\mathbb{P}(\mathcal{C}_x)$ is the image of the map
$$
\mathbb{R}\mathbb{P}^1\to\mathbb{P}(\mathcal{C}_x)\cong\mathbb{R}\mathbb{P}^3,\quad [t,s]\mapsto [t^3,t^2 s,t s^2,s^3]\,;
$$
such a curve  is called a twisted cubic curve (also, rational normal curve of degree three). Moreover, assume that  $\hat{\gamma}_x$ is a Lagrangian in the sense that the tangent space at each non-zero point is a $2$-dimensional subspace of $\mathcal{C}_x$ on which the  conformal symplectic form vanishes identically. Further suppose that $\gamma=\bigsqcup_{x\in\mathcal{M}^5}\gamma_x\to\M^5$  is a subbundle of   $\mathbb{P}(\mathcal{C})\to\M^5$. Then $(\M^5,\mathcal{C},\gamma)$ is called a \emph{contact twisted cubic structure}.

 In 1893 Cartan and Engel, in the same journal but independent articles \cite{cartan, engel}, provided the first explicit realizations of the Lie algebra of the exceptional Lie group $\mathrm{G}_2$\footnote{In this paper $\mathrm{G}_2$ denotes a Lie group whose Lie algebra is the split real form of the smallest of the complex exceptional simple Lie algebras, see Section \ref{LieTheory}.}
as infinitesimal automorphisms of geometric structures on $5$-dimensional manifolds. One of these structures was the simplest maximally non-integrable rank two distribution, while the other was the simplest contact twisted cubic structure. (In other words, Cartan and Engel gave local coordinate descriptions of the geometric structures on the two $5$-dimensional homogeneous spaces $\mathrm{G}_2/\mathrm{P}_1$ and $\mathrm{G}_2/\mathrm{P}_2$ whose automorphism groups are precisely  $\mathrm{G}_2$.) Engel's description of the invariant contact twisted cubic structure
 was (up to a different choice of coordinates) as follows:
 Let $(x^0,x^1,x^2,x^3,x^4)$ be local coordinates on an open subset $\mathcal{U}\subset \mathbb{R}^5$ and consider
co-frame $(\alpha^0,\alpha^1,\alpha^2,\alpha^3,\alpha^4),$
\begin{equation}\label{coframeintro}\alpha^0=\mathrm{d}x^0+x^1 \mathrm{d}x^4- 3 x^2 \mathrm{d}x^3,\quad \alpha^1=\der x^1,\quad \alpha^2=\der x^2,\quad \alpha^3=\der x^3,\quad \alpha^4=\der x^4,\end{equation} with dual frame
 $(X_0,X_1,X_2,X_3,X_4)$,
\begin{equation}\label{frameintro}
\begin{aligned}X_0=\partial_{x^0},\quad
X_1=\partial_{x^1}, \quad X_2=\partial_{x^2},\quad X_3=3x^2\partial_{x^0}+\partial_{x^3},\quad X_4=-x^1\partial_{x^0}+\partial_{x^4}.
\end{aligned}
\end{equation}
Here $\alpha^0$ is a contact form and defines a contact distribution $\mathcal{C}=\mathrm{ker}(\alpha^0)$.
Now consider the set of horizontal null vectors
$$\hat{\gamma}=\{\,Y\in\mathcal{C}:\,g_1(Y,Y)=g_2(Y,Y)=g_3(Y,Y)=0\,\}$$
 of the  three degenerate metrics \footnote{We refer to the tensor fields $g_1, g_2, g_3\in\Gamma(\smash{\bigodot^2T^*\mathcal{U}})$ as metrics, although strictly speaking these are not metrics, since all three of them are indeed degenerate, and  $g_1$ and $g_2$ are even degenerate when restricted to the distribution $\mathcal{C}$.}
\begin{equation}\begin{aligned}\label{cone}
  g_1=\alpha^1\alpha^3- (\alpha^2)^2,\quad
  g_2= \alpha^2\alpha^4- (\alpha^3)^2,\quad
  g_3= \alpha^2\alpha^3- \alpha^1 \alpha^4,
\end{aligned}
\end{equation}
where  $\alpha^i\alpha^j=\tfrac{1}{2}(\alpha^i\otimes\alpha^j+\alpha^j\otimes\alpha^i)$. Then $Y\in\Gamma(\mathcal{C})$ takes values in $\hat{\gamma}$ if and only if  is of the form $$Y=t^3  X_1+t^2 s X_2+ ts^2X_3+s^3X_4.$$
Hence  the projectivization $\gamma_x\subset\mathbb{P}(\mathcal{C}_x)$ of $\hat{\gamma}_x$ is a twisted cubic curve, and it is straightforward to verify that  $\hat{\gamma}_x\subset\mathcal{C}_x$ is Lagrangian.
 We shall call the structure $(\mathcal{U},\mathcal{C},\gamma)$
 the \emph{contact Engel structure} in view of Engel's classical work.\footnote{Contact Engel structures should not be confused with Engel distributions, sometimes also called Engel structures, which are maximally non-integrable rank $2$ distributions on $4$-dimensional manifolds.}


The contact Engel structure is the \emph{flat model} for contact twisted cubic structures in the following sense.
One can show that a contact twisted cubic structure is the underlying structure of a certain type of Cartan geometry, more specifically parabolic geometry, see \cite{tanaka, book}. 
As such it admits a \emph{canonical Cartan connection}, which has in general  \emph{nonzero curvature}. There is a unique, up to a local equivalence, contact twisted cubic structure
whose curvature vanishes identically.
This is the one we call the contact Engel structure.

A specialization of contact twisted cubic structures can go \emph{independently} in other directions. For example, instead of imposing restrictions on the curvature
of a given contact twisted cubic structure, one can restrict its structure group by
adding more structure.
The structure group of the corresponding enriched geometry must preserve this additional structure, and gets reduced. We will explain below that a natural choice for such a reduction is to add a section
$$\sigma:\mathcal{M}^5\to\gamma\subset \mathbb{P}(\mathcal{C})$$
of the bundle
 $\,\mathbb{R}\mathbb{P}^1\to\gamma\to \M^5$ of twisted cubics to the geometric structure.
 Since  such a section $\sigma$ marks a point $*=\sigma(x)$ in each twisted cubic $\gamma_x$, $x\in \mathcal{M}^5$, we refer to the enriched structure $(\M^5,\mathcal{C},\gamma,\sigma)$
 as a \emph{marked contact twisted cubic structure}. If the underlying contact twisted cubic structure is flat, then the resulting structure will be called a \emph{marked contact Engel structure}.

One may think of a marked contact twisted cubic structure as a foliation of a contact twisted cubic structure by special horizontal curves.
 Suppose we are given a marked contact twisted cubic structure $(\M^5,\mathcal{C},\gamma,\sigma)$. For each $x\in \M^5$, the point $\sigma(x)\in\gamma_x$ corresponds to a direction $\ell^{\sigma}_x$ in the contact plane $\mathcal{C}_x$.
Therefore, the section $\sigma$ defines a rank one distribution $\ell^{\sigma}\subset T\M^5$ whose integral manifolds define a foliation  of $\mathcal{M}^5$ by  curves (a \emph{congruence}). Conversely, a congruence on $\M^5$ by  curves whose tangent directions are everywhere contained in $\gamma\subset\mathbb{P}(\mathcal{C})$ uniquely determines a section $\sigma:\mathcal{M}^5\to\gamma$.  Since $\gamma_x\subset\mathbb{P}(\mathcal{C}_x)$ is cut out by three polynomials, the congruences corresponding to sections $\sigma:\M^5\to\gamma$ can be also seen as null congruences.

\subsection{Context and motivation}
Before we outline the main results of this paper, a few words of motivation are in order:

It follows from the above brief description that the marked contact Engel structures, or their more general cousins, the marked contact twisted cubic structures, are \emph{special} contact twisted cubic structures. This places the area of our present study in the context of \emph{special geometries}, which are mostly developed in Riemannian geometry. For example, similarly to the addition of a section $\sigma$ to a contact twisted cubic structure $(\M^5,{\mathcal C},\gamma)$, one can add an almost Hermitian structure $\mathcal{J}$ to an even-dimensional Riemannian manifold $(\M^{2n},g)$. In this way one passes from the Riemmannian geometry $(\M^{2n},g)$ to the \emph{special} Riemannian geometry (almost Hermitian geometry) $(\M^{2n},g,\mathcal{J})$, as we are passing from $(\M^5,{\mathcal C},\gamma)$ to the special geometry $(\M^5,{\mathcal C},\gamma,\sigma)$.

The analogy between our marked contact Engel structures and special geometries is particularly striking if we replace Riemannian geometry by \emph{conformal Lorentzian geometry in 4-dimensions} $(\M^4,[g])$. These are the geometries studied in General Relativity, when the related physics is concerned with massless particles only.
Of particular importance in General Relativity are \emph{null congruences}, i.e. \emph{foliations} of $(\M^4,[g])$ \emph{by null curves}. Suppose that we have such a congruence on $(\M^4,[g])$. Let $K\subset T\M^4$ be  the null line subbundle such that any section $s: \M^4\to K$ be tangent to the congruence. Then we have a \emph{special} Lorentzian conformal geometry $(\M^4,[g],K)$, which we call a \emph{null congruence structure}. One can study the local equivalence problem of such geometries, where two null congruence structures $(\M^4_i,[g_i],K_i)$, $i=1,2$, are locally equivalent if and only if there exists a local diffeomorphism $\phi:\M^4_1\to \M^4_2$ such that $\phi^*(g_2)=f^2 g_1$, $\phi^*(K_2)=K_1$, with $f$ a non-vanishing function on $\M^4_1$. One very quickly establishes that there are locally non-equivalent null congruence structures even if both conformal structures are conformally flat. For example, if the curves of one null congruence are \emph{geodesics} (this is a conformally invariant property) and the curves of the other one are not, the two congruences are locally non-equivalent. Even if we have two null congruences such that both are weaved by geodesics, they are still  in general  \emph{not} locally equivalent. The next important conformally invariant property distinguishing locally non-equivalent structures  is  \emph{shearfreeness} \cite{Robinson}, see \cite{GHN, HJN, nurTraut, ArmanTwistor, Curtis} for more details. So here is our analogy:\\
\begin{center}
\begin{tabular}{|c|c|}
  \hline
  conformal spacetime &contact twisted cubic structure\\ $(\M^4,[g])$  &  $(\M^5,{\mathcal C},\gamma)$\\
  \hline
  conformally flat spacetime & Engel structure\\
  \hline
  null congruence & marked contact twisted cubic \\structure $(\M^4,[g],K)$&  structure $(\M^5,{\mathcal C},\gamma,\sigma)$\\
  \hline
  conformally flat null &marked contact \\ congruence structure & Engel structure\\
  \hline
  conformally flat null & integrable marked contact \\congruence structure & Engel structure \\  of geodesics & \\     \hline
  conformally flat null &integrable marked contact \\ congruence structure & Engel structure \\
  of shearfree geodesics&  \\\hline
   Kerr theorem&Kerr theorem for contact\\& Engel structures\\ \hline
\end{tabular}
\end{center}
The relevance of the integrability condition on marked contact Engel structure, which appears in the above Table,
will be explained in Section \ref{SecKerr}. Here we only mention that in our analogy it is related to the celebrated Kerr theorem of General Relativity, see \cite{penroserindler2, Tafel}, which gives a construction of all null congruence structures of shearfree geodesics that can live in conformally flat spacetimes. This theorem is the origin of Penrose's \emph{twistor theory} \cite{RP}. The analogy described above shows that it has a well defined interesting counterpart for marked contact Engel structures.

\subsection{Structure and main results of the article}

Section \ref{sec_marked}  introduces the notions of a contact twisted cubic structure, Engel structure, marked contact twisted cubic structure and marked contact Engel structure. First observations about these structures are presented.
In particular, the so-called ``osculating filtration'' determined by a marked contact twisted cubic structure is introduced: This is a filtration of the contact bundle $\mathcal{C}$ by distributions
$$\ell^{\sigma}\subset\mathcal{D}^{\sigma}\subset\mathcal{H}^{\sigma}\subset\mathcal{C},$$
with respective ranks $1$, $2$, $3$, $4$, where $\mathcal{D}^{\sigma}$ is a \emph{Legendrian} rank two distribution.
It   corresponds fibre-wise to the osculating sequence
of the twisted cubic $\gamma_x\subset\mathbb{P}(\mathcal{C}_x)$ at a point $\sigma(x)$.
We call a marked contact twisted cubic structure (respectively the section $\sigma$) \emph{integrable} if  the distribution $\mathcal{D}^{\sigma}$  is integrable.

The core of the present paper is Section \ref{CartanEquiv}, where  we apply Cartan's method of equivalence to study the local equivalence problem  of  marked contact Engel structures. Throughout this paper, we shall refer to  the set of \emph{all} vector fields  preserving a given marked contact Engel structure as \emph{the infinitesimal symmetry algebra}, or simply \emph{the symmetry algebra}, of the marked contact Engel structure. We shall denote by $\mathfrak{g}$ the Lie algebra of the exceptional Lie group $\mathrm{G}_2$.\footnote{We chose to denote the Lie algebra of the Lie group $\mathrm{G}_2$  by $\mathfrak{g}$ in order to avoid confusion with a certain grading component that is commonly denoted by $\mathfrak{g}_2$.}
\begin{itemize}
\item We show that there exists a  (locally unique) maximally symmetric  model for marked contact Engel structures. Its symmetry algebra is isomorphic to  the  $9$-dimensional parabolic subalgebra  $\mathfrak{p}_1$ of  $\mathfrak{g}$ that may be realized as the stabilizer of a highest weight line in the $7$-dimensional irreducible representation of $\mathfrak{g}$ on $\mathbb{R}^{3,4}$ (Theorem \ref{main}).

\item We provide an explicit construction of a unique coframe (absolute parallelism)  on a $9$-dimensional  bundle naturally associated with any marked contact Engel structure (Proposition \ref{propo_coframe}). Differentiating this coframe yields a complete set of local invariants for marked contact Engel structures.

\item In particular, we obtain a filtration of  differential conditions for marked contact Engel structures, where the first is the integrability condition described above, and the last is equivalent to flatness, i.e., to local equivalence with  the aforementioned maximally symmetric model (Theorem \ref{main}).

\item We systematically use the filtration of invariant conditions to classify, up to local equivalence,  all homogeneous marked contact Engel structures whose  symmetry algebra is of dimension $\geq 6$. Our analysis  shows that there are precisely two locally non-equivalent homogeneous marked contact Engel structures whose symmetry algebras are $8$-dimensional (they are $\mathfrak{sl}(3,\mathbb{R})$ and $\mathfrak{su}(1,2)$). Moreover, we provide differential conditions characterizing these sub-maximally symmetric marked contact Engel structures. We show that there are no homogeneous marked contact Engel structures with $7$-dimensional symmetry algebra, and that there are precisely two locally non-equivalent homogeneous marked contact Engel structures with $6$-dimensional  symmetry algebra (one of them is semisimple and isomorphic to $\mathfrak{sl}(2,\mathbb{R})\oplus\mathfrak{sl}(2,\mathbb{R})$). We provide examples of locally non-equivalent homogeneous marked contact Engel structures with $5$-dimensional  symmetry algebra as well. These results are summarized in Theorem \ref{thmsummary}, see also Table \ref{table}.

\end{itemize}

Sections \ref{sec_123} and \ref{SecKerr} provide geometric interpretations of some of the invariant properties of contact Engel structures derived in Section \ref{CartanEquiv}.  In particular, the central notion of integrability will be revisited.

In Section \ref{SecKerr} we prove an analogue of the Kerr Theorem (Theorem \ref{Kerr1}), which provides a construction method of all integrable marked contact Engel structures. We subsequently recast the result in terms of the double filtration for the exceptional Lie group $\mathrm{G}_2$:
\begin{align}\label{doublefib}
   \xymatrix{
        &\mathrm{G}_2/\mathrm{P}_{1,2}  \ar[dl]_{\pi_2} \ar[dr]^{\pi_1} & \\
                   \mathrm{G}_2/\mathrm{P}_2 &            & \mathrm{G}_2/\mathrm{P}_1\ . }
\end{align}
Here $\mathrm{P}_1$ and $\mathrm{P}_2$ are the $9$-dimensional parabolic subgroups  of $\mathrm{G}_2$ and $\mathrm{P}_{1,2}=\mathrm{P}_1\cap \mathrm{P}_2$ is the $8$-dimensional Borel subgroup of $\mathrm{G}_2$.
The contact Engel structure  is a  local coordinate description of the $\mathrm{G}_2$-invariant structure on the $5$-dimensional  space  $\mathrm{G}_2/\mathrm{P}_2$.
  The total space of the $\mathbb{R}\mathbb{P}^1$-bundle $\gamma\to \mathrm{G}_2/\mathrm{P}_2$ can be identified with the $6$-dimensional homogeneous space $\mathrm{G}_2/\mathrm{P}_{1,2}$.
Marked contact Engel structures can be identified with local sections  $\sigma$ of the first leg in the double fibration,
$$\mathrm{G}_2/\mathrm{P}_2\supset\U\xrightarrow{\sigma} \sigma(\U)\subset\mathrm{G}_2/\mathrm{P}_{1,2}.$$
The image of such a section defines a hypersurface in $\mathrm{G}_2/\mathrm{P}_{1,2}$, which   descends to a hypersurface in the second $5$-dimensional homogeneous space $\mathrm{G}_2/\mathrm{P}_1$ if and only if $\sigma$ is integrable. This yields a local one-to-one  correspondence between integrable sections and generic hypersurfaces in $\mathrm{G}_2/\mathrm{P}_1$ (Corollary \ref{Kerr2} of Theorem \ref{Kerr1}). The correspondence is then used  to describe the maximal and submaximal marked contact Engel structures; these correspond to the simplest hypersurfaces in $\mathrm{G}_2/\mathrm{P}_1$,
namely, identifying $\mathrm{G}_2/\mathrm{P}_1$ with the projectivized null cone in $\mathbb{R}^{3,4}$, they correspond to intersections of the null cone with hyperplanes in $\mathbb{R}^{3,4}$.

 Section \ref{sec_general} provides a first analysis  of
 general marked contact twisted cubic structures. Following the general framework due to Tanaka, see \cite{tanaka, Morimoto, zelenko}, they are  viewed as particular types of filtered $G_0$-structures in this section. We compute the (algebraic) Tanaka prolongation associated with these structures, which implies the existence of a canonical coframe on a $9$-dimensional bundle associated with any marked contact twisted cubic structure in a natural manner.
Finally, we investigate the question whether a normalization  condition in the sense of \cite{CapCartan} can be found. We prove that this is not the case, and thereby provide an example of a structure where such a normalization condition does not exist.

\

\subsection{Conventions and Notation}\label{notation}
Throughout the paper all of our objects are smooth, all of our considerations are local and it follows from the context which neighbourhoods are taken into account.

We use the notations
\begin{equation}\label{abbriv}E^1E^2\dots E^k=E^1\odot E^2\odot \dots \odot E^k=\smash{\tfrac{1}{k!}\sum_{\sigma\in S_k}(E^{\sigma 1}\otimes E^{\sigma 2}\otimes\dots\otimes E^{\sigma k})},\end{equation}
where $S_k$ is the symmetric group of degree $k$, for the symmetrized tensor product.

For a general coframe $(\omega^i)$ we write $F_{\omega^i}$ for the derivatives with respect to the coframe, i.e.,  $\der F=\sum_{i} F_{\omega^i}\omega^i$. If we consider a coordinate coframe $(\der x^i) $, we simply write $F_{x^i}$.

\subsection{Acknowledgements}
Discussions with Jan Gutt and Giovanni Moreno were influential for the development of this project. Moreover, we acknowledge fruitful discussions with many of the participants of the Simon's semester \emph{Symmetry and Geometric Structures} hosted by the Institute of Mathematics of the Polish Academy of Sciences in Warsaw. In particular, we thank Andreas \v Cap, Michael Eastwood,  Wojciech Krynski, Tohru Morimoto, Katharina Neusser and  Arman Taghavi-Chabert for helpful comments.

This work was  supported by the Simons Foundation grant 346300 and the Polish Government MNiSW 2015--2019 matching fund.
All the authors acknowledge support from the project ``FIR-2013 Geometria delle equazioni differenziali''. In particular, Katja Sagerschnig was an INdAM-research fellow funded by that project during the period 01/02/2016 - 31/12/2017, where part of the present research was accomplished. All the authors were also partially supported by the ``Starting grant per giovani ricercatori $53\_RSG16MANGIO$'', Politecnico di Torino. Gianni Manno is a member of GNSGA of INdAM.
Katja Sagerschnig also acknowledges support  by the Polish National Science Center (NCN) via the
POLONEZ  grant  2016/23/P/ST1/04148.
This project has received funding
from the European Union's Horizon 2020 research and innovation
programme under the Marie Sk\l odowska-Curie grant agreement No
665778.\

\section{Marked contact twisted cubic structures}\label{sec_marked}
Marked contact twisted cubic structures are  $5$-dimensional contact structures equipped with additional geometric structures, and we shall introduce these additional geometric structures in the following section. We shall start with purely pointwise considerations, that is, facts about Legendrian twisted cubics in a conformal symplectic vector space in Section \ref{algebra}. Then we will define and discuss general contact twisted cubic structures and marked contact twisted cubic structures on $5$-dimensional manifolds in Section \ref{contact_twisted_section}. We shall introduce the notion of an \emph{integrable} marked contact twisted cubic structure. Finally, we will focus on marked contact twisted cubic structures whose underlying contact twisted cubic structure is flat, which we call marked contact Engel structures, in Section \ref{LieTheory}.

\subsection{Preliminaries on Legendrian twisted cubics}\label{algebra}
We shall first collect some algebraic background. References are e.g. \cite{Harris, bu}.

The \emph{twisted cubic} (rational normal curve of degree three)  $\gamma\subset \mathbb{R}\mathbb{P}^3$ is the image of the Veronese map
\begin{equation}\textstyle{\mathbb{R}\mathbb{P}^1=\mathbb{P}(\mathbb{R}^2)\to\mathbb{P}(\smash{\bigodot^3\mathbb{R}^2})=\mathbb{R}\mathbb{P}^3,\quad [w]\mapsto [w\odot w\odot w].}\end{equation}
In coordinates with respect to bases $(e_1,e_2)$ of $\mathbb{R}^2$ and $(E_1=e_1\odot e_1\odot e_1, E_2=3 e_1\odot e_1\odot e_2, E_3=3 e_1\odot e_2\odot e_2, E_4=e_2\odot e_2\odot e_2)$ of $\smash{\bigodot^3\mathbb{R}^2}$  it is of the form
$$\gamma=[s^3,s^2 t,s t^2,t^3].$$
Denoting by $(E^1, E^2, E^3, E^4)$ the dual basis, the twisted cubic is also given by the zero set of the three quadratic forms
 \begin{equation}\label{3polynomials} g_1=E^1E^3-(E^2)^2,\quad g_2=E^2 E^4-(E^3)^2, \quad g_3=E^2E^3- E^1E^4.\end{equation}
With respect to the  introduced bases, the irreducible representation
\begin{equation}\label{rho}
\textstyle{\rho:\mathrm{GL}(2,\mathbb{R})\to\mathrm{End}(\mathbb{R}^4)},\quad\mathbb{R}^4=\smash{\bigodot^3\mathbb{R}^2},
\end{equation}
is of the form \begin{equation} \label{irrepres}
 \bma\alpha&\beta\\\rho&\delta\ema \mapsto
\bma\alpha^3&3\alpha^2\beta&3\alpha\beta^2&\beta^3\\
\alpha^2\rho&\alpha^2\delta+2\alpha\beta\rho&2\alpha\beta\delta+\beta^2\rho&\beta^2\delta\\
\alpha\rho^2&2\alpha\delta\rho+\beta\rho^2&\alpha\delta^2+2\beta\delta\rho&\beta\delta^2\\
\rho^3&3\delta\rho^2&3\delta^2\rho&\delta^3\ema.\end{equation}
The tangent map at the identity of \eqref{rho} defines the irreducible Lie algebra representation
$$\rho'=T_e\rho:\mathfrak{gl}(2,\mathbb{R})\to \mathrm{End}(\mathbb{R}^4).$$
One can check the following.
 \begin{proposition}\label{subalg}
 The subalgebra of $\mathrm{End}(\mathbb{R}^4)$ preserving $\gamma\subset\mathbb{P}(\mathbb{R}^4)$ is precisely $\rho'\left(\mathfrak{gl}(2,\mathbb{R})\right)$.
\end{proposition}

The
 decomposition
$\bigwedge^2(\smash{\bigodot^3\mathbb{R}^2})\cong \bigodot^4\mathbb{R}^2\oplus \mathbb{R}$
shows that there is a unique (up to scalars) skew-symmetric bilinear form on $\smash{\bigodot^3\mathbb{R}^2}$ preserved by the  $\mathrm{GL}(2,\mathbb{R})$-action up to scalars. Explicitly, it is given by \begin{equation}\label{symp}\omega=E^1\wedge E^4 - 3 E^2\wedge E^3
.\end{equation}
In order to  characterize the $\mathrm{GL}(2,\mathbb{R})$-invariant conformal class of  the symplectic form \eqref{symp} in terms of the twisted cubic, we shall introduce some more terminology:
Let $\omega$ be a symplectic form on $\mathbb{R}^4$ and let $[\omega]$ be the conformal class of all non-zero multiples of $\omega$.
Recall that a maximal subspace $\mathbb{W}$ on which a symplectic form $\omega$ (and then any $\omega'\in[\omega]$) vanishes identically is called \emph{Lagrangian}. A  twisted cubic $\gamma\subset\mathbb{P}(\mathbb{R}^4)$ is called \emph{Legendrian} with respect to $[\omega]$, see \cite{bu},  if the  cone $$\hat{\gamma}=\{\,w\odot w\odot w : w\in\mathbb{R}^2\,\}\subset \mathbb{R}^4$$ is  Lagrangian, i.e., the tangent space  at each point $p$ of $\hat{\gamma}\setminus \{0\}$  is a Lagrangian subspace of $T_p\mathbb{R}^4\cong\mathbb{R}^4$.
\begin{proposition}\label{confsymp}
The
conformal symplectic structure $[\omega]$ generated by $\omega=E^1\wedge E^4 - 3 E^2\wedge E^3$
is the unique conformal symplectic structure
 such that $\gamma=[s^3,s^2 t,s t^2,t^3]$ is Legendrian with respect to $[\omega]$.
\end{proposition}
\begin{proof}
The tangent space to $\hat{\gamma}$ at a point
\begin{align}\label{point}\hat{p}={s}^3E_1+{s}^2{t}E_2+{s}{t}^2E_3+{t}^3E_4\end{align} is spanned by
\begin{align}\label{X,Y}
X= 3{s}^2 E_1+ 2{s}{t} E_2+ {t}^2 E_3\quad \mathrm{and} \quad Y={s}^2 E_2+ 2{s}{t} E_3+ 3{t}^2 E_4.
\end{align}
Let $\omega=\tfrac12\omega_{ij}E^i\wedge E^j$ be a symplectic form, then
\begin{align*}
\omega(X,Y)= 3{s}^4 \omega_{12} + 6{s}^3{t} \omega_{13}+ 3{s}^2{t}^2 (3 \omega_{14}+  \omega_{23}) + 6{s}{t}^3\omega_{24}+ 3{t}^4\omega_{34}.
\end{align*}
Hence $\gamma$ is Legendrian with respect to $\omega$ if and only if
\begin{align*}\omega_{14}=-\tfrac13\omega_{23}\end{align*}
and, modulo antisymmetry,  the remaining $\omega_{ij}$ vanish. This  determines $\omega$ uniquely up to scale.
\end{proof}

We now introduce some additional data. Namely, we suppose the twisted cubic is marked, that is, a point $p\in\gamma\subset\mathbb{P}(\mathbb{R}^4)$ is distinguished. The point $p$ corresponds to a line $\ell\subset\hat{\gamma}\subset\mathbb{R}^4$. Such a line is of the form
$\ell=\mathrm{Span}(\{l\odot l\odot l :l\in L\})$ for a unique $1$-dimensional subspace $L \subset\mathbb{R}^2$. Clearly $\mathrm{GL}(2,\mathbb{R})$ acts transitively on $\gamma$ and we may choose our line $\ell$ to be spanned by the first basis vector $e_1\odot e_1\odot e_1$. The stabilizer subgroup
\begin{equation}\label{borel}
B:=\{A\in\mathrm{GL}(2,\mathbb{R}): \rho (A)(\ell)\subset\ell\}=\{A\in\mathrm{GL}(2,\mathbb{R}): A(L)\subset L\}
\end{equation}  is a Borel subgroup $B\subset \mathrm{GL}(2,\mathbb{R})$; in the presentation \eqref{irrepres}, it is given by those matrices for which $\gamma=0$. In particular, $B$ preserves a full filtration of $\mathbb{R}^4$. This immediately implies:

\begin{lemma}\label{lemfilt}
A distinguished point $p\in\gamma$ determines a filtration by subspaces
\begin{align}
\ell\subset\mathrm{D}\subset \mathrm{H}\subset\mathbb{R}^4.
\end{align}
If $\gamma$ is Legendrian, then $\mathrm{D}$ is a Lagrangian subspace  and $\mathrm{H}$ is the symplectic orthogonal to $\ell$.
\end{lemma}
In terms of  $\mathbb{R}^4=\smash{\bigodot^3\mathbb{R}^2}$,  $\mathrm{D}=\mathrm{Span}(\{l\odot l\odot e$\,: $l\in L, \,e\in\mathbb{R}^2\})$, and  $\mathrm{H}=\mathrm{Span}(\{l\odot e\odot f$\,: $l\in L, \,e,f\in\mathbb{R}^2\})$.
Geometrically,  $\mathrm{D}$ is the de-projectivized tangent line to $\gamma$ at $p$ and $\mathrm{H}$ is the de-projectivized osculating plane to $\gamma$ at $p$. Thus we refer to the above filtration as the \emph{osculating sequence} at  $p$.

\begin{remark}\label{rem3metrics}
We underline that we need all the three quadratic forms $g_1, g_2, g_3$ in \eqref{3polynomials} to define a twisted cubic $\gamma$. In fact, the common zero locus in $\mathbb{R}\mathbb{P}^3$ of any two of the quadric forms belonging to $\mathrm{Span}(g_1, g_2, g_3)$  gives a twisted cubic plus a line (the so called residual intersection, see \cite{Harris}). In the present paper we are interested in the case when this line is tangent to the twisted cubic. The point of tangency   is the distinguished point $p\in\gamma$.
\end{remark}

\subsection{Definitions and descriptions of (marked) contact twisted cubic structures}\label{contact_twisted_section}

We are now in the position to define the central objects of this article.
\begin{definition} \label{contact_twisted_def} A \emph{contact twisted cubic structure} on a $5$-dimensional smooth manifold $\M$ is  a contact distribution $\mathcal{C}\subset T\M$ together with a sub-bundle $\gamma\subset \mathbb{P}(\mathcal{C})$ whose fibre $\gamma_x$ at each point $x\in \M$ is a Legendrian twisted cubic with respect to the conformal symplectic structure $\mathcal{L}_x$ on $\mathcal{C}_x$.
An equivalence between contact twisted cubic structures $(\M,\mathcal{C},\gamma)$ and $(\widetilde{\M},\widetilde{\mathcal{C}},\widetilde{\gamma})$ is a diffeomorphism $f:\M\to\widetilde{\M}$ such that
$f_*(\mathcal{C}_x)=\widetilde{\mathcal{C}}_{f(x)}$ and $f_*(\gamma_x)=\widetilde{\gamma}_{f(x)}$ for all $x\in \M$. A self equivalence is called an automorphism, or a symmetry.
\end{definition}

\begin{definition}
A \emph{marked contact twisted cubic structure} is a contact twisted cubic structure  equipped with a smooth section $\sigma$ of $\gamma\to \M$.
 An equivalence between marked contact twisted cubic structures $(\M,\mathcal{C},\gamma,\sigma)$ and $(\widetilde{\M},\widetilde{\mathcal{C}},\widetilde{\gamma},\widetilde{\sigma})$ is an equivalence $f$ between the underlying contact twisted cubic structures $(\M,\mathcal{C},\gamma)$ and $(\widetilde{\M},\widetilde{\mathcal{C}},\widetilde{\gamma})$ such that $f_*(\sigma_x)=\widetilde{\sigma}_{f(x)}$ for all $x\in \M$. A self equivalence is called an automorphism, or a symmetry.
\end{definition}

Throughout this paper we will use various, locally equivalent,   viewpoints on (marked) contact twisted cubic structures, which we shall summarize in Propositions \ref{prop_twisted} and \ref{prop_punctured}. Yet another important description, in terms of \emph{adapted coframes}, shall be given in Section \ref{adaptedcoframes}.

Before stating the Propositions, we recall that the  $5$-dimensional  Heisenberg Lie algebra is the graded nilpotent Lie algebra  $\mathfrak{m}=\mathfrak{m}_{-1}\oplus\mathfrak{m}_{-2}$, where $\mathfrak{m}_{-1}\cong\mathbb{R}^4$, $\mathfrak{m}_{-2}\cong\mathbb{R}$,  and the only non-trivial component of the Lie bracket $[,]:\Lambda^2\mathfrak{m}_{-1}\to\mathfrak{m}_{-2}$ defines a non-degenerate skew-symmetric bilinear form.
 It then follows from non-degeneracy of the Levi bracket \eqref{Levi} that the associated graded $\mathrm{gr}(T\M)=\mathcal{C}\oplus T\M/\mathcal{C}$ of the contact structure $\mathcal{C}\subset T\M$ equipped with the Levi bracket $\mathcal{L}$  is a bundle of graded nilpotent Lie algebras modeled on the Heisenberg Lie algebra $\mathfrak{m}$.
  It has an associated   graded frame bundle $\mathcal{F}\to \M$  with structure group the grading preserving Lie algebra automorphisms $\mathrm{Aut}_{gr}(\mathfrak{m})\cong\mathrm{CSp}(2,\mathbb{R})$ of $\mathfrak{m}$; its fibre $\mathcal{F}_x$, at each point $x\in  \M$, comprises all graded Lie algebra isomorphisms $\varphi:\mathrm{gr}(T_x\M)\to\mathfrak{m}$.

\begin{proposition}\label{prop_twisted}
A contact twisted cubic structure on a $5$-dimensional manifold $\M$, locally, admits the following locally equivalent descriptions:
\begin{enumerate}
\item It is given by a contact distribution $\mathcal{C}\subset T\M$, an auxiliary rank $2$ bundle $\mathcal{E}\to \M$ and a vector bundle isomorphism
\begin{equation}\label{Psi}\Psi: \textstyle{\bigodot^3\mathcal{E}}\to \mathcal{C}\end{equation}
compatible in the sense that it pulls back the conformal symplectic structure $\mathcal{L}_x$ on $\mathcal{C}_x$ to the $\mathrm{GL}(\mathcal{E}_x)$-invariant one on $ \smash{\bigodot^3\mathcal{E}_x}$ for all $x\in \M$.
\item It is given by a reduction of the graded frame bundle  $\mathcal{F}\to \M$ of a contact structure
 to the structure group $\rho(\mathrm{GL}(2,\mathbb{R}))$
with respect to an irreducible representation $\rho:\mathrm{GL}(2,\mathbb{R})\to\mathrm{CSp}(2,\mathbb{R})$.

\item It is given by a contact distribution $\mathcal{C}=\mathrm{ker}(\alpha)$ on $\M$ and a reduction of the structure group of the frame bundle of $\mathcal{C}$ from $\mathrm{GL}(4,\mathbb{R})$ to the irreducible $\mathrm{GL}(2,\mathbb{R})\subset \mathrm{CSp}(\mathrm{d}\alpha)$.
\end{enumerate}
\end{proposition}

We only sketch the proof. Given  an isomorphism \eqref{Psi},
  the image of  the map
\begin{align*}
\iota:\mathbb{P}^1=\mathbb{P}(\mathcal{E}_x)\to\mathbb{P}(\textstyle{\bigodot^3\mathcal{E}_x})\cong\mathbb{P}(\mathcal{C}_x),\quad [\lambda]\mapsto [\Psi(\lambda\odot\lambda\odot\lambda)],
\end{align*}
 is a twisted cubic $\gamma_x$. By the compatibility requirement of the conformal symplectic structures and  Proposition \ref{confsymp}, the twisted cubic is Legendrian.

 Conversely, given  a sub-bundle $\gamma\subset \mathbb{P}(\mathcal{C})$ of twisted cubics, then in a neighbourhood of  each point there exists a rank $2$ bundle $\mathcal{E}$ and a vector bundle isomorphism $\Psi: \smash{\bigodot^3\mathcal{E}\cong\mathcal{C}}$.   The compatibility of the conformal symplectic structures follows from the fact that the twisted cubic is Legendrian and by Proposition \ref{confsymp}.

 The equivalence between the first and the second description is explained in \cite{book}. The equivalence of the second and third follows from the fact that any graded Lie algebra automorphism of $\mathfrak{m}$ is uniquely determined by its restriction to $\mathfrak{m}_{-1}$.

\begin{remark}
A  contact twisted cubic structure is the natural contact analogue of an irreducible $\mathrm{GL}(2,\mathbb{R})$-structure in dimension four, as studied, for instance, in \cite{bryant, nurowski}. In particular, one could also call it an irreducible $\mathrm{GL}(2,\mathbb{R})$-contact structure.

 Contact twisted cubic structures are also known as a $\mathrm{G}_2$-contact structure in the literature,
since they are the underlying structures of  regular, normal parabolic geometries of type $(\mathrm{G}_2,\mathrm{P}_2)$, see \cite{book}.
\end{remark}

\begin{proposition}\label{prop_punctured}
A marked contact twisted cubic structure, locally, admits the following locally  equivalent descriptions:
\begin{enumerate}
\item It is given by a contact distribution $\mathcal{C}\subset T\M$, an auxiliary rank $2$ bundle $\mathcal{E}\to \M$,  a vector bundle isomorphism
$\Psi:\smash{\bigodot^3\mathcal{E}}\to \mathcal{C}$ compatible with the conformal symplectic structures and, in addition, a line subbundle $L\subset \mathcal{E}.$
\item It is given by  a  reduction of structure group of the graded frame bundle $\mathcal{F}\to \M$ of a contact structure in dimension $5$ with respect to the restriction
$$\rho:B\to\mathrm{CSp}(2,\mathbb{R})$$
of an irreducible $\mathrm{GL}(2,\mathbb{R})$-representation $\rho$ to the Borel subgroup $B\subset\mathrm{GL}(2,\mathbb{R})$.
\item It is given by a contact twisted cubic structure equipped with a $\gamma$-congruence, that is,   a foliation  of $\M$ by curves whose tangent directions are everywhere contained in  $\gamma\subset\mathbb{P}(\mathcal{C})$.
\end{enumerate}
%
\end{proposition}
In view of Proposition \ref{prop_twisted}, the equivalence of the first two descriptions is obvious. Concerning the last description, note that  a section $\sigma: \M\to\gamma$ is the same  as a rank $1$ distribution $\ell^{\sigma}\subset\hat{\gamma}\subset\mathcal{C},$ where $\hat{\gamma}\subset\mathcal{C}$ is the cone over $\gamma\subset\mathbb{P}(\mathcal{C})$. The integral manifolds of this line distribution define the $\gamma$-congruence. Conversely, one obtains $\ell^{\sigma}$ from the $\gamma$-congruence by considering the field of tangent directions to the curves.

By Lemma \ref{lemfilt}, we have the following ``osculating filtration''.
\begin{proposition}\label{propfilt}
A marked contact twisted cubic structure $(\M,\mathcal{C},\gamma,\sigma)$ is equipped with a  flag of distributions
\begin{align}\label{filtration}
\ell^{\sigma}\subset\mathcal{D}^{\sigma}\subset \mathcal{H}^{\sigma}\subset\mathcal{C}\subset T\M,
\end{align}
where the rank $2$ distribution $\mathcal{D}^{\sigma}\subset\mathcal{C}$ is Legendrian (i.e., totally null with respect to the conformal symplectic structure on $\mathcal{C}$)  and the rank $3$ distribution $\mathcal{H}^{\sigma}$ is the symplectic orthogonal  to $\ell^{\sigma}$.
\end{proposition}

\begin{definition}
We call a   marked contact twisted cubic structure  \emph{integrable} if the distribution $\mathcal{D}^{\sigma}$ is integrable. In this case the section $\sigma:\M\to\gamma$  is called an integrable section.
\end{definition}

\subsection{(Marked) contact Engel structures and the exceptional Lie group $\mathrm{G}_2$}
\label{LieTheory}
As mentioned in the introduction, the most symmetric contact twisted cubic structure, that we refer to as the contact Engel structure, is intimately related to the exceptional Lie group $\mathrm{G}_2$. We shall explain this relationship in the following section. For further references see e.g. \cite{engel, yamaguchi, BryantCartan, book}.

Let $\mathrm{G}_2$ denote the connected Lie group with center $\mathbb{Z}_2$ whose Lie algebra  $\mathfrak{g}$ is the split real form of
the smallest of  the  exceptional complex simple Lie algebras.
 $\mathrm{G}_2$ can be defined as the stabilizer subgroup in $\mathrm{GL}(7,\mathbb{R})$ of a generic $3$-form $\Phi\in\Lambda^3(\mathbb{R}^7)^*$. It preserves a non-degenerate bilinear form $h\in \bigodot^2(\mathbb{R}^7)^*$ of signature $(4,3)$.  

The Lie algebra $\mathfrak{g}$ of $\mathrm{G}_2$ has, up to conjugacy, three parabolic subalgebras: the maximal parabolic  algebras $\mathfrak{p}_1,$ $\mathfrak{p}_2$ and the Borel subalgebra $\mathfrak{p}_{1,2}$. Corresponding parabolic subgroups of $\mathrm{G}_2$ can be realized as follows:
$\mathrm{P}_1$ is the  stabilizer of a null line in $\mathbb{R}^7$ with respect to the $\mathrm{G}_2$-invariant bilinear form $h$,  $\mathrm{P}_2$ is the stabilizer of a totally null $2$-plane in $\mathbb{R}^7$ that inserts trivially into $\Phi$, and $\mathrm{P}_{1,2}=\mathrm{P}_1\cap \mathrm{P}_2$.

For a parabolic subgroup $P$ of a simple Lie group $\mathrm{G}$, let $G_+\subset P$ be the unipotent radical and $G_0=P/G_+$ the reductive Levi factor, so that $P=G_0\ltimes G_+$. Denote by $\mathfrak{g}_+$ and $\mathfrak{g}_0=\mathfrak{p}/\mathfrak{g}_+$ the corresponding Lie algebras.
Via the adjoint action,  $P$ preserves a filtration
\begin{align}\label{para_filt}\mathfrak{g}=\mathfrak{g}^{-k}\supset \mathfrak{g}^{-k+1}\supset\cdots\supset\mathfrak{g}^0\supset\mathfrak{g}^1\supset\cdots\supset\mathfrak{g}^k,\end{align}
where $\mathfrak{g}^1=\mathfrak{g}_+$, $\mathfrak{g}^j=[\mathfrak{g}^{j-1},\mathfrak{p}_+]$ for $j\geq 2$, $\mathfrak{g}^{j+1}=(\mathfrak{g}^{-j})^{\perp}$ for $j\leq -1$ (the complement is taken with respect to the Killing form) and, in particular, $\mathfrak{g}^0=\mathfrak{p}$. Any splitting $\mathfrak{g}_0\to\mathfrak{p}$ determines an identification of the filtered Lie algebra $\mathfrak{g}$ with its associated graded Lie algebra
$$\mathrm{gr}(\mathfrak{g})=\mathfrak{g}_{-k}\oplus\cdots\oplus\mathfrak{g}_0\oplus\cdots\oplus\mathfrak{g}_{k}.$$

For complex simple Lie algebras (and their split-real forms) conjugacy classes of parabolic subalgebras  are in on-to-one correspondence with subsets of simple roots (having fixed a Cartan subalgebra $\mathfrak{h}$ and a set of simple roots $\Delta^0$). The correspondence is given as follows: Recall that any root can be uniquely decomposed into a sum of simple roots $\alpha=\sum_{i}a_i \alpha_i$ where all coefficients $a_i$ (if non-zero) are integers of the same sign. For any subset $\Sigma\subset\Delta^0$ one now defines the $\Sigma$-height $\mathrm{ht}_{\Sigma}(\alpha)$ of a root  to be
 $\mathrm{ht}_{\Sigma}(\alpha)=\sum_{i:\alpha_i\in\Sigma}a_i.$ Then  $$\mathfrak{p}=\mathfrak{h}\oplus_{\{\alpha:\mathrm{ht}_{\Sigma}(\alpha)\geq0\}}\mathfrak{g}_{\alpha}$$
is a parabolic subalgebra. In fact, these choices determine a grading: $\mathfrak{g}_0=\mathfrak{h}\oplus_{\{\alpha:\mathrm{ht}_{\Sigma}(\alpha)=0\}}\mathfrak{g}_{\alpha}$
is a Levi subalgebra and the remaining grading components  are given by
$\mathfrak{g}_{i}=\oplus_{\{\alpha:\mathrm{ht}_{\Sigma}(\alpha)=i\}}\mathfrak{g}_{\alpha}.$

In the $\mathrm{G}_2$ case  we have two simple roots $\Delta^0=\{\alpha_1,\alpha_2\}$, and the parabolic subalgebras $\mathfrak{p}_1$, $\mathfrak{p}_2$ and $\mathfrak{p}_{1,2}$ correspond  to the  sets $\Sigma_1=\{\alpha_1\}$, $\Sigma_2=\{\alpha_2\}$ and $\Sigma=\Delta^0$.

In this paper we are particularly interested in the contact grading, corresponding to $\Sigma_2=\{\alpha_2\}$. Here we have $\mathfrak{g}_0\cong\mathfrak{gl}(2,\mathbb{R})$, $\mathfrak{g}_{-}=\mathfrak{g}_{-1}\oplus\mathfrak{g}_{-2}$ and $\mathfrak{g}_+=\mathfrak{g}_1\oplus\mathfrak{g}_2$ are dual with respect to the Killing form and  isomorphic to the $5$-dimensional Heisenberg algebra. Moreover,  the $\mathfrak{g}_0$-representation $\mathfrak{g}_{-1}$ is irreducible; hence $\mathfrak{g}_{-1}\cong \smash{\bigodot^3\mathbb{R}^2}$ as a representation of the semisimple part ${\mathfrak{g}_0}^{ss}\cong\mathfrak{sl}(2,\mathbb{R})$.
\begin{equation}\label{cont_grad}
\begin{tikzpicture}[scale=1,baseline=-5pt]
    \node[black] at (-1.6,2.3)
    {$\mathfrak{p}_2=\mathfrak{g}_0\oplus\mathfrak{g}_1\oplus\mathfrak{g}_2$};
  \draw[black] (-1.4,  -0.2) rectangle (1.4,0.2);
   \draw[black] (-1.85,  0.6) rectangle (1.85,1.85);
   \draw[black] ( 0  , 1.732) -- (0,0);
    \filldraw[black] ( 0  ,  1.732) circle (0.06);
    \draw[black] (0, 0) -- (1.5,  0.866); 
    \filldraw[black] ( 1.5,  0.866) circle (0.06);
    \draw[black] (0, 0) -- (0.5,  0.866); 
    \filldraw[black] ( 0.5,  0.866) circle (0.06);
    \draw[black] (0, 0) -- (-0.5,  0.866); 
    \filldraw[black] ( -0.5,  0.866) circle (0.06);
    \draw[black] (0, 0) -- (-1.5,  0.866); 
    \filldraw[black] ( -1.5,  0.866) circle (0.06);
    \draw[black] (0, 0) -- (0.866,  0); 
    \filldraw[black] (0.866,  0) circle (0.06);
    \draw[black] (0, 0) -- (1.5,  -0.866); 
    \filldraw[black] ( 1.5,  -0.866) circle (0.06);
    \draw[black] ( 0  , -1.732) -- (0,0);
    \filldraw ( 0  ,  -1.732) circle (0.06);
    \draw[black] (0, 0) -- (-1.5,  -0.866); 
    \filldraw ( -1.5,  -0.866) circle (0.06);
    \draw[black] (0, 0) -- (-0.5,  -0.866); 
    \filldraw ( -0.5,  -0.866) circle (0.06);
    \draw[black] (0, 0) -- (0.5,  -0.866); 
    \filldraw ( 0.5,  -0.866) circle (0.06);
    \draw[black] (0, 0) -- (-0.866,  0); 
    \filldraw[black] (-0.866,  0) circle (0.06);
    \filldraw[color=black] (0,0) circle (0.1); 
          \draw[dashed](-  2, 1.2)--(2, 1.2);
   \draw[dashed](-  2, 0.4)--(2, 0.4);
   \draw[dashed](-  2, -1.2)--(2, -1.2);
   \draw[dashed](-  2, -0.4)--(2, -0.4);

\draw (2.5,1.7) node {$\mathfrak{g}_{2}$};
\draw (2.5,-1.7) node {$\mathfrak{g}_{-2\,}$};
\draw (2.5,0.9) node {$\mathfrak{g}_{1}$};
\draw (2.5,-0.9) node {$\mathfrak{g}_{-1\,}$};
\draw (2.5,0) node {$\mathfrak{g}_{0}\ $};

\end{tikzpicture}
 \end{equation}

The model for contact twisted cubic structures  is the
 homogeneous space $\mathrm{G}_2/\mathrm{P}_2$.
\begin{proposition}\label{homog}
The homogeneous space $\mathrm{G}_2/\mathrm{P}_2$ is naturally equipped with a $\mathrm{G}_2$-invariant contact twisted cubic structure.
\end{proposition}
\begin{proof}
The tangent bundle of $\mathrm{G}_2/\mathrm{P}_2$ is the associated bundle
\begin{align}\label{iso}T(\mathrm{G}_2/\mathrm{P}_2)=\mathrm{G}_2\times_{\mathrm{P}_2}(\mathfrak{g}/\mathfrak{p}_2),\end{align} where $\mathfrak{g}$ denotes the Lie algebra of $\mathrm{G}_2$. The identification is induced by the trivialization of the tangent bundle of the Lie group $\mathrm{G}_2$ by left-invariant vector fields.
Using the Maurer-Cartan form $\omega_{G_2}\in\Omega^1(\mathrm{G}_2,\mathfrak{g}),$ $\omega_{G_2}(\xi_g)=T\lambda_{g^{-1}}\xi_g$, it can be written as
$$\xi_{x}\mapsto [g, \omega_{G_2}(\xi_{x})+\mathfrak{p}_2],$$
where $g\in\mathrm{G}_2, x=g\mathrm{P}_2 $ and  $\xi_x\in T_x(\mathrm{G}_2/\mathrm{P}_2) .$ The filtration \eqref{para_filt} induces a $\mathrm{P}_2$-invariant filtration $$\mathcal{C}_o=\mathfrak{g}^{-1}/\mathfrak{p}_2\subset\mathfrak{g}/\mathfrak{p}_2=T_o(\mathrm{G}_2/\mathrm{P}_2)$$ and, via the identification \eqref{iso}, a  subbundle $\mathcal{C}\subset T(\mathrm{G}_2/\mathrm{P}_2)$ of codimension one. The Levi bracket $\mathcal{L}:\Lambda^2\mathcal{C}\to T\M/\mathcal{C}$ corresponds to the Lie bracket on $\mathfrak{g}_{-}=\mathrm{gr}(\mathfrak{g}/\mathfrak{p}_2)$. Since this is the $5$-dimensional Heisenberg Lie algebra, $\mathcal{C}$ is contact. Moreover, since the unipotent radical   acts trivially on $\mathfrak{g}^{-1}/\mathfrak{p}_2$, the $\mathrm{P}_2$ action  factors to a $G_0$ action on $\mathcal{C}_o=\mathfrak{g}^{-1}/\mathfrak{p}_2$.
The latter action is irreducible, and  the orbit through a highest weight line defines a $G_0$-invariant Legendrian twisted cubic $\gamma_o\subset\mathbb{P}(\mathcal{C}_o)$.
\end{proof}
\begin{definition}
A contact twisted cubic structure is called \emph{flat}, or \emph{contact Engel structure},  if and only if it is locally equivalent to the $\mathrm{G}_2$-invariant structure on $\mathrm{G}_2/\mathrm{P}_2$.
\end{definition}

 \begin{remark}\label{remark6}
It follows from the general theory, see \cite{book}, that there is an equivalence of categories between general contact twisted cubic structures and certain regular, normal parabolic geometries.
The Engel structure is the locally unique contact twisted cubic structure with infinitesimal symmetry algebra of maximal dimension, and it is characterized, up to local equivalence, by the vanishing of the harmonic part  of the curvature of the canonically associated Cartan connection. The infinitesimal automorphisms of a general contact twisted cubic structure form a Lie algebra of dimension $\leq 14$. In fact, if the structure is non-flat, it is known that the symmetry algebra is of dimension $\leq 7,$ see \cite{gap}.
  \end{remark}

  \begin{proposition} \label{G2modG12} Let $\gamma\subset\mathbb{P}(\mathcal{C})$ be the $\mathrm{G}_2$-invariant contact twisted cubic structure on $\mathrm{G}_2/\mathrm{P}_2$. Then
 $$\gamma=\mathrm{G}_2\times_{\mathrm{P}_2} {\mathrm{P}_2/\mathrm{P}_{1,2}}=\mathrm{G}_2/\mathrm{P}_{1,2} .$$
 \end{proposition}
 \begin{proof}
The left $\mathrm{G}_2$ action on $\mathrm{G}_2/\mathrm{P}_2$ lifts to a $\mathrm{G}_2$ action on $\gamma$.
  Consider the fibre $\gamma_o\subset\mathbb{P}(\mathfrak{g}^{-1}/\mathfrak{p}_2)$ over the origin $o=e\mathrm{P}_2$. Then the $\mathrm{G}_2$ action on $\gamma$ restricts to a $\mathrm{P}_2$ action on $\gamma_o$, which
factors to an action of  $G_0=\mathrm{GL}(2,\mathbb{R})$, since the unipotent radical acts trivially. The latter action is transitive on $\gamma_o$ and the stabilizer of a point in $\gamma_o$ (which is a highest weight line in $\mathfrak{g}^{-1}/\mathfrak{p}_2$) is the Borel subgroup $B\subset \mathrm{GL}(2,\mathbb{R})$ as in \eqref{borel}.  Then the stabilizer in $\mathrm{P}_2$ of the point is $B\ltimes\mathrm{exp}(\mathfrak{g}_+)$, which is the Borel subgroup $\mathrm{P}_{1,2}\subset\mathrm{G}_2$, and so
$$\gamma=\mathrm{G}_2\times_{\mathrm{P}_2} {\gamma_o}=G_2\times_{\mathrm{P}_2} {\mathrm{P}_2/\mathrm{P}_{1,2}}=G_2/\mathrm{P}_{1,2} .$$
\end{proof}

\begin{definition}\label{contEngel}
A \emph{marked contact Engel structure} is a marked contact twisted cubic structure whose underlying contact twisted cubic structure is flat.
\end{definition}

\begin{remark}
Also in the general, non-flat case, we can  identify $\gamma$ with the so-called \emph{correspondence space} $\mathcal{G}\times_{\mathrm{P}_2} {\mathrm{P}_2/\mathrm{P}_{1,2}}=\mathcal{G}/\mathrm{P}_{1,2}$ by means of the associated canonical Cartan connection $\omega\in\Omega^1(\mathcal{G},\mathfrak{g})$.
\end{remark}

\section{Local invariants and homogeneous models of marked contact Engel structures via Cartan's equivalence method}\label{CartanEquiv}
In this section we apply Cartan's method of equivalence (see e.g. \cite{Olver} for an introduction to the general method) to the local equivalence problem of marked contact Engel structures. We derive a set of local differential invariants of marked contact Engel structures. These allow us, in particular, to characterize the maximal and submaximal symmetric models. We further  obtain a tree of locally non-equivalent branches of marked contact Engel structures, and we derive the structure equations for the maximally symmetric homogeneous structures in (almost all) branches. In particular, this yields a complete classification of all homogeneous marked contact Engel structures with the symmetry algebra of dimension $\geq 6$ up to local equivalence.

\subsection{Adapted coframes}\label{adaptedcoframes}

In order   to apply Cartan's method to the equivalence problem of marked contact Engel structures, we shall recast  the problem in terms of adapted coframes. A (marked) contact twisted cubic structure on a manifold $\M$ defines a natural coframe  bundle, and  adapted coframes are the sections of these bundles.

\begin{definition}\label{adapted1}
 Let $\gamma\subset\mathbb{P}(\mathcal{C})$ be a contact twisted cubic structure on $\mathcal{U}$.
A (local) coframe $(\omega^0,\omega^1,\omega^2,\omega^3,\omega^4)$ on $\mathcal{U}$ is adapted to  the contact twisted cubic structure $\gamma\subset\mathbb{P}(\mathcal{C})$ if  in terms of this coframe
$$\mathcal{C}=\mathrm{ker}(\omega^0)$$
and $\gamma\subset \mathbb{P}(\mathcal{C})$ is the projectivization of the set of all tangent vectors contained in $\mathcal{C}$ that are simultaneously null for the following three symmetric tensor fields
 \begin{equation}\label{3metrics}
 \begin{aligned}
 g_1=\omega^1\omega^3- (\omega^2)^2,\quad
  g_2= \omega^2\omega^4- (\omega^3)^2,\quad
  g_3= \omega^2\omega^3-\omega^1\omega^4.
 \end{aligned}
 \end{equation}
\end{definition}

\begin{proposition}
Two coframes $(\hat{\omega}^0,\hat{\omega}^1,\hat{\omega}^2,\hat{\omega}^3,\hat{\omega}^4)$ and $(\omega^0,\omega^1,\omega^2,\omega^3,\omega^4)$ on $\mathcal{U}$ are adapted to the same
contact twisted cubic structure if and only if
\begin{align}\label{coframe2}
\begin{pmatrix}
\hat{\omega}^0\\ \hat{\omega}^1\\ \hat{\omega}^2\\ \hat{\omega}^3\\ \hat{\omega}^4\\
\end{pmatrix}=
\begin{pmatrix}
s_0& 0&0&0&0\\
s_1&{s_5}^3& 3{s_5}^2s_6&3 s_5{s_6}^2&{s_6}^3\\
s_2&{s_5}^2s_7& {s_5}(s_5 s_8+2 s_6s_7)& s_6(2s_5s_8+s_6s_7)&{s_6}^2s_8\\
s_3& s_5 {s_7}^2& s_7(2 s_5  s_8+s_6s_7) & s_8(s_5{s_8}+2s_6s_7)&s_6{s_8}^2\\
s_4& {s_7}^3& 3 {s_7}^2s_8&3s_7{s_8}^2&{s_8}^3
\end{pmatrix}
               \begin{pmatrix}\omega^0\\ \omega^1\\ \omega^2\\ \omega^3\\ \omega^4\\   \end{pmatrix}
\end{align}
where  $s_0, s_1,s_2,s_3,s_4,s_5,s_6,s_7,s_8$ are smooth functions on $\mathcal{U}$ such that the determinant $s_0(s_6s_7-s_5s_8)^6\neq 0$.

Two contact twisted cubic structures represented by coframes $(\omega^0,\dots, \omega^4)$ on $\mathcal{U}$ and $(\bar{\omega}^0,\dots,\bar{\omega}^4)$ on $\mathcal{V}$ are (locally) equivalent if and only if there exists a (local) diffeomorphism $f:\mathcal{U}\to \mathcal{V}$ such that
$(f^*(\bar{\omega}^0),\dots, f^*(\bar{\omega}^4))$ is related to $(\omega^0,\dots,\omega^4)$ by a transformation matrix of the form as in \eqref{coframe2}.
\end{proposition}

Note that the bottom right $4\times 4$ block in the transformation matrix from \eqref{coframe2} is $\mathrm{GL}(2,\mathbb{R})$ in the $4$-dimensional irreducible representation \eqref{irrepres}.

\begin{definition} Let $\sigma:\mathcal{U}\to\gamma\subset\mathbb{P}(\mathcal{C})$ be a marked contact twisted cubic structure on $\mathcal{U}$.
A (local) coframe $(\omega^0,\omega^1,\omega^2,\omega^3,\omega^4)$ is  adapted to the marked contact twisted cubic structure $\sigma:U\to\gamma\subset\mathbb{P}(\mathcal{C})$ if it is adapted to the underlying contact twisted cubic structure as in Definition \ref{adapted1}
 and moreover the line field $\ell^{\sigma}$  is given by
\begin{align*} \ell^{\sigma} =\mathrm{ker} (\omega^0,\omega^1, \omega^2, \omega^3). \end{align*}
\end{definition}

\begin{proposition}
Two coframes $(\hat{\omega}^0,\hat{\omega}^1,\hat{\omega}^2,\hat{\omega}^3,\hat{\omega}^4)$ and $(\omega^0,\omega^1,\omega^2,\omega^3,\omega^4)$ on $\mathcal{U}$  are adapted to the same marked contact twisted cubic structure if and only if
\begin{align}\label{equiv_punct}
\begin{pmatrix} \hat{\omega}^0\\ \hat{\omega}^1\\ \hat{\omega}^2\\ \hat{\omega}^3\\ \hat{\omega}^4\\ \end{pmatrix}=
\begin{pmatrix}
s_0& 0&0&0&0\\
s_1&{s_5}^3& 0 & 0&0\\
s_2&{s_5}^2s_7& {s_5}s_5 s_8& 0 &0\\
s_3& s_5 {s_7}^2& 2s_7 s_5  s_8 & s_8s_5{s_8}& 0\\
s_4& {s_7}^3& 3 {s_7}^2s_8&3s_7{s_8}^2&{s_8}^3
              \end{pmatrix}
               \begin{pmatrix}  \omega^0\\ \omega^1\\ \omega^2\\ \omega^3\\ \omega^4\\  \end{pmatrix}
\end{align}
where $s_0, s_1,s_2,s_3,s_4,s_5,s_7,s_8$ are smooth functions on $\mathcal{U}$ such that $s_0 s_5s_8\neq 0$.

 Two marked contact twisted cubic structures represented by coframes $(\omega^0,\dots, \omega^4)$ on $\mathcal{U}$ and $(\bar{\omega}^0,\dots,\bar{\omega}^4)$ on $\mathcal{V}$ are (locally) equivalent if and only if there exists a (local) diffeomorphism $f:\mathcal{U}\to \mathcal{V}$ such that
$(f^*(\bar{\omega}^0),\dots, f^*(\bar{\omega}^4))$ is related to $(\omega^0,\dots,\omega^4)$ by a transformation matrix of the form as in \eqref{equiv_punct}.
\end{proposition}

 Here,  the bottom right $4\times 4$ block in the transformation matrix \eqref{equiv_punct} is  the Borel subgroup $B\subset\mathrm{GL}(2,\mathbb{R})$, defined in \eqref{borel}, in the irreducible representation as in \eqref{irrepres}.

\begin{remark}
Alternatively, we may describe a marked contact twisted cubic structure by considering the intersection of the null cones  of only the two metrics $g_1$ and $g_3$ from \eqref{3metrics}.
\end{remark}

 \subsection{Structure equations for marked contact Engel structures}
 From now on we shall concentrate on    marked contact Engel structures as defined in Definition \ref{contEngel}.

Consider the Maurer-Cartan equations of $\mathrm{G}_2$ as displayed in the Appendix in  \eqref{MaurerCartan}, written with respect to the basis $(E_0,E_1,\dots,E_{13})$ as in \eqref{basis_g2} of  $\mathfrak{g}$, which is adapted to the contact grading $$\mathfrak{g}=\mathfrak{g}_{-}\oplus\mathfrak{g}_0\oplus\mathfrak{g}_{+}=\mathfrak{g}_{-2}\oplus\mathfrak{g}_{-1}\oplus\mathfrak{g}_0\oplus\mathfrak{g}_1\oplus\mathfrak{g}_2.$$   Then the kernel of the nine left-invariant forms $\theta^5, \theta^6,\dots,\theta^{13}$ from \eqref{MaurerCartan} defines an integrable distribution. The leaves of the corresponding  foliation  correspond to certain sections of $\mathrm{G}_2\to\mathrm{G}_2/P_{2}$. The pullbacks of the forms $\theta^5, \theta^6,\dots,\theta^{13}$ with respect to any of these sections vanish on $\mathrm{G}_2/\mathrm{P}_2$,  and the pullbacks
 of the remaining forms $\theta^0,\theta^1,\theta^2,\theta^3,\theta^4$ define an adapted coframe $(\alpha^0, \alpha^1, \alpha^2, \alpha^3, \alpha^4)$ for the contact Engel structure on $\mathrm{G}_2/\mathrm{P}_2$, which
satisfies the system
  \begin{equation}\label{MCg-}\der\alpha^0=\alpha^1\dz\alpha^4-3\alpha^2\dz\alpha^3, \quad \der\alpha^1=0, \quad\der\alpha^2=0, \quad \der\alpha^3=0, \quad \der\alpha^4=0.\end{equation}
  Integrating this system yields  local coordinates $(x^0, x^1, x^2, x^3, x^4)$ such that
\begin{equation}\label{flatcoframe}\alpha^0=\mathrm{d}x^0+x^1 \mathrm{d}x^4- 3 x^2 \mathrm{d}x^3,\quad \alpha^1=\der x^1,\quad \alpha^2=\der x^2,\quad \alpha^3=\der x^3,\quad \alpha^4=\der x^4. \end{equation}
Hence such a coframe $(\alpha^0, \alpha^1,\alpha^2,\alpha^3,\alpha^4)$ is an adapted coframe for the  contact Engel structure.

  \begin{remark}
 Note that \eqref{MCg-} are the Maurer-Cartan equations of $G_{-}=\mathrm{exp}(\mathfrak{g}_{-})$ for the Maurer-Cartan form $\theta_{MC}$ of $G_{-}$.
 Alternatively, the coordinate representation \eqref{flatcoframe} can be obtained from the  parameterisation  $\phi:\mathbb{R}^5\to G_{-} \, \cdot \, o \subset \mathrm{G}_2/P_{2}$ given by
 $$\phi(x^0, x^1, x^2, x^3, x^4)=\mathrm{exp}(x^0E_0)\mathrm{exp}(x^1E_1)\mathrm{exp}(x^2 E_2)\mathrm{exp}(x^3 E_3)\mathrm{exp}(x^4 E_4) o,$$
 with $E_0\in\mathfrak{g}_{-2}$ and  $E_1, E_2, E_3, E_4\in\mathfrak{g}_{-1}$ and the well-known formula $\theta_{MC}=\phi^{-1}d\phi=\alpha^iE_i$.
 \end{remark}


Now denote by $(X_0, X_1, X_2, X_3, X_4)$ the frame dual to the coframe $(\alpha^0, \alpha^1, \alpha^2, \alpha^3, \alpha^4)$ as in  \eqref{flatcoframe}. We may assume that the section $\sigma:\mathcal{U}\to\gamma$ defining a general \emph{marked} contact Engel structure on $\mathrm{G}_2/\mathrm{P}_2$ is  of the form \begin{equation}\label{section}\sigma=[-t^3 X_1+t^2  X_2-t X_3+ X_4],\end{equation} where $t=t(x^0, x^1, x^2, x^3, x^4)$ is a smooth function on $\mathcal{U}$. In this sense, the choice of a function $t$ determines a marked contact Engel structure, and up to local equivalence, all marked contact Engel structures can be obtained in this way. Note however, that different $t$'s can correspond to the same structure (up to local equivalence).
The osculating filtration from Proposition \ref{propfilt} of the marked Engel structure is of the form
\begin{equation}\label{oscfilt}
\ell^{\sigma}=\mathrm{Span}(\xi_4)\subset\mathcal{D}^{\sigma}=\mathrm{Span}(\xi_4,\xi_3)\subset\mathcal{H}^{\sigma}=\mathrm{Span}(\xi_4,\xi_3,\xi_2)\subset\mathcal{C}=\mathrm{Span}(\xi_4,\xi_3,\xi_2,\xi_1),
\end{equation}
where
\begin{equation}\label{frame1}
\begin{aligned}
\xi_4&:=-t^3 X_1+t^2  X_2-t X_3+ X_4=\, -(x^1+3tx^2)\partial_{x^0}-t^3\partial_{x^1}+t^2\partial_{x^2}-t\partial_{x^3}+\partial_{x^4}\\
\xi_3&:= 3 t^2 X_1-2 t X_2+ X_3=\, 3x^2\partial_{x^0}+3t^2\partial_{x^1}-2t\partial_{x^2}+\partial_{x^3}\\
\xi_2&:=-3 t X_1+ X_2=\, -3 t \partial_{x^1}+\partial_{x^2}\\
\xi_1&:=X_1=\,\partial_{x^0}\\
\xi_0&:= X_0=\, \partial_{x^1}\\
\end{aligned}
\end{equation}
Passing to the coframe $(\omega^0, \omega^1, \omega^2, \omega^3, \omega^4)$ dual to the frame $(\xi_0, \xi_1, \xi_2, \xi_3, \xi_4)$ yields the following.

\begin{lemma}\label{lemma1marked}
 The most general marked contact Engel structure can be locally represented in terms of the following adapted coframe
 \begin{align}\label{coframet}
\begin{pmatrix} \omega^0\\ \omega^1\\ \omega^2\\ \omega^3\\ \omega^4\\  \end{pmatrix}=
               \begin{pmatrix}\begin{aligned}
 &\der x^0+x^1\der x^4- 3 x^2\der x^3\\
  &\der x^1+3 t\der x^2+ 3 t^2 \der x^3+t^3\der x^4\\
  &\der x^2 +2 t \der x^3 +t^2 \der x^4\\
  &\der x^3+ t \der x^4\\
  &\der x^4\\
  \end{aligned}
 \end{pmatrix},
\end{align}
where $t=t(x^0, x^1, x^2, x^3, x^4)\in C^{\infty}(\mathcal{U})$. The filtration \eqref{oscfilt} associated to a marked contact Engel structure is given in terms of this coframe as
\begin{equation}\label{eq.filtration.wo.w1.w2.w3}
\ell^{\sigma}=\mathrm{ker}(\omega^0,\omega^1,\omega^2,\omega^3)\subset\mathcal{D}^{\sigma}=\mathrm{ker}(\omega^0,\omega^1,\omega^2)\subset\mathcal{H}^{\sigma}=\mathrm{ker}(\omega^0,\omega^1)\subset\mathcal{C}=\mathrm{ker}(\omega^0).
\end{equation}

 \end{lemma}


Our problem is to  produce differential invariants that allow us to distinguish non-equivalent classes of marked contact Engel structures.
In particular, all of these invariants should vanish for the simplest marked contact Engel structure, the one corresponding to $t=0$, which we call \emph{flat}.

\begin{definition}\label{flatEngel}
A  marked contact Engel structure is called \emph{flat} if it can be locally represented in terms of an adapted coframe $(\alpha^0,\alpha^1,\alpha^2,\alpha^3,\alpha^4)$ as in \eqref{flatcoframe}.
\end{definition}

Using Lemma \ref{lemma1marked}, we next observe the following.

  \begin{lemma}\label{diffcof}
   Any marked contact Engel structure admits an adapted coframe $(\omega^0,\omega^1,\omega^2,\omega^3,\omega^4)$ satisfying
   \begin{equation}\label{differentiatedcoframe2}
  \begin{aligned}
  &\der\omega^0=\omega^1\wedge\omega^4-3\omega^2\wedge\omega^3\\
  &\der\omega^1=\tfrac{3}{4}(b^2-4ac+M-P)\omega^0\wedge\omega^2+3c\omega^1\wedge\omega^2-3 a  \omega^2\wedge\omega^3 +3 J \omega^2\wedge\omega^4\\
  &\der\omega^2= \tfrac{1}{2}(b^2-4ac+M-P) \omega^0\wedge\omega^3+2 c\omega^1\wedge\omega^3-2 b\omega^2\wedge\omega^3+2 J \omega^3\wedge\omega^4\\
 & \der\omega^3=\tfrac{1}{4}(b^2-4ac+M-P)\omega^0\wedge\omega^4+c\omega^1\wedge\omega^4-b\omega^2\wedge\omega^4+a\omega^3\wedge\omega^4\\
  &\der\omega^4=0\\
  \end{aligned}
  \end{equation}
  for  functions $a, b, c, J, M, P$.
  \end{lemma}

  \begin{proof}
  We work in the representation from Lemma \ref{lemma1marked}.
  Differentiating the coframe \eqref{coframet} gives
   \begin{equation}\label{differentiatedcoframe1}
  \begin{aligned}
  \der\omega^0=&\ \omega^1\wedge\omega^4-3\omega^2\wedge\omega^3\\
    \der\omega^1=&\ 3 \der t \wedge \omega^2\\
  \der\omega^2= &\ 2  \der t \wedge \omega^3\\
  \der\omega^3=&\ \der t\wedge\omega^4\\
  \der\omega^4=&\ 0,\\
  \end{aligned}
  \end{equation}
  Then one expands $\der t$ in terms of the coframe  $(\omega^0, \omega^1, \omega^2, \omega^3, \omega^4)$ and then there is a unique solution for $a, b, c, J$ and $ M-P$
   in terms of the function $t$ and its derivatives.

  \end{proof}

  \begin{remark}
 Indeed,  any marked contact twisted cubic structure admitting an adapted coframe as in Lemma \ref{diffcof} is flat as a contact twisted cubic structure, i.e., it is a marked contact Engel structure.
  \end{remark}

  Applying the exterior derivative  on both sides of \eqref{differentiatedcoframe2} we get information about the exterior derivatives of the functions $a, b, c$ and $J$. Explicitly, we obtain the following lemmas. Recall that  a subscript $\omega^i$ denotes the $i$th frame derivative as in Section \ref{notation}.
  \begin{lemma}\label{lemma3marked}
  The functions $a, b, c$ and $J$ from Lemma \ref{diffcof} satisfy
   \begin{equation}\label{differentiatedJabc2}
  \begin{aligned}
  &\der J=J_{\omega^0}\omega^0+J_{\omega^1}\omega^1+J_{\omega^2}\omega^2+J_{\omega^3}\omega^3+J_{\omega^4}\omega^4\\
  &\der a=a_{\omega^0}\omega^0+a_{\omega^1}\omega^1+\tfrac{1}{4}(-3b^2+M+3P)\omega^2+L\omega^3+(a^2-2bJ-J_{\omega^3})\omega^4\\
  &\der b=\tfrac{1}{4}(-4 a_{\omega^1} b + 6 b^2 c - 8 a c^2 + 4 c M - M_{\omega^2} + P_{\omega^2} + 2 b Q - 4 a R)\omega^0+(2c^2+R)\omega^1+(2 a_{\omega^1}-3bc-Q)\omega^2\\&\quad+\tfrac{1}{2}(-b^2+M-3P)\omega^3+(ab-3cJ+J_{\omega^2})\omega^4\\
 & \der c =c_{\omega^0}\omega^0+S\omega^1+(c^2-R)\omega^2+(a_{\omega^1}-2bc)\omega^3+\tfrac{1}{4}(b^2-4J_{\omega^1}+M-P)\omega^4 ,\\
    \end{aligned}
  \end{equation}
   for functions $L, Q, R, S$  on $\M$.
\end{lemma}

 \begin{lemma}
 The functions $a, b, c, J, L, M, P, Q, R, S$ are uniquely determined by \eqref{differentiatedcoframe1} and \eqref{differentiatedJabc2}. Explicitly,
    \begin{align*}
   &a= t_{\omega^3}, \quad  b=-t_{\omega^2},\quad  c=t_{\omega^1},\quad J=-t_{\omega^4},\quad L=t_{\omega^3\omega^3},\quad M=6t_{\omega^0}-2(t_{\omega^2})^2+6 t_{\omega^3} t_{\omega^1}+ t_{\omega^2\omega^3},\\& P=2t_{\omega^0}-(t_{\omega^2})^2+2 t_{\omega^3} t_{\omega^1}+t_{\omega^2\omega^3},\quad
    Q=2 t_{\omega^3\omega^1} +t_{\omega^2\omega^2}+3t_{\omega^2} t_{\omega^1},\quad R=-t_{\omega^2\omega^1}-2(t_{\omega^1})^2,\quad S=t_{\omega^1\omega^1}.
   \end{align*}
   \end{lemma}

\subsection{The main invariants and a characterization of the flat model}\label{themain}
In this section we shall formulate our first main theorem, which in particular justifies the importance of the functions $J, L, M, P, Q, R, S$.  Note that the flat marked contact Engel structure corresponding to $t=0$ in the parametrization from Lemma \ref{lemma1marked}  satisfies $J=L=M=P=Q=R=S=0$.

Before stating the theorem, we introduce the following notation for  the Maurer-Cartan equations, given in the Appendix by formula  \eqref{MaurerCartan_Q}, of  the $9$-dimensional  parabolic subgroup $\mathrm{P}_1\subset\mathrm{G}_2$:
 \begin{equation}\label{MaurerCartrew} \begin{aligned}
&e^0= \der\theta^0 -(-6 \theta^0\wedge\theta^5 + \theta^1\wedge \theta^4 - 3\theta^2\wedge\theta^3)=0\\
&e^1= \der\theta^1 -(- 3\theta^1\wedge\theta^5 - 3\theta^1\wedge\theta^8)=0 \\
&e^2= \der\theta^2 -(\theta^1\wedge\theta^6 - 3\theta^2\wedge\theta^5 - \theta^2\wedge\theta^8)=0 \\
&e^3= \der\theta^3 -(2\theta^2\wedge\theta^6 -3\theta^3\wedge\theta^5 + \theta^3\wedge\theta^8)=0 \\
&e^4= \der\theta^4-(6 \theta^0\wedge\theta^{12}+3\theta^3\wedge\theta^6-3\theta^4\wedge\theta^5+3\theta^4\wedge\theta^8)=0\\
&e^5=\der\theta^5-(-\theta^1\wedge\theta^{12})=0\\
&e^6=\der\theta^6-(6\theta^2\wedge\theta^{12}+2\theta^6\wedge\theta^8)=0\\
&e^8=\der\theta^8-(-3\theta^1\wedge\theta^{12})=0\\
&e^{12}=\der\theta^{12}-( -3\theta^5\wedge\theta^{12}- 3\theta^8\wedge\theta^{12})=0\,.
 \end{aligned}\end{equation}

\begin{remark}
Anticipating the material that will be explained in Section \ref{sec_tanaka}, we advice a reader familiar with Tanaka theory  to look at Proposition \ref{p_1}  for the reason why we expect the parabolic subalgebra $\mathfrak{p}_1$ to be the infinitesimal symmetry algebra of the flat marked contact Engel structure.
\end{remark}

We call the group  $\mathbf{S}\cong B\ltimes\mathbb{R}^5$,
\begin{equation}
\label{matS}
\begin{aligned}\mathbf{S}=\left\{(\mathbf{S}^{\mu}{}_{\nu})=\begin{pmatrix}
s_0& 0&0&0&0\\
s_1&{s_5}^3& 0 & 0&0\\
s_2&{s_5}^2s_7& {s_5}^2 s_8& 0 &0\\
s_3& s_5 {s_7}^2& 2 s_7 s_5  s_8 & s_5{s_8}^2& 0\\
s_4& {s_7}^3& 3 {s_7}^2s_8&3s_7{s_8}^2&{s_8}^3
              \end{pmatrix} :\, \mathrm{det}(\mathbf{S}^{\mu}{}_{\nu})=s_0 {s_5}^6{s_8}^6\neq 0\right\}
\end{aligned}
\end{equation}
 the \emph{structure group} of the equivalence problem for marked contact twisted cubic structures.

\begin{theorem}\label{main}
Given the most general marked contact Engel structure on $\mathcal{U}$, consider an adapted coframe $\omega=(\omega^0,\omega^1,\omega^2,\omega^3,\omega^4)$ that satisfies  structure equations \eqref{differentiatedcoframe2}, and let $J, L, M, P, Q, R, S$ be the functions defined via \eqref{differentiatedcoframe2} and \eqref{differentiatedJabc2}.
 \begin{enumerate}
 \item
Let $\hat{\omega}=(\hat{\omega}^0,\hat{\omega}^1,\hat{\omega}^2,\hat{\omega}^3,\hat{\omega}^4)$ be another  coframe related to $\omega$  via $\hat{\omega}=A\cdot\phi^*(\omega)$, with $\phi:\mathcal{U}\to\mathcal{U}$ a diffeomorphism  and $A:\mathcal{U}\to \mathbf{S}$ a function with values in the structure group  \eqref{matS}.  Further suppose that $\hat{\omega}$ satisfies the structure equations \eqref{differentiatedcoframe2} for some functions $\hat{a}, \hat{b}, \hat{c}, \hat{J}, \hat{M}, \hat{P}$, and let $\hat{Q}, \hat{R}, \hat{S}$ be the derived functions as in \eqref{differentiatedJabc2}. Then
\begin{enumerate}
\item $J=0$ iff $\hat{J}=0$
\item $J=L=0$ iff $\hat{J}=\hat{L}=0$
\item $J=L=M=0$ iff $\hat{J}=\hat{L}=\hat{M}=0$
\item $J=L=M=P=0$ iff $\hat{J}=\hat{L}=\hat{M}=\hat{P}=0$
\item $J=L=M=P=Q=0$ iff $\hat{J}=\hat{L}=\hat{M}=\hat{P}=\hat{Q}=0$
\item $J=L=M=P=Q=R=0$ iff $\hat{J}=\hat{L}=\hat{M}=\hat{P}=\hat{Q}=\hat{R}=0$
\item $J=L=M=P=Q=R=S=0$ iff $\hat{J}=\hat{L}=\hat{M}=\hat{P}=\hat{Q}=\hat{R}=\hat{S}=0$\label{flatness_co}
\end{enumerate}
\item A marked contact Engel structure is flat
if and only if
 \begin{equation}\label{flatness_co}  J=L=M=P=Q=R=S=0 \end{equation}
holds. In this case the structure has a $9$-dimensional algebra of infinitesimal symmetries isomorphic to the parabolic subalgebra $\mathfrak{p}_1$.
\end{enumerate}
\end{theorem}

\begin{remark}\label{rem_cond}
Part 1. of the Theorem says that each of the below itemized differential conditions
\begin{enumerate}
\item $J=0$
\item $J=L=0$
\item $J=L=M=0$
\item $J=L=M=P=0$
\item $J=L=M=P=Q=0$
\item $J=L=M=P=Q=R=0$
\item $J=L=M=P=Q=R=S=0$
\end{enumerate}
is an invariant  condition on the marked contact Engel structure defined by the equivalence class $[\omega]$. Note however  that e.g. $a=0$, or  $L=0$ alone, is \emph{not} an invariant condition.
\end{remark}

\begin{proof} of the Theorem \ref{main}.

 We choose an adapted coframe  $(\omega^0,\omega^1,\omega^2,\omega^3,\omega^4)$ that satisfies \eqref{differentiatedcoframe2}. This  determines a trivialization of  the bundle of all adapted coframes, which may thus be identified with   $\pi:\mathcal{U}\times \mathbf{S}\to \mathcal{U}$.
%
We  can now lift  $(\omega^0,\omega^1,\omega^2,\omega^3,\omega^4)$ to the $5$ well-defined (tautological) $1$-forms
  \begin{align}\label{thetas}
  \theta^{\mu}=\mathbf{S}^{\mu}{}_{\nu}\omega^{\nu},\quad \mu=0,1,2,3,4,
  \end{align}
on $\mathcal{U}\times \mathbf{S}$.
Writing equations \eqref{differentiatedcoframe2} symbolically as
\begin{equation}\label{diffcoframesymb}\der\omega^{\mu}=-\tfrac{1}{2}F^{\mu}{}_{\nu\rho}\omega^{\nu}\wedge\omega^{\rho},\end{equation}
we express  the differentials  $\der\theta^0,...,\der\theta^4$ as
$$\der\theta^{\mu}=\der(\mathbf{S}^{\mu}{}_{\nu}\omega^{\nu})=\der \mathbf{S}^{\mu}{}_{\nu}\wedge\omega^{\nu}+ \mathbf{S}^{\mu}{}_{\nu}\der\omega^{\nu}=\der \mathbf{S}^{\mu}{}_{\rho}(\mathbf{S}^{-1})^{\rho}{}_{\sigma}\wedge\theta^{\sigma}-\tfrac{1}{2}\mathbf{S}^{\mu}{}_{\nu} F^{\nu}{}_{\rho\sigma}(\mathbf{S}^{-1})^{\rho}{}_{\alpha}(\mathbf{S}^{-1})^{\sigma}{}_{\beta}\theta^{\alpha}\wedge\theta^{\beta}.$$
For computational reasons we set $$\delta=-s_5 s_8.$$

Now we will solve equations \eqref{MaurerCartrew}. The unknowns in these equations are the group parameters $s_0, s_1, s_2,$ $ s_3,s_4,s_5, s_7,$ $ \delta$ and the four $1$-forms $\theta^5$, $\theta^6$, $\theta^8$ and $\theta^{12}$. What is given is the coframe $\omega$ and the derived functions $a, b, c, J,$ etc, as defined in \eqref{differentiatedcoframe2} and \eqref{differentiatedJabc2}.
Therefore, if we say that we solve equations $$e^0=0, \,e^1=0,\, ...,\,e^{12}=0,$$ we mean that we are searching  for $s_0, s_1, s_2,$ $ s_3,s_4,s_5, s_7,$ $ \delta$ and $\theta^5$, $\theta^6$, $\theta^8$ and $\theta^{12}$ such that the equations are satisfied.


We start by solving  equation $e^0=0$.
Computing
\begin{align*}
\der\theta^0=
\tfrac{1}{s_0} \der s_0\wedge \theta^0 - \tfrac{s_4}{\delta^3}  \theta^0\wedge\theta^1 +   \tfrac{3 s_3}{\delta^3} \theta^0\wedge\theta^2- \tfrac{3 s_2}{\delta^3}\theta^0\wedge\theta^3+\tfrac{s_1}{\delta^3}\theta^0\wedge\theta^4
- \tfrac{s_0}{\delta^3}\theta^1\wedge\theta^4 +\tfrac{3 s_0}{\delta^3}\theta^2\wedge\theta^3
\end{align*}
and inserting it into $e^0\wedge\theta^0=0$ gives
\begin{align*}
 (-1- \tfrac{s_0}{\delta^3})\theta^1\wedge\theta^4\wedge\theta^0 +(3+\tfrac{3 s_0}{\delta^3})\theta^2\wedge\theta^3\wedge\theta^0=0,
\end{align*}
whose unique solution is \begin{equation}\label{s0} {s_0}=-{\delta^3}.\end{equation}
Having established this,   the most general solution of $e^0=0$ for $\theta^5$ is
\begin{equation}\label{deftheta5}\theta^5:=\tfrac{1}{2\delta} \der \delta + \tfrac{s_4}{6\delta^3}\theta^1 -  \tfrac{ s_3}{2\delta^3}\theta^2+ \tfrac{s_2}{2\delta^3}\theta^3-\tfrac{s_1}{6\delta^3}\theta^4 - \tfrac{1}{6}u_0\,\theta^0.\end{equation}
Note that we had to introduce a new variable $u_0$, since adding to any particular solution for $\theta^5$ a functional multiple of $\theta^0$ is a solution as well. At this point the equation $e^0=0$ is satisfied.

We next consider the equation $e^1\wedge\theta^0\wedge\theta^1=0$, which reads
\begin{equation}\label{theta1}\tfrac{3(s_1 \delta + a  {s_5}^3\delta -3 J {s_5}^4 s_7)}{{\delta}^4}\theta^0\wedge\theta^1\wedge\theta^2\wedge\theta^3+\tfrac{3 J {s_5}^5}{{\delta}^4}\theta^0\wedge\theta^1\wedge\theta^2\wedge\theta^4=0.\end{equation}
Since ${s_5}$ cannot be zero, the vanishing of the coefficient at the $\theta^0\wedge\theta^1\wedge\theta^2\wedge\theta^4$-term in \eqref{theta1} is equivalent to $J=0$.
In other words, we have shown that under the most general transformation that maps one adapted coframe $\omega$ to another adapted coframe $\hat{\omega}$,  the coefficient $F^1_{24}$ in the structure equations \eqref{diffcoframesymb} transforms as
$$\hat{F}^1{}_{24}=\tfrac{3{s_5}^5}{{\delta}^4}F^1{}_{24}.$$
 This shows that it defines a density invariant (or, \emph{relative invariant}) of the marked contact twisted cubic structure. In particular,  its vanishing or not is an invariant property of the structure. For those coframes that satisfy the structure equations  \eqref{differentiatedcoframe2}, the coefficient $\hat{F}^1{}_{24}$ is proportional to $ J$. Moreover, for the (particular) flat structure corresponding to $t=0$ we have $J=0$. This further shows that vanishing of this density invariant that we discovered is a necessary condition for flatness.

From now on we assume $$J=0$$
(which means that also the consequences  $J_{\omega^0}=J_{\omega^1}=J_{\omega^2}=J_{\omega^3}=J_{\omega^4}=0$ hold).
We return to equation \eqref{theta1}. We can now solve it by setting
\begin{equation}\label{s1}s_1=- a {s_5}^3.\end{equation} Then we look at equation $e^1\wedge\theta^1=0$, which reads
$$
\tfrac{-{s_5}^2}{\delta^5}(M\delta +2 L s_5 s_7)\theta^0\wedge\theta^1\wedge\theta^2+\tfrac{{s_5}^4}{\delta^5}L\theta^0\wedge\theta^1\wedge\theta^3=0.
$$
The same argument as above applied to the second term in this equation shows that $L$ must be zero for $e^1\wedge\theta^1=0$ to admit a solution.
We  also infer from this that the simultaneous vanishing of $J$ and $L$ is an invariant condition on  marked contact Engel structures, and that $J=L=0$ is another necessary condition for a structure to be flat.
We now  assume that
$$
J=L=0.
$$
With this assumption $e^1\wedge\theta^1=0$ reads
$$
\tfrac{-{s_5}^2}{\delta^4}M\theta^0\wedge\theta^1\wedge\theta^2=0.
$$
As before, we may now conclude that the simultaneous  vanishing of $J$, $L$ and $M$ is an invariant property and necessary for flatness. We will from now on  assume that
$$
J=L=M=0
$$
holds.
Now the general solution for $e^1=0$ is
\begin{equation}\begin{aligned}\label{deftheta8}\theta^8&=-\tfrac{1}{2\delta}\der\delta+\tfrac{1}{s_5}\der{s_5}-\tfrac{(-6 c s_2\delta^2 + 2 a_1 s_5 \delta^3 + 2 a s_4 {s_5}^4 - 6 a c {s_5}^2{s_7}\delta^2- 6 a {s_3}{s_5}^3 s_7 +6 a s_2 {s_5}^2{s_7}^2 +2 a^2 {s_5}^4 {s_7}^3 - s_5 u_0 {\delta}^6}{6 {s_5} {\delta}^6}\theta^0\\
 &+\tfrac{2c\delta^2+s_3s_5-2a{s_5}^2{s_7}^2}{2{s_5}\delta^3}\theta^2-\tfrac{s_2-2a{s_5}^2s_7}{2\delta^3}\theta^3-\tfrac{a{s_5}^3}{2\delta^3}\theta^4 - \tfrac{1}{3}u_1\,\theta^1,\end{aligned}\end{equation} where we have introduced a new variable $u_1$. In this way, $e^1=0$ is solved.

We next attempt to solve  $e^2=0$. We start with $e^2\wedge\theta^0\wedge\theta^1=0,$ which reads
$$\tfrac{2(2 s_2-b s_5 \delta +2 a {s_5}^2 s_7)}{\delta^3}\theta^0\wedge\theta^1\wedge\theta^2\wedge\theta^3=0.$$
Its unique solution is given by
\begin{equation}\label{s2}s_2=\tfrac{1}{2}b s_5 \delta- a {s_5}^2s_7.\end{equation}
Computing $e^2\wedge\theta^1=0$ and looking at the coefficient at the $\theta^0\wedge\theta^1\wedge\theta^3$ term, we see that in order to be able to solve the equation,  $P$ has to be zero. We also conclude that
$$
J=L=M=P=0
$$
is an invariant property,  which we from now on assume to hold. Then the unique solution of $e^2\wedge\theta^1=0$ is
\begin{equation}\label{defu0}u_0=\tfrac{(4a_1-6bc-3Q)\delta^3+3b^2s_5s_7\delta^2-3(s_3 +2a{s_7}^2s_5)b{s_5}\delta+2a{s_5}^2(-s_4s_5+3s_3s_7+2as_5{s_7}^3)}{2{\delta}^6}.\end{equation}
Now the most general $1$-form $\theta^6$ such that $e^2=0$ holds is
\begin{equation}\label{deftheta6}\begin{aligned}
&\theta^6= \tfrac{s_7}{s_5\delta}\der\delta - \tfrac{s_7}{{s_5}^2}\der {s_5} -\tfrac{1}{s_5}\der {s_7}\\
&-\tfrac{2(2 c^2+R)\delta^4-2(4 b c +Q)s_5 s_7\delta^3 +(8cs_3  +5 b^2 {s_5} {s_7}^2+8 a c {s_5}{s_7}^2)s_5 \delta^2 +2(s_4{s_5}-4 s_3s_7-6a  {s_5}{s_7}^3)b{s_5}^2\delta
-4(s_4 {s_5}-3s_3 {s_7}-2 a{s_5}{s_7}^3)a{s_5}^3s_7}{4{s_5}^2{\delta}^6}\theta^0\\
&+\tfrac{2{s_5}^2u_1\delta^3+6cs_7\delta^2-12b s_5{s_7}^2\delta-3s_4{s_5}^2+18a {s_5}^2{s_7}^3}{6{s_5}^2\delta^3}\theta^2-\tfrac{2c \delta^2-2bs_5s_7\delta+3a{s_5}^2{s_7}^2}{s_5\delta^3}\theta^3
-\tfrac{s_5(b\delta - 2 a s_5s_7)}{2\delta^3}\theta^4 +u_2\theta^1.\end{aligned}\end{equation}

Next we compute $e^3\wedge\theta^0=0$, which can be solved
 by
\begin{equation}\label{s3}s_3=\tfrac{-c\delta^2+b s_5s_7\delta-a{s_5}^2{s_7}^2}{s_5},\end{equation}
\begin{equation}\label{u1}u_1=\tfrac{-3(2 c s_7\delta^2-2b s_5{s_7}^2\delta+s_4{s_5}^2+2a{s_5}^2{s_7}^3)}{2{s_5}^2\delta^3},\end{equation}
\begin{equation}\label{u2}u_2=\tfrac{-{s_7}^2(c\delta^2-bs_5s_7\delta+a{s_5}^2{s_7}^2)}{{s_5}^3\delta^3}.\end{equation}
Equation $e^3=0$ now reads
\begin{equation*}\begin{aligned}
\tfrac{1}{{s_5}^4\delta^3}(-S\delta^2+2 Rs_5 s_7 \delta-Q{s_5}^2{s_7}^2)\theta^0\wedge\theta^1-\tfrac{2}{{s_5}^2\delta^3}(R\delta-Qs_5s_7)\theta^0\wedge\theta^2-\tfrac{1}{\delta^3}Q\theta^0\wedge\theta^3=0.
\end{aligned}\end{equation*}
From here we conclude that
$$J=L=M=P=Q=0$$ is an invariant condition. Assuming that it be satisfied, we see that  in order to  be able to solve equation $e^3=0$, we also have to assume $R$  to be zero. We also see  that
$$J=L=M=P=Q=R=0$$
is an invariant condition.
Assuming that it holds,  we see that also $S$ has to be zero. Assuming that the invariant condition
$$J=L=M=P=Q=R=S=0$$ holds, equation $e^3=0$ is now solved.

Now we consider $e^4=0$.  The most general $1$-form  $\theta^{12}$ solving this equation is
\begin{equation}\label{deftheta12}\begin{aligned}
\theta^{12}&=\tfrac{-(cs_7\delta^2-bs_5{s_7}^2\delta+s_4{s_5}^2+a{s_5}^2{s_7}^3)}{2{s_5}^2\delta^4}\der\delta+\tfrac{1}{6\delta^3}\der{s_4}+\tfrac{cs_7\delta^2-bs_5{s_7}^2\delta+s_4{s_5}^2+a{s_5}^2{s_7}^3}{2{s_5}^3\delta^3}\der{s_5}
+\tfrac{c\delta^2-b s_5s_7\delta+a{s_5}^2{s_7}^2}{2{s_5}^2\delta^3}\der{s_7}\\
&+\tfrac{3c^2{s_7}^2\delta^4-6bcs_5{s_7}^3\delta^3+3(cs_4{s_5}^2s_7+b^2{s_5}^2{s_7}^4+2ac{s_5}^2{s_7}^4)\delta^2-3(bs_4{s_5}^3{s_7}^2+2 a b {s_5}^3{s_7}^5)\delta+{s_4}^2{s_5}^4+3a s_4{s_5}^4{s_7}^3+3a^2{s_5}^4{s_7}^6}{6{s_5}^4\delta^6}\theta^1\\
&+\tfrac{bc{s_7}^2\delta^3+(c s_4s_5-b^2s_5{s_7}^3-2acs_5{s_7}^3)\delta^2+3ab{s_5}^2{s_7}^4\delta-a s_4{s_5}^3{s_7}^2-2a^2{s_5}^3{s_7}^5}{2{s_5}^2\delta^6}\theta^2\\
&+\tfrac{c^2\delta^4-2bcs_5s_7{\delta}^3+(b^2{s_5}^2{s_7}^2+3ac{s_5}^2{s_7}^2)\delta^2-3ab{s_5}^3{s_7}^3\delta+as_4{s_5}^4{s_7}+2a^2{s_5}^4{s_7}^4}{2{s_5}^2\delta^6}\theta^3\\
&+\tfrac{4a_1\delta^3-(3b^2s_5s_7+12acs_5s_7)\delta^2+12ab{s_5}^2{s_7}^2\delta-4as_4{s_5}^3-8a^2{s_5}^3{s_7}^3}{24\delta^6}\theta^4+\tfrac{1}{6}u_3\theta^0,
\end{aligned}\end{equation}
where $u_3$ is a new variable. Next we consider equation $e^5=0$, which we solve for
\begin{equation}\label{u3}\begin{aligned}
u_3&=\tfrac{-8c^3\delta^6+24bc^2s_5s_7\delta^5-(21b^2c{s_5}^2{s_7}^2+36ac^2{s_5}^2{s_7}^2)\delta^4+(6bcs_4{s_5}^3+5b^3{s_5}^3{s_7}^3+60abc{s_5}^3{s_7}^3)\delta^3}{4{s_5}^3\delta^9}\\
&-\tfrac{(3b^2s_4{s_5}^4{s_7}+24acs_4{s_5}^4{s_7}+21ab^2{s_5}^4{s_7}^4+36a^2c{s_5}^4{s_7}^4)\delta^2-(18abs_4{s_5}^5{s_7}^2+24a^2b{s_5}^5{s_7}^5)\delta+4{s_4}^2{s_5}^6+12a^2s_4{s_5}^6{s_7}^3+8a^3{s_5}^6{s_7}^6}{4{s_5}^3\delta^9}.
\end{aligned}\end{equation}
Computing shows that now $e^6=0,$  $e^8=0$ and $e^{12}=0$ are satisfied as well.

Concluding, we proved that the conditions displayed in Remark \ref{rem_cond}
 are  invariant conditions on marked contact Engel structures.
The flat marked contact Engel structure satisfies $J=L=M=P=Q=R=S=0$, so this is evidently a necessary condition  for flatness.

  Moreover, assuming $J=L=M=P=Q=R=S=0$, we  uniquely determined
 \begin{itemize}
 \item a $9$-dimensional sub-bundle $\mathcal{P}$ of the $13$-dimensional bundle $\mathcal{U}\times \mathbf{S}\to \mathcal{U}$ we started out with (parametrized by the coordinates $x^0, x^1, x^2, x^3, x^4$ and the remaining fibre coordinates  $s_4, s_5, \delta, s_7$)
 \item and a well defined coframe $(\theta^0, \theta^1, \theta^2, \theta^3, \theta^4, \theta^5, \theta^6, \theta^8, \theta^{12})$  on $\mathcal{P}$ satisfying the Maurer-Cartan equations \eqref{MaurerCartan} whose first five forms $(\theta^0, \theta^1, \theta^2, \theta^3, \theta^4)$ when pulled back with respect to any section of  $\mathcal{P}\to\mathcal{U}$ are contained in the equivalence class $[(\omega^0,\omega^1,\omega^2,\omega^3,\omega^4)]$.
 \end{itemize} Hence a structure that satisfies these conditions  has a $9$-dimensional algebra of infinitesimal symmetries isomorphic to the parabolic subalgebra $\mathfrak{p}_1$.  Taking a section corresponding to a leaf of the integrable distribution given by the kernel of  $\theta^5, \theta^6, \theta^8, \theta^{12}$, the pullbacks of  $\theta^0, \theta^1, \theta^2, \theta^3, \theta^4$ to $\mathcal{U}$ satisfy
 $$\der\theta^0 = \theta^1\wedge \theta^4 - 3\theta^2\wedge\theta^3,\,\der \theta^1=0,\, \der\theta^2=0,\,\der\theta^3=0,\,\der \theta^4=0.$$
 In particular, there exist local coordinates $(x^0, x^1, x^2, x^3, x^4)$ such that $\theta^0=\der x^0+x^1\der x^4-3 x^2\der x^3,\, \theta^1=\der x^1, \, \theta^2=\der x^2,\, \theta^3=\der x^3, \theta^4=\der x^4$, which means that the marked Engel structure is flat.

\end{proof}

\subsection{A rigid coframe for marked contact Engel structures}
In the previous section we have explicitly constructed a rigid coframe on a $9$-dimensional bundle over the \emph{flat} marked contact Engel structure. In  this section we apply Cartan's equivalence method to show how to associate a rigid coframe  on a $9$-dimensional bundle to a \emph{general} marked contact Engel structure.

We start as in the proof of Theorem \ref{main}.
We choose an adapted coframe  $(\omega^0,\omega^1,\omega^2,\omega^3,\omega^4)$ that satisfies the structure equations \eqref{differentiatedcoframe2}
%
and as in the beginning of the proof of Theorem \ref{main} we lift it
 to the $5$ well-defined (tautological) $1$-forms
  \begin{align*}\label{thetas}
  \theta^{\mu}=\mathbf{S}^{\mu}{}_{\nu}\omega^{\nu},\quad \mu=0,1,2,3,4,
  \end{align*}
on $\mathcal{U}\times \mathbf{S}$, where $\mathbf{S}$ is the structure group \eqref{matS}. We again reparametrize $\delta=-s_5 s_8$.

Since
\begin{align*}
\der\theta^0 \wedge\theta^0=- \tfrac{s_0}{\delta^3}\theta^1\wedge\theta^4\wedge\theta^0 +\tfrac{3 s_0}{\delta^3}\theta^2\wedge\theta^3\wedge\theta^0
\end{align*}
 we  normalize the coefficient of the $\theta^1\wedge\theta^4$--term in the expansion of $\der\theta^0$  to $1$ by setting
 \begin{equation}\label{szero}s_0=-{\delta^3}.\end{equation}
Then there exists a $1$-form $\theta^5$, which is uniquely defined up to addition of multiples of $\theta^0$,
 satisfying the equation
$$\der\theta^0= -6 \theta^0\wedge \theta^5 + \theta^1\wedge\theta^4 - 3 \theta^2\wedge\theta^3.$$

Computing
\begin{align*}
\der\theta^1 \wedge\theta^0\wedge\theta^1\wedge\theta^4=\tfrac{3(s_1\delta+a {s_5}^3\delta-3 J{s_5}^4{s_7})}{{\delta}^4}\theta^0\wedge\theta^1\wedge\theta^2\wedge\theta^3\wedge\theta^4
\end{align*}
 shows that we can further normalize the $\theta^2\wedge\theta^3$--coefficient in the expansion of $\der\theta^1$ to $0$ by setting
 \begin{equation}\label{sone}s_1=\tfrac{-a \delta {s_5}^3+3 J {s_5}^4{s_7}}{\delta}.\end{equation}
Then there exists a $1$-form $\theta^8$, uniquely defined up to addition of multiples of $\theta^0$ and $\theta^1$, satisfying
$$\der\theta^1\wedge\theta^0=-3 \theta^0\wedge\theta^1\wedge\theta^5-3 \theta^0\wedge\theta^1\wedge\theta^8+\tfrac{3 J {s_5}^5}{\delta^4}\theta^0\wedge\theta^2\wedge\theta^4.$$

Now
\begin{align*}
\der\theta^2\wedge\theta^0\wedge\theta^1=&\tfrac{2(2 \delta s_2-b\delta^2 s_5+2 a \delta {s_5}^2s_7-3J {s_5}^3{s_7}^2)}{\delta^4}\theta^0\wedge\theta^1\wedge\theta^2\wedge\theta^3-3 \theta^0\wedge\theta^1\wedge\theta^2\wedge\theta^5\\&-\theta^0\wedge\theta^1\wedge\theta^2\wedge\theta^8 +\tfrac{2 J{s_5}^5}{\delta^4}\theta^0\wedge\theta^1\wedge\theta^3\wedge\theta^4
\end{align*}
shows that we can normalize the $\theta^2\wedge\theta^3$--term in the expansion of $\der\theta^2$ to $0$ by setting
 \begin{equation}\label{stwo}s_2=\tfrac{s_5}{2\delta}(b\delta^2-2a\delta s_5 s_7+3J{s_5}^2{s_7}^2),\end{equation} and
\begin{align*}
\der\theta^3\wedge\theta^0&\wedge\theta^2\wedge\theta^3  = -\tfrac{c \delta^3+\delta s_3 s_5 - b \delta^2 s_5 s_7 + a\delta {s_5}^2 {s_7}^2-J{s_5}^3{s_7}^3}{\delta^4 s_5} \theta^0\wedge\theta^1\wedge\theta^2\wedge\theta^3\wedge\theta^4
\end{align*}
shows that we can normalize the $\theta^1\wedge\theta^4$--term in the expansion of $\der\theta^3$ to $0$ by setting
\begin{equation}\label{sthree}s_3=-\tfrac{1}{\delta s_5}(c\delta^3 - b \delta^2 s_5 {s_7}+a\delta {s_5}^2{s_7}^2-J {s_5}^3{s_7}^3).\end{equation}

Having performed these normalizations, on the subbundle $\mathcal{G}^9\subset(\mathcal{U}\times \mathbf{S})$ defined  by \eqref{szero}, \eqref{sone}, \eqref{stwo}, \eqref{sthree},  we now have
\begin{equation}\label{5forms}
\begin{aligned}
&\theta^0=-\delta^3\omega^0\\
&\theta^1=\tfrac{{s_5}^3(3J {s_5}{s_7}-a\delta)}{\delta}\omega^0+{s_5}^3\omega^1\\
&\theta^2=\tfrac{s_5(b\delta^2 - 2 a\delta s_5 s_7+3 J {s_5}^2{s_7}^2)}{2\delta}\omega^0+{s_5}^2 s_7\omega^1-\delta s_5\omega^2\\
&\theta^3=\tfrac{-c\delta^3+b\delta^2 s_5 s_7- a\delta {s_5}^2{s_7}^2+J{s_5}^3{s_7}^3}{s_5}\omega^0+s_5{s_7}^2\omega^1-2 \delta s_7 \omega^2 +\tfrac{\delta^2}{s_5}\omega^3\\
&\theta^4=s_4\omega^0+{s_7}^3\omega^1-\tfrac{3\delta {s_7}^2}{s_5}\omega^2+\tfrac{3\delta^2s_7}{{s_5}^2}\omega^3-\tfrac{\delta^3}{{s_5}^3}\omega^4.
\end{aligned}
\end{equation}


  We have further introduced two additional forms $\theta^5$ and $\theta^8$, but on the $9$-dimensional bundle $\mathcal{G}^9$ given by \eqref{szero}, \eqref{sone}, \eqref{stwo}, \eqref{sthree} they are defined up to a certain freedom. 
  It turns out that imposing further normalizations determines  forms $\theta^5$, $\theta^8$ uniquely and in addition picks up unique $1$-forms $\theta^6$ and $\theta^{12}$ that together with the five $1$-forms \eqref{5forms} constitute a coframe on $\mathcal{G}^9$. The normalizations needed are included in the following proposition:
\begin{proposition}\label{propo_coframe}
 The five forms \eqref{5forms} on the $9$-dimensional subbundle $\mathcal{G}^9\subset \mathcal{U}\times \mathbf{S}$ given by \eqref{szero}, \eqref{sone}, \eqref{stwo}, \eqref{sthree} can be supplemented to a rigid coframe $(\theta^0, \theta^1, \theta^2, \theta^3, \theta^4, \theta^5, \theta^6, \theta^8, \theta^{12}) $   which is uniquely determined by the fact that it satisfies
  \begin{equation}\label{5structureeq}
\begin{aligned}
\der\theta^0=&-6\theta^0\wedge\theta^5+\theta^1\wedge\theta^4-3\theta^2\wedge\theta^3\\
\der\theta^1=& - 3 \theta^1\wedge\theta^5 - 3 \theta^1\wedge\theta^8+
 T^1{}_{02} \theta^0\wedge\theta^2 +
 T^1{}_{03} \theta^0\wedge\theta^3 + T^1{}_{04} \theta^0\wedge\theta^4+ T^1{}_{06}  \theta^0\wedge\theta^6 +
 T^1{}_{24}\theta^2\wedge\theta^4\\
  \der\theta^2=& \theta^1\wedge\theta^6- 3\theta^2\wedge\theta^5 -
 \theta^2\wedge\theta^8 +   T^2{}_{03} \theta^0\wedge\theta^3 - T^2{}_{04} \theta^0\wedge\theta^4 +T^2{}_{34}\theta^3\wedge\theta^4\\
  \der\theta^3=&2 \theta^2\wedge\theta^6 - 3\theta^3\wedge\theta^5 + \theta^3\wedge\theta^8+T^3{}_{01}\theta^0\wedge\theta^1 + T^3{}_{02} \theta^0\wedge\theta^2 - T^3{}_{03} \theta^0\wedge\theta^3 + T^3{}_{04} \theta^0\wedge\theta^4\\
\der\theta^4=&6 \theta^0\wedge\theta^{12}+3\theta^3\wedge\theta^6-3\theta^4\wedge\theta^5+3\theta^4\wedge\theta^8,\\
\end{aligned}
\end{equation}
for some functions $T^i{}_{jk}$, and the additional normalization that $\der\theta^5$, when written with respect to the basis of forms $\theta^i\wedge\theta^j$, has zero coefficient at the $\theta^0\wedge\theta^1$ term.
\end{proposition}
We remark that the normalizations given in Proposition \ref{propo_coframe} also uniquely determine the structure functions $T^k{}_{jl}$. In particular we have
\begin{equation}
T^1{}_{24}=-T^1{}_{06}=\tfrac{3}{2}T^2{}_{34}=\tfrac{3 J {s_5}^5}{\delta^4},\end{equation}
and
\begin{equation}
T^1{}_{02}=\tfrac{{s_5}^2 (\delta^4 M - 6 \delta J s_4 {s_5}^3 - 9 c \delta^3 J s_5 s_7 -
 3 \delta^3 J_{\omega^2} s_5 s_7 + 2 \delta^3 L s_5 s_7 - 9 b \delta^2 J {s_5}^2 {s_7}^2 -   9 \delta^2 J_{\omega^3} {s_5}^2 {s_7}^2 + 21 a \delta J {s_5}^3 {s_7}^3 - 9 \delta J_{\omega^4} {s_5}^3 {s_7}^3 - 27 J^2 {s_5}^4 {s_7}^4)}{\delta^8}.
\end{equation}

\begin{remark}
The bundle $\mathcal{G}^9\to \mathcal{U}$  has as  structure group  the subgroup of $\mathbf{S}$ of matrices of the form
\begin{align*}\begin{pmatrix}
-{\delta}^3& 0&0&0&0\\
0&{s_5}^3& 0 & 0&0\\
0&{s_5}^2 s_7& -\delta{s_5}& 0 &0\\
0& s_5 {s_7}^2& -2 \delta  s_7 & \tfrac{\delta^2}{s_5}& 0\\
s_4& {s_7}^3& -\tfrac{3 \delta{s_7}^2}{s_5}&\tfrac{3\delta^2 s_7}{{s_5}^2}&-\tfrac{\delta^3}{{s_5}^3}
              \end{pmatrix}.
\end{align*}
\end{remark}

\begin{remark}
The coframe  constructed in Proposition \ref{propo_coframe} does  not define a Cartan connection. In order to obtain a Cartan connection, more elaborate normalizations are necessary.
\end{remark}


\subsection{Integrable structures and the submaximal models}
Recall that any marked contact Engel structure is called \emph{integrable} if  the rank $2$ distribution $\mathcal{D}^{\sigma}$, which in terms of an adapted coframe is given by $$\mathcal{D}^\sigma=\mathrm{ker}(\omega^0,\omega^1,\omega^2),$$ is integrable.
%
The following proposition shows that integrability of a marked contact Engel structure precisely corresponds to the vanishing of the first (relative) invariant from Theorem \ref{main}.
\begin{proposition}\label{prop.integrability}
A marked contact Engel structure is integrable if and only if $J=0$.
\end{proposition}

\begin{proof}
Let $(\omega^0,\omega^1,\omega^2,\omega^3,\omega^4)$ be any adapted coframe that satisfies the structure equations \eqref{differentiatedcoframe2} with associated function $J$. A direct computation shows that
\begin{align*}&\der\omega^0\wedge\omega^0\wedge\omega^1\wedge\omega^2=0\\
&\der\omega^1\wedge\omega^0\wedge\omega^1\wedge\omega^2=0\\
&\der\omega^2\wedge\omega^0\wedge\omega^1\wedge\omega^2= 2 \ J\; \omega^0\wedge\omega^1\wedge\omega^2\wedge\omega^3\wedge\omega^4.
\end{align*}
\end{proof}

For integrable marked  contact Engel structures the structure equations simplify as follows.
\begin{proposition}\label{propo_coframe_int}
Consider an integrable marked contact Engel structure. Then the five forms \eqref{5forms} on the $9$-dimensional  subbundle $\mathcal{G}^9$ of $\mathcal{U}\times \mathbf{S}$ given by \eqref{szero}, \eqref{sone}, \eqref{stwo}, \eqref{sthree} can be supplemented to a rigid coframe $(\theta^0, \theta^1, \theta^2, \theta^3, \theta^4, \theta^5, \theta^6, \theta^8, \theta^{12}) $   which is uniquely determined by the structure equations
 \begin{equation}\label{5structureeq}
\begin{aligned}
\der\theta^0=&-6\theta^0\wedge\theta^5+\theta^1\wedge\theta^4-3\theta^2\wedge\theta^3\\
\der\theta^1=&-3 \theta^1\wedge\theta^5-3\theta^1\wedge\theta^8+\tfrac{{s_5}^2(\delta M +2 L {s_5}{s_7})}{{\delta}^5}\theta^0\wedge\theta^2-\tfrac{{s_5}^4 L}{{\delta}^5} \theta^0\wedge\theta^3\\
\der\theta^2 = &  \theta^1\wedge\theta^6 - 3\theta^2\wedge\theta^5 - \theta^2\wedge\theta^8 -\tfrac{{s_5}^2(5\delta P-3 \delta M +4 L {s_5}{s_7})}{4{\delta}^5}\theta^0\wedge\theta^3\\
\der\theta^3 = & 2\theta^2\wedge\theta^6 -3\theta^3\wedge\theta^5 + \theta^3\wedge\theta^8-\tfrac{{\delta}^4 U -2 {\delta}^3 R {s_5}{s_7}+{\delta}^2 Q {s_5}^2{s_7}^2+2\delta P {s_5}^3{s_7}^3 + L {s_5}^4{s_7}^4}{{\delta}^5{s_5}^5}\theta^0\wedge\theta^1 \\
 &-\tfrac{2(\delta^3 R-{\delta}^2 Q {s_5}{s_7} - 3 {\delta} P {s_5}^2{s_7}^2 -2 L {s_5}^3{s_7}^3)}{\delta^5{s_5}^2}\theta^0\wedge\theta^2-\tfrac{\delta^2 Q+6\delta P s_5 s_7 +6 L {s_5}^2{s_7}^2}{\delta^3}\theta^0\wedge\theta^3+\tfrac{(M-P){s_5}^2}{2{\delta}^4}\theta^0\wedge\theta^4\\
\der\theta^4=&6 \theta^0\wedge\theta^{12}+3\theta^3\wedge\theta^6-3\theta^4\wedge\theta^5+3\theta^4\wedge\theta^8\\
\end{aligned}
\end{equation}
and the additional normalization that $\der\theta^5$, when expressed with respect to the basis of forms $\theta^i\wedge\theta^j$, has zero coefficient at the $\theta^0\wedge\theta^1$ term.
\end{proposition}

In particular, the structure equations for  integrable structures exhibit a new relative invariant for these structures that is independent of the filtration of invariant conditions from Section \ref{themain}.

\begin{proposition}\label{prop_seconddi}
Consider an integrable  marked contact Engel structure on $\mathcal{U}$.
Let $(\omega^0,\omega^1,\omega^2, \omega^3,\omega^4)$ be an adapted  coframe satisfying the structure equations \eqref{differentiatedcoframe2} with $J=0$, and let $\phi$ be the $3$-form  defined as
\begin{equation}\phi=\omega^1\wedge\omega^2\wedge\omega^3- a \omega^0\wedge\omega^2\wedge\omega^3+\tfrac{1}{2}b \omega^0\wedge\omega^1\wedge\omega^3- c \omega^0\wedge\omega^1\wedge\omega^2.\end{equation}
\begin{enumerate}
\item Then the  rank $2$-distribution
\begin{align*}\mathcal{R}^{\sigma}=\mathrm{ker}(\phi)
\end{align*}
on $\mathcal{U}$ is invariantly associated to the marked contact twisted cubic structure.
\item This distribution is integrable if and only if $M-P$ vanishes.
\end{enumerate}
 \end{proposition}

\begin{proof}
Let $\theta^0,\theta^1,\theta^2,\theta^3,\theta^4$ be the invariant forms \eqref{5forms} on $\mathcal{G}^9$ with $J=0$.
A direct calculation gives
\begin{align*}
\theta^1\wedge\theta^2\wedge\theta^3=\delta^3 {s_5}^3(\omega^1\wedge\omega^2\wedge\omega^3- a \omega^0\wedge\omega^2\wedge\omega^3+\tfrac{1}{2}b \omega^0\wedge\omega^1\wedge\omega^3- c \omega^0\wedge\omega^1\wedge\omega^2).
\end{align*}
This shows that the kernel of $\theta^1\wedge\theta^2\wedge\theta^3$ descends to a  distribution $\mathcal{R}^{\sigma}=\mathrm{ker}(\phi)$ on $\mathcal{U}$, which is independent of the choice of adapted coframe, and thus invariantly associated to the marked contact twisted cubic structure.

%
Since
$$\phi=(\omega^1-a \omega^0)\wedge ( \omega^2-\tfrac{1}{2}b\omega^0)\wedge ( \omega^3- c \omega^0)$$
and
\begin{equation*}
\begin{aligned}
&\der(\omega^1- a\omega^0)\wedge(\omega^1-a\omega^0)\wedge(\omega^2-\tfrac{b}{2}\omega^0)\wedge(\omega^3-c\omega^0)=0,\\
&\der(\omega^2- \tfrac{b}{2}\omega^0)\wedge(\omega^1-a\omega^0)\wedge(\omega^2-\tfrac{b}{2}\omega^0)\wedge(\omega^3-c\omega^0)=0,\\
&\der(\omega^3- c\omega^0)\wedge(\omega^1-a\omega^0)\wedge(\omega^2-\tfrac{b}{2}\omega^0)\wedge(\omega^3-c\omega^0)=\tfrac{1}{2} (M-P) {s_5}^2\omega^0\wedge\omega^1\wedge\omega^2\wedge\omega^3\wedge\omega^4 ,
\end{aligned}
\end{equation*}
  integrability of $\mathcal{R}^{\sigma}$ is equivalent to the vanishing of $M-P$.
%
\end{proof}


\begin{remark}{\textbf{(Submaximal branch)}}
The structure equations for integrable marked contact Engel structures displayed in Proposition \ref{propo_coframe_int} show that for the subclass of structures with nowhere vanishing relative invariant $M-P$, we can further normalize the coefficient $T^3{}_{04}=\tfrac{(M-P){s_5}^2}{2\delta^4}$.  It is also visible that the sign of $M-P$ is an invariant of integrable marked contact Engel structures.

We could now proceed as follows. We could  normalize the coefficient $T^3{}_{04}$ to $\tfrac{\epsilon}{2}$ with $\epsilon=\mathrm{sign}(M-P)$ (or any non-zero multiple of $\epsilon$).
 This means that we restrict to the $8$-dimensional subset $\mathcal{G}^8\subset \mathcal{G}^9$ defined by
$$ s_5 =\tfrac{\delta^2}{\sqrt{\epsilon (M-P)}}.$$
On this subset, the pullbacks of the $1$-form $\theta^8$  is linearly dependent on the pullbacks of the remaining forms $\theta^0,\theta^1,\theta^2,\theta^3,\theta^4,\theta^5,\theta^6,\theta^{12},$ which define a coframe on $\mathcal{G}^8$.
If we now compute the structure equations with respect to the coframe on $\mathcal{G}^8$ and assume that all of the structure functions  are constants, we arrive at the structure equations  \eqref{submaxsyst}. These are Maurer-Cartan equations for $\mathfrak{sl}(3,\mathbb{R})$ if $\epsilon> 0$ and Maurer-Cartan equations for $\mathfrak{su}(2,1)$ if $\epsilon< 0$. The analysis in Section \ref{sec_tree} (where we will start by normalizing $T^1{}_{03}=-\tfrac{{s_5}^4 L}{{\delta}^5}$ rather than $T^3{}_{04}$) will show that these are, up to local equivalence,  the only marked contact Engel structures with $8$-dimensional transitive symmetry algebra, and we will refer to these structures as the \emph{submaximal} marked contact Engel structures.
\end{remark}

\subsection{A tree of homogeneous models}\label{sec_tree}
The goal of this section is to find all locally non-equivalent homogeneous marked contact Engel structures with symmetry group of dimension $\geq 6$.
To this end, we return to the conditions from Theorem \ref{main}, which  divide marked contact Engel structures into classes of mutually non-equivalent structures.
We apply  Cartan's reduction procedure   to determine the maximally symmetric homogeneous structures in each of the  branches determined by the conditions from Theorem \ref{main}.

We will, in the following, often abuse notation. In particular, we will denote various different subbundles $\mathcal{G}^i\subset\mathcal{G}^9$ of dimension $i$ by the same symbol. Moreover, we will frequently pullback the forms $\theta^0,\theta^1,\theta^2, \theta^3,\theta^4,\theta^5,\theta^6,\theta^8,\theta^{12}$ to these various subbundles and always reuse the same names for the pulled back forms. For different $\mathcal{G}^i$, we will be choosing subsets of these forms that constitute  coframes on the subbundles $\mathcal{G}^i$. We will express the exterior derivatives $\der\theta^i$ of these coframe forms in terms of the bases of $2$-forms given by the wedge products $\theta^i\wedge\theta^j$ of the coframe forms,   and refer to the equations
$$\der\theta^k=T^k{}_{ij}\theta^i\wedge\theta^j,$$
as the \emph{structure equations}  and to the  functions $T^k{}_{ij}$ as the \emph{structure functions} (with respect to the coframe).

\subsubsection{The branch $J\neq 0$}
Here we shall assume that $J\neq0$. This assumption allows us to perform a number of normalizations. We proceed as follows. First, looking at $\der\theta^1$ in Proposition \ref{propo_coframe}, we see that we can normalize the coefficient $T^1_{24}=\tfrac{3 {s_5}^5}{\delta^4} J$ to any non-zero value, and we shall normalize it to $3$. We also see that we can normalize the coefficient
$T^1_{02}$ to zero. This means that we restrict to a subbundle $\mathcal{G}^7\subset\mathcal{G}^9$ given by
$$s_5=\left(\tfrac{\delta^4}{J}\right)^{\frac{1}{5}},\, s_4=\tfrac{\delta^4 M - 9 c \delta^3 J s_5 s_7 - 3 \delta^3 J_{\omega^2} s_5 s_7 + 2 \delta^3 L s_5 s_7 -  9 b \delta^2 J {s_5}^2 {s_7}^2 - 9 \delta^2 J_{\omega^3} {s_5}^2 {s_7}^2 +
 21 a \delta J {s_5}^3 {s_7}^3 - 9 \delta J_{\omega^4} {s_5}^3 {s_7}^3 - 27 J^2 {s_5}^4 {s_7}^4}{6 \delta J {s_5}^3}.$$
 We pullback the forms $\theta^0, \theta^1, \theta^2,  \theta^3, \theta^4, \theta^5, \theta^6, \theta^8, \theta^{12}$ to $\mathcal{G}^7$, where they are no longer independent, and express $\theta^8$ and $\theta^{12}$  as linear combinations with functional coefficients of the remaining forms. Now we compute the structure equations with respect to the coframe on $\mathcal{G}^7$ given by $\theta^0,\dots,\theta^6$. Looking at these structure equations shows that we can now normalize the coefficient of $\der\theta^1$ at the $\theta^1\wedge\theta^4$ term to zero, which determines a $6$-dimensional subbundle $\mathcal{G}^6\subset\mathcal{G}^7$ given by
 $$s_7=\tfrac{\delta^{\frac{1}{5}}(3 a J -J_{\omega^4})}{14 J^{\frac{9}{5}}}.$$
 On this subbundle, which is parametrized by the coordinates on $\mathcal{U}$ and the fibre coordinate $\delta$, the forms $\theta^0,\dots,\theta^5$ define a coframe
 that satisfies structure equations  of the form
\begin{equation}\label{Jnot0coframe}
\begin{aligned}
\der\theta^0=&-6 \theta^0\wedge\theta^5+\theta^1\wedge\theta^4-3 \theta^2\wedge\theta^3\\
\der\theta^1=&\tfrac{\alpha_1}{\delta^3} \theta^0 \wedge\theta^1 +\tfrac{\alpha_2}{\delta^{\frac{12}{5}}}\theta^0\wedge\theta^2+ \tfrac{\alpha_3}{\delta^{\frac{9}{5}}}\theta^0\wedge\theta^3 +
 \tfrac{\alpha_4}{\delta^{\frac{6}{5}}} \theta^0\wedge\theta^4 + \tfrac{\alpha_5}{\delta^{\frac{9}{5}}} \theta^1\wedge\theta^2 +
 \tfrac{\alpha_6}{\delta^{\frac{6}{5}}} \theta^1\wedge\theta^3 \\&
  -  \tfrac{24}{5} \theta^1\wedge\theta^5  + 3 \theta^2\wedge\theta^4\\
\der\theta^2=&\tfrac{\alpha_7}{\delta^{\frac{18}{5}}} \theta^0 \wedge\theta^1 + \tfrac{\alpha_8}{\delta^3} \theta^0 \wedge\theta^2 +
 \tfrac{\alpha_9}{\delta^{\frac{12}{5}}} \theta^0 \wedge\theta^3 + \tfrac{5\alpha_5}{6\delta^{\frac{9}{5}}} \theta^0 \wedge\theta^4 +
 \tfrac{\alpha_{10}}{\delta^{\frac{12}{5}}} \theta^1 \wedge\theta^2 + \tfrac{\alpha_{11}}{\delta^{\frac{9}{5}}} \theta^1 \wedge\theta^3\\ &
 -\tfrac{3\alpha_4 + 5 \alpha_6}{9\delta^{\frac{6}{5}}} \theta^1 \wedge\theta^4 + \tfrac{\alpha_6}{3\delta^{\frac{6}{5}}}  \theta^2 \wedge\theta^3  - \tfrac{18}{5} \theta^2 \wedge\theta^5 + 2 \theta^3 \wedge\theta^4\\
\der\theta^3=&\tfrac{\alpha_{12}}{\delta^{\frac{21}{5}}} \theta^0 \wedge\theta^1 + \tfrac{\alpha_{13}}{\delta^{\frac{18}{5}}}\theta^0\wedge\theta^2+
 \tfrac{\alpha_{14}}{\delta^3} \theta^0 \wedge\theta^3
    + \tfrac{6 \alpha_9 +75\alpha_{10}+25\alpha_2}{15 \delta^{\frac{12}{5}}} \theta^0 \wedge\theta^4
+\tfrac{2(\alpha_1-3\alpha_8)}{3\delta^3}\theta^1\wedge\theta^2   \\&
-\tfrac{3 \alpha_{10}+\alpha_2}{3\delta^{\frac{12}{5}}} \theta^1 \wedge\theta^3 + \tfrac{\alpha_5 +6 \alpha_{11}}{3\delta^{\frac{9}{5}}} \theta^2 \wedge\theta^3 -
 \tfrac{6 \alpha_4+10\alpha_6}{9\delta^{\frac{6}{5}}}  \theta^2 \wedge\theta^4
   -  \tfrac{12}{5} \theta^3 \wedge\theta^5\\
\der\theta^4=&\tfrac{\alpha_{15}}{\delta^{\frac{24}{5}}} \theta^0 \wedge\theta^1 + \tfrac{\alpha_{16}}{\delta^{\frac{21}{5}}} \theta^0 \wedge\theta^2 + \tfrac{\alpha_{17}}{\delta^{\frac{18}{5}}} \theta^0 \wedge\theta^3 + \tfrac{\alpha_{18}}{\delta^3} \theta^0 \wedge\theta^4+\tfrac{\alpha_1-3\alpha_8}{\delta^3}\theta^1\wedge\theta^3 -\tfrac{3\alpha_{10}+\alpha_2}{\delta^{\frac{12}{5}}} \theta^1 \wedge\theta^4 \\& +\tfrac{\alpha_2}{\delta^{\frac{12}{5}}} \theta^2 \wedge\theta^3 +
 \tfrac{\alpha_5}{\delta^{\frac{9}{5}}} \theta^2 \wedge\theta^4 - \tfrac{3 \alpha_4+2\alpha_6}{3\delta^{\frac{6}{5}}} \theta^3 \wedge\theta^4 -
 \tfrac{6}{5} \theta^4 \wedge\theta^5\\
\der\theta^5=&\tfrac{\alpha_{19}}{\delta^{\frac{27}{5}}} \theta^0 \wedge\theta^1 + \tfrac{\alpha_{20}}{\delta^{\frac{24}{5}}} \theta^0 \wedge\theta^2 +  \tfrac{\alpha_{21}}{\delta^{\frac{21}{5}}} \theta^0 \wedge\theta^3 + \tfrac{\alpha_{22}}{\delta^{\frac{18}{5}}} \theta^0 \wedge\theta^4 -\tfrac{3 \alpha_{12}+\alpha_{16}}{6\delta^{\frac{21}{5}}} \theta^1 \wedge\theta^2 -  \tfrac{\alpha_{17}-3\alpha_7}{6 \delta^{\frac{18}{5}}} \theta^1 \wedge\theta^3\\
& -\tfrac{\alpha1+\alpha_{18}}{6\delta^3} \theta^1 \wedge\theta^4 +
\tfrac{\alpha_8+\alpha_{14}}{2\delta^3} \theta^2 \wedge\theta^3 +
 \tfrac{6\alpha_9+75\alpha_{10}+20\alpha_2}{30\delta^{\frac{12}{5}}} \theta^2 \wedge\theta^4 -
 \tfrac{2\alpha_3 - 5 \alpha_5}{12\delta^{\frac{9}{5}}} \theta^3 \wedge\theta^4,
 \end{aligned}
 \end{equation}
 where $\alpha_1,\dots,\alpha_{21}$ are the pullbacks of functions on $\mathcal{U}$, that is,   as functions on $\mathcal{G}^6$ they do not depend on $\delta$.

 Now we are looking for homogeneous structures with six dimensional symmetry algebra in this branch. For such structures  all of the structure functions  are constants. In particular, all of those that depend on $\delta$ have to be identically zero. On the other hand, one easily checks that this constant coefficient  system
   \begin{equation}\label{Jnotzerosyst}
\begin{aligned}
\der\theta^0=&-6 \theta^0\wedge\theta^5+\theta^1\wedge\theta^4-3 \theta^2\wedge\theta^3\\
\der\theta^1=& -  \tfrac{24}{5} \theta^1\wedge\theta^5 + 3 \theta^2\wedge\theta^4\\
\der\theta^2=&- \tfrac{18}{5} \theta^2 \wedge\theta^5 + 2 \theta^3 \wedge\theta^4\\
\der\theta^3=& -  \tfrac{12}{5} \theta^3 \wedge\theta^5\\
\der\theta^4=&-  \tfrac{6}{5} \theta^4 \wedge\theta^5\\
\der\theta^5=&0.
 \end{aligned}
\end{equation}
 is closed, that is, $\der^2\theta^i=0$, for all $i=0,1,2,3,4,5$.  This means that there is a unique local model with $6$-dimensional symmetry algebra in this branch whose  symmetry algebra has Maurer-Cartan equations \eqref{Jnotzerosyst}.

There may be homogeneous models with $5$-dimensional symmetry algebra in this branch as well.

\subsubsection{The branch $J=0$, $L\neq 0$}
For integrable structures, we have seen that $L$ defines a relative invariant. We shall assume here that it be nowhere vanishing. Similar as before, this assumption allows us to perform normalisations. We normalize the coefficient
$T^1_{02}$
  to zero, and the coefficient $T^1_{03}$
   to $1$. On the subbundle determined by these normalizations, $\theta^6$ and $\theta^8$ are expressible in terms of the remaining forms, which constitute a coframe. Looking at $\der\theta^3$ (with the expressions for $\theta^6$ and $\theta^8$ inserted) we now see that the coefficient at the $\theta^1\wedge\theta^3$ term can be normalized to zero.
Together, these normalizations determine a $6$-dimensional subbundle $\mathcal{G}^6\subset\mathcal{G}^9$ defined by
$$s_7=\tfrac{M}{2 L^{\frac{4}{5}} {s_5}^{\frac{1}{5}}},\quad   \delta=-L^{\frac{1}{5}} {s_5}^{\frac{4}{5}},\quad s_4=-\tfrac{8 L^2 L_{\omega^1} +16 c L^2 M -4 L L_{\omega^2} M +8 bL M^2 +2 L_{\omega^3} M^2 +a M^3}{8 L^{\frac{12}{5}}{s_5}^{\frac{3}{5}} },$$
on which (the pullbacks of) the forms $\theta^0, \theta^1,\theta^2, \theta^3, \theta^4, \theta^5$ define a coframe.

Now, if there were homogeneous structures with $6$-dimensional symmetry algebra in this branch, then for these structures all of the structure functions of the structure equations with respect to the coframe $(\theta^0, \theta^1,\theta^2, \theta^3, \theta^4, \theta^5)$ on $\mathcal{G}^6$ must be constant. However, this assumption leads to a contradiction, and we conclude that there are no homogeneous models with $6$-dimensional symmetry algebra in this branch.

It turns out that there are structures with $5$-dimensional transitive symmetry algebra in this branch, and below we describe how to find them.
The structure equations lead us to distinguish two subclasses of structures, those for which the relative invariant $M-P$ vanishes and those for which it does not vanish.

 We first consider the class of structures for which $M-P\neq 0$, which allows us to normalize
 the coefficient at the $\theta^0\wedge\theta^3$ term of $\der\theta^2$. This determines a $5$-dimensional subbundle of $\mathcal{G}^6\to \mathcal{U}$, and thus a rigid coframe $\theta^0,\theta^1,\theta^2,\theta^3,\theta^4$ on $\mathcal{U}$. However, assuming that the structure equations with respect to this coframe have only constant structure functions quickly leads to a contradiction, and we conclude that there are no homogeneous structures in this branch.

We shall henceforth assume that $M-P=0$.  In this case, the structure equations exhibit a new relative invariant, namely $5bL+2L_{\omega^3}$. This leads us to branch further into the subclass of structures for which $5bL+2L_{\omega^3}$ is vanishing and the subclass for which is non-vanishing.
Assuming that $5bL+2L_{\omega^3}\neq 0$  allows us to normalize, namely we normalize the coefficient of $\der\theta^1$ at the $\theta^1\wedge\theta^3$ term to $\frac{3}{8}$. This determines a $5$-dimensional subbundle $\mathcal{G}^5\subset\mathcal{G}^6,$ given by
$$s_5=\left(\tfrac{(5b L+2 L_{\omega^3})^5}{L^7}\right)^{\frac{1}{3}},$$ and a rigid coframe $\theta^0,\theta^1,\theta^2,\theta^3,\theta^4$ on $\mathcal{U}$.  Assuming that all of the structure functions with respect to this coframe  are constant and using that $\der^2=0$, we find that there is a locally unique  homogeneous model with $5$-dimensional symmetry algebra in this branch. It has  Maurer-Cartan equations
\begin{equation}\begin{aligned}\label{J=0Lneq0}
&\der\theta^0 = -\tfrac{5}{6} \theta^0\wedge\theta^3 -24 \theta^0\wedge \theta^4 + \theta^1\wedge\theta^4 -3 \theta^2\wedge\theta^3\\
&\der\theta^1 =   \theta^0\wedge\theta^3 -\tfrac{2}{3}\theta^1\wedge\theta^3 -30 \theta^1\wedge\theta^4\\
&\der\theta^2 = -\tfrac{1}{2}  \theta^2\wedge\theta^3 -18 \theta^2\wedge\theta^4\\
&\der\theta^3 = -6\theta^3\wedge\theta^4 \\
&\der\theta^4=\tfrac{1}{6} \theta^3\wedge\theta^4.\\
 \end{aligned}\end{equation}
Further analysis shows that there are no homogeneous models in the branch $5bL+2L_{\omega^3}=0$.

\subsubsection{The branch $J=L=0$, $M\neq0$}



Looking at \eqref{5structureeq}, we see that under the assumption $J=L=0$, the coefficient  $T^1_{02}$ reads $\tfrac{{s_5}^2 M}{\delta^4}$. This shows  that the sign of $M$ is an invariant, and  we normalize this coefficient to $\epsilon=\mathrm{sign}(M)$.
More precisely, we restrict to a hypersurface $\mathcal{G}^8$ in $\mathcal{G}^9$ defined by  $$s_5=\tfrac{\delta^2}{\sqrt{\epsilon M}}.$$
We pullback the $1$-forms $\theta^0, \theta^1, \theta^2, \theta^3,\theta^4, \theta^5, \theta^6, \theta^8, \theta^{12}$ to $\mathcal{G}^8$, and find that on this hypersurface $\theta^8$ is linearly dependent on the other $1$-forms.

 Having done that, we compute $\der \theta^i$, for all $i=0,1,2,3,4,5,6,12,$ on $\mathcal{G}^8$  in terms of the basis of $2$-forms  $\theta^i\wedge\theta^j$.
 Inspecting the system shows that the coefficient of $\der\theta^5$ at the $\theta^2\wedge\theta^3$ term reads
 $$\tfrac{\sqrt{\epsilon M} Q+6 \delta P s_7}{2\delta^3\sqrt{\epsilon M}}.$$
  We now branch according to whether $P$ vanishes or not.

  Assuming that $P\neq 0$, allows to normalize the above coefficient to zero. This determines a $7$-dimensional subbundle $\mathcal{G}^7$ of $\mathcal{G}^8$, given by $s_7= -\tfrac{\sqrt{\epsilon M} Q}{6 \delta P}.$ We pullback the forms $\theta^i$, $i=0,1,2,3,4,5,6,12,$ and express $\theta^6$ as a combination of the remaining forms.  We compute the structure equations with respect to the coframe on $\mathcal{G}^7$, and note that we can now normalize the coefficient
  of $\der\theta^4$ at the $\theta^2\wedge\theta^3$ term to zero. This determines a $6$-dimensional subbundle $\mathcal{G}^6$ of $\mathcal{G}^7$, given by
  $$s_4 =\tfrac{(\epsilon M)^{\frac{3}{2}} (60 c M P Q - 30 M_{\omega^2} P Q + 180 c P^2 Q + 126 P P_{\omega^2} Q +  84 a_{\omega^0} Q^2 + 84 a a_{\omega^1} Q^2 - 21 b^3 Q^2 - 21 b M Q^2 + 81 b P Q^2 +  2 a Q^3 - 72 P^2 Q_{\omega^2})}{432 (\delta P)^3},$$
  on which we express $\theta^{12}$ in terms of the remaining forms. Moreover, as a consequence of the assumption that $P\neq 0,$ we have $2 c M-M_{\omega^2}\neq 0.$ This  allows to further normalize the coefficient of $\der\theta^1$ at the $\theta^1\wedge\theta^2$ term (with respect to the coframe on $\mathcal{G}^6$) to any non-zero value, and we shall normalize it to $12$. This determines a $5$-dimensional subbundle $\mathcal{G}^5\subset\mathcal{G}^6$, given by $\delta=\left(\tfrac{2 c M-M_{\omega^2}}{8\sqrt{\epsilon M}}\right)^{\frac{1}{3}}$. We have now obtained
   a unique coframe on the $5$-manifold $\M$. Inspecting the structure equations of this coframe shows that there are two locally non-equivalent homogeneous models  with $5$-dimensional symmetry algebras in this branch, whose Maurer-Cartan equations read
\begin{equation}\begin{aligned}\label{J=0L=0Mneq0Pneq0}
&\der\theta^0 = -\tfrac{15}{2} \theta^0\wedge\theta^2 -\tfrac{1}{6}\epsilon \theta^0\wedge \theta^4 + \theta^1\wedge\theta^4 -3 \theta^2\wedge\theta^3\\
&\der\theta^1 =  \epsilon \theta^0\wedge\theta^2 -3 \theta^1\wedge\theta^2 -\tfrac{1}{3}\epsilon \theta^1\wedge\theta^4\\
&\der\theta^2 = \tfrac{1}{4}  \theta^0\wedge\theta^1 -\tfrac{1}{12}\epsilon \theta^0\wedge\theta^3-\tfrac{1}{2}\theta^1\wedge\theta^3-\tfrac{1}{6}\epsilon\theta^2\wedge\theta^4 \\
&\der\theta^3 = \tfrac{9}{2}\theta^0\wedge\theta^2+\tfrac{1}{6}\epsilon\theta^0\wedge\theta^4+9\epsilon\theta^1\wedge\theta^2+3\theta^2\wedge\theta^3 \\
&\der\theta^4=-\tfrac{27}{4}\epsilon \theta^0\wedge\theta^1+\tfrac{9}{4}\theta^0\wedge\theta^3+\tfrac{27}{2}\epsilon\theta^1\wedge\theta^3+\tfrac{9}{2}\theta^2\wedge\theta^4\\
 \end{aligned}\end{equation}
  where $\epsilon=\pm 1$. For these structures $3P-2M=0$.

 Next we assumes that $P=0$. Analysing the differential consequences of this assumption, we obtain that for such structures
 $$\der\theta^1=-\tfrac{9 Q}{2\delta^3}\theta^0\wedge\theta^1+\epsilon\theta^0\wedge\theta^2+\tfrac{3(2 c M-M_{\omega^2})}{2\delta^3\sqrt{\epsilon M}}\theta^1\wedge\theta^2-12\theta^1\wedge\theta^5.$$
 In particular, we can further branch into those structures for which $Q$ vanishes and  those for which it does not vanish. The assumption  $Q\neq 0$ allows to perform further normalizations, which determine a unique coframe on the $5$-dimensional manifold. Further analysis shows that there are no homogeneous models in this branch.

 On the other hand,  assuming that $Q=0$ and analyzing the differential consequences  one obtains also that $R=S=0$. The only structures satisfying these assumptions are the submaximally symmetric structures, with structure equations
 \begin{equation}\label{submaxsyst}
\begin{aligned}
&\der\theta^0 =-6\theta^0\wedge\theta^5+\theta^1\wedge\theta^4-3\theta^2\wedge\theta^3\\
&\der\theta^1 = \epsilon \theta^0\wedge\theta^2-12\theta^1\wedge\theta^5  \\
&\der\theta^2 =\tfrac{3}{4} \epsilon \theta^0\wedge\theta^3+\theta^1\wedge\theta^6-6\theta^2\wedge\theta^5\\
&\der\theta^3 =\tfrac{1}{2}\epsilon\theta^0\wedge\theta^4+2\theta^2\wedge\theta^6 \\
&\der\theta^4=6\theta^0\wedge\theta^{12}+3\theta^3\wedge\theta^6+6\theta^4\wedge\theta^5\\
&\der\theta^5=-\tfrac{1}{12}\epsilon\theta^0\wedge\theta^6-\theta^1\wedge\theta^{12}+\tfrac{1}{12}\epsilon\theta^2\wedge\theta^4\\
&\der\theta^6=6\theta^2\wedge\theta^{12}-\tfrac{3}{4}\epsilon\theta^3\wedge\theta^4-6\theta^5\wedge\theta^6\\
&\der\theta^{12}=\tfrac{1}{6}\epsilon\theta^4\wedge\theta^6-12\theta^5\wedge\theta^{12}.\\
\end{aligned}
\end{equation}
These are Maurer-Cartan equations for $\mathfrak{sl}(3,\mathbb{R})$ if $\epsilon<0$ and Maurer-Cartan equations for $\mathfrak{su}(2,1)$ if $\epsilon> 0$.

\subsubsection{The branch $J=L=M=0$, $P\neq0$}
Looking at the structure equations \eqref{5structureeq}, we see that under the assumptions $J=L=M=0$ and $P\neq0$ we can normalize the coefficient  $T^3_{03}$
 to zero, and then, on the subbundle determined by this reduction, express $\theta^6$ in terms of the other forms. Having done that,  we compute $\der\theta^4$   and normalize the coefficient at the $\theta^2\wedge\theta^3$ term to zero,   and then we  normalize the  coefficient at the $\theta^0\wedge\theta^3$ term in $\der\theta^2$ to $-\tfrac{5 \epsilon}{4}$, where $\epsilon=\mathrm{sign}(P)$. These normalizations determine a $6$-dimensional subbundle $\mathcal{G}^6\subset\mathcal{G}^9$
 on which  $\theta^6, \theta^8, \theta^{12}$ are expressible in terms of the remaining forms $\theta^0,\dots,\theta^5$, which form a coframe. Assuming that the structure equations have only constant coefficients yields a contradiction, and we conclude that there are no homogeneous models with $6$-dimensional symmetry algebra in this branch. There may be models with $5$-dimensional transitive symmetry algebra in this branch.

%

 \subsubsection{The branch $J=0$, $L=0$, $M=0$, $P=0$, $Q\neq 0$}
Here the assumptions allow to normalize the coefficient at $\theta^0\wedge\theta^2$ of $\der\theta^3$ to zero, the coefficient at $\theta^0\wedge\theta^3$ of $\der\theta^3$ to one, and the coefficient at $\theta^0\wedge\theta^2$ of $\der\theta^6$ to zero.  This determines a $6$-dimensional subbundle $\mathcal{G}^6\subset\mathcal{G}^9$  given by
$$s_7=-\tfrac{R}{Q^{\frac{2}{3}}s_5},\quad s_5=-Q^{\frac{1}{3}},\quad s_4=\tfrac{2a_{\omega^1\omega^1}Q^2-4c^3Q^2+8cQ^2R+2QQ_{\omega^2}R-3bQR^2+2aR^3-2Q^2R_{\omega^2}-3bQ^2S}{2Q^2{s_5}^3}.$$  We express the pullbacks of the forms $\theta^5, \theta^6$ and $\theta^{12}$ to $\mathcal{G}^6$  in terms of $\theta^0, \theta^1, \theta^2, \theta^3, \theta^4, \theta^8$. Assuming that the structure equations have only constant coefficients then quickly implies that they are of the form
 \begin{equation}\begin{aligned}\label{J=0L=0M=0P=0Qneq0}
&\der\theta^0 =  \theta^1\wedge\theta^4 -3 \theta^2\wedge\theta^3\\
&\der\theta^1 =   \tfrac{1}{2}\theta^0\wedge\theta^1  -3 \theta^1\wedge\theta^8\\
&\der\theta^2 =   \tfrac{1}{2}\theta^0\wedge\theta^2  - \theta^2\wedge\theta^8\\
&\der\theta^3 =   -\tfrac{1}{2}\theta^0\wedge\theta^3  + \theta^3\wedge\theta^8\\
&\der\theta^4=  -\tfrac{1}{2}\theta^0\wedge\theta^4  +3 \theta^4\wedge\theta^8\\
&\der\theta^8=  -\tfrac{1}{2}\theta^1\wedge\theta^4  +\tfrac{1}{2} \theta^2\wedge\theta^3\,.
\end{aligned}\end{equation}
This system is closed, and can be viewed as the Maurer-Cartan equations of $\mathfrak{sl}(2,\mathbb{R})\oplus\mathfrak{sl}(2,\mathbb{R})$ with respect to a basis of left-invariant forms.  In particular, there is a locally unique maximally symmetric homogeneous model in this branch with $6$-dimensional symmetry algebra isomorphic to $\mathfrak{sl}(2,\mathbb{R})\oplus\mathfrak{sl}(2,\mathbb{R})$.

There may be homogeneous models with $5$-dimensional symmetry algebras in this branch  as well.

\subsection{Summary}

\tikzset{
  round/.style  = {circle,draw},
  notround/.style   = {draw, rounded corners=6pt, thin,align=center,
                   text width=6em},
   longer/.style   = {draw, rounded corners=6pt, thin,align=center,
                   text width=8em},
    shorter/.style   = {draw, rounded corners=6pt, thin,align=center,
                   text width=3em}
}


\begin{table}[h] \label{table}
\caption{The following graph shows the maximal symmetry dimension for homogeneous models in various branches of marked contact Engel structures. }
\centering
\begin{tikzpicture}[
  level 1/.style={sibling distance=30mm},
  >=latex]

\node [circle,draw] (j){$J$}
  child {node [round] (l) {$L$}
  child {node [round]  (m) {$M$}
  child {node [round]  (p) {$P$}
   child {node [round]  (q) {$Q$}
   child {node [round]  (r) {$R$}
   child {node[round]  (u) {$S$}
   child {node [notround] (max) {maximal symmetry (Theorem \ref{main})}
  }
  child {node [notround] (g) {no homog. model}
};
  }
child {node [notround] (f) {no homog. model}
};
}
child {node [notround] (e) {
6-dim. symmetry (\ref{J=0L=0M=0P=0Qneq0})}
};
  }
child {node [shorter] (d) {dim. < 6}
};
  }
child {node [round] (c) {$P$}
child[grow = down,level distance=2.3cm]{node[round] (q2){$Q$}
child{node[notround](submax){submaximal symmetry (\ref{submaxsyst})}
}
child{node[notround](k){no homog. model}
};
}
child
[grow = right,level distance=3.5cm]
{node [longer] (h){
5-dim. symmetry (\ref{J=0L=0Mneq0Pneq0})}
}
};
  }
child {node [notround] (b) {
5-dim. symmetry (\ref{J=0Lneq0}) }
};
}
child {node [longer] (a) {
6-dim. symmetry (\ref{Jnotzerosyst})}
};
 \path (j) edge node[anchor=south, above] {$= 0\quad\quad$} (l);
  \path (j) edge node[anchor=north, above] {$\quad\quad\neq 0$} (a);
    \path (l) edge node[anchor=south, above] {$= 0\quad\quad$} (m);
  \path (l) edge node[anchor=north, above] {$\quad\quad\neq 0$} (b);
     \path (m) edge node[anchor=south, above] {$= 0\quad\quad$} (p);
  \path (m) edge node[anchor=north, above] {$\quad\quad\neq 0$} (c);
     \path (p) edge node[anchor=south, above] {$= 0\quad\quad$} (q);
  \path (p) edge node[anchor=north, above] {$\quad\quad\neq 0$} (d);
\path (q) edge node[anchor=south, above] {$= 0\quad\quad$} (r);
  \path (q) edge node[anchor=north, above] {$\quad\quad\neq 0$} (e);
   \path (r) edge node[anchor=south, above] {$= 0\quad\quad$} (u);
  \path (r) edge node[anchor=north, above] {$\quad\quad\neq 0$} (f);
    \path (u) edge node[anchor=south, above] {$= 0\quad\quad$} (max);
  \path (u) edge node[anchor=north, above] {$\quad\quad\neq 0$} (g);
   \path (c) edge node[anchor=south, above] {$= 0\quad\quad$} (q2);
  \path (c) edge node[anchor=north, above] {$\neq 0$} (h);
   \path (q2) edge node[anchor=south, above] {$= 0\quad\quad$} (submax);
  \path (q2) edge node[anchor=north, above] {$\quad\quad\neq 0$} (k);
\end{tikzpicture}

\end{table}

We summarize the main results of this section in the following theorem:
 \begin{theorem}\label{thmsummary}
 \begin{itemize}

 \

 \item Up to local equivalence, there is a unique homogeneous marked contact Engel structure with $9$-dimensional  infinitesimal symmetry algebra. The infinitesimal symmetry algebra is isomorphic to $\mathfrak{p}_1$. The structure is characterized by $$J=L=M=P=Q=R=S=0.$$

\item Up to local equivalence, there are precisely two homogeneous marked contact Engel structures with $8$-dimensional infinitesimal symmetry algebra. The infinitesimal symmetry algebras are isomorphic to $\mathfrak{sl}(3,\mathbb{R})$ and $\mathfrak{su}(1,2)$, respectively.
The structures are characterized by  $$J=L=P=Q=0\quad \mbox{and}\quad M\neq0.$$

\item There are no homogeneous marked contact Engel structure with $7$-dimensional infinitesimal symmetry algebra.

\item Up to local equivalence, there are precisely two homogeneous marked contact Engel structures with $6$-dimensional  infinitesimal symmetry algebras. The respective Maurer-Cartan equations are given in \eqref{Jnotzerosyst} and \eqref{J=0L=0M=0P=0Qneq0}; the second symmetry algebra is isomorphic to $\mathfrak{sl}(2,\mathbb{R})\oplus\mathfrak{sl}(2,\mathbb{R})$.

\item There are examples of homogeneous marked contact Engel structures with $5$-dimensional  infinitesimal symmetry algebra, whose Maurer-Cartan equations are given in \eqref{J=0Lneq0} and \eqref{J=0L=0Mneq0Pneq0}.
\end{itemize}
 \end{theorem}
There may be  other, locally non-equivalent homogeneous marked contact Engel structures with $5$-dimensional symmetry algebra as well.

\section{Geometric characterizations of certain branches of  marked contact Engel structures}\label{sec_123}

In this section we geometrically interpret some of the invariant conditions on marked contact Engel structures from Theorem \ref{main}.
Namely, we shall see how the first three of these conditions can be understood as  properties of the filtration  $$\ell^{\sigma}\subset\mathcal{D}^{\sigma}\subset \mathcal{H}^{\sigma}\subset\mathcal{C}\subset T\M$$ from Proposition \ref{propfilt} associated with a marked contact Engel structure. We have already shown that the first condition for Theorem \ref{main}, $J=0$, is equivalent to the integrability of the rank two distribution $\mathcal{D}^{\sigma}$. In Section \ref{equivalent} we show that locally only two cases can occur:   either $\mathcal{D}^{\sigma}$ is indeed integrable, or $\mathcal{D}^{\sigma}$ is $(2,3,5)$ (see Definition \ref{235} below).
We further characterize the integrability of $\mathcal{D}^{\sigma}$ in terms of special properties, introduced in Section \ref{type}, of the line field $\ell^{\sigma}$.
Moreover, in Section \ref{morecond},  we show how to characterize, in the integrable case, further geometric conditions starting from $\ell^{\sigma}$.

\subsection{Various types of vector fields inside a contact distribution}\label{type}

Let $\M$ be a manifold with a distribution  $\mathcal{D}\subset T\M$. Taking Lie brackets of sections of $\mathcal{D}$ defines a filtration of the tangent bundle of $\M$,   called the (weak) derived flag $\mathcal{D}\subset\mathcal{D}'\subset\mathcal{D}''\subset\dots$ of $\mathcal{D}$, where
$$
\mathcal{D}'_x:=[\mathcal{D},\mathcal{D}]_x=\mathrm{Span}\{\xi_x, [\xi,\eta]_x \,:\, \xi, \eta \in\Gamma(\mathcal{D}) \},\,\,\mathcal{D}''_x:=[\mathcal{D},\mathcal{D}']_x=\mathrm{Span}\{\xi_x, [\xi,\eta]_x \,:\, \xi\in\Gamma(\mathcal{D}'),\,\eta\in\Gamma(\mathcal{D})  \}
$$ and so on.   The sequence $(\mathrm{dim}(\mathcal{D}_x), \mathrm{dim}(\mathcal{D}'_x),\dots)$ is called the growth (vector) at a point $x\in \M$.

\begin{definition}\label{235}
We say that a distribution  $\mathcal{D}$ on a $5$-dimensional manifold $\M$ is $(2,3,5)$ if its growth vector is $(2,3,5)$, i.e.,  $\mathrm{dim}(\mathcal{D}_x)=2$, $\mathrm{dim}(\mathcal{D}'_x)=3$ and $\mathrm{dim}(\mathcal{D}''_x)=5$ at all points $x\in\M$.
\end{definition}

Let us focus on a $5$-dimensional contact manifold $(\M,\mathcal{C})$, where, locally, $\mathcal{C}=\ker\theta$ for a contact form $\theta$. Let $\mathcal{D}\subset\mathcal{C}$ be a $2$-dimensional Legendrian distribution. The growth of $\mathcal{D}$ is strictly related to the notion of \emph{type} (see \cite{GianniDGA,GianniAnnales} for more details) of a vector field inside $\mathcal{C}$  defined below.
\begin{definition}
The \emph{type} of a vector field $Y\in\Gamma(\mathcal{C})$ is the rank of the following system:
$$
\theta\,,\,\, \mathcal{L}_Y(\theta)\,,\,\, \mathcal{L}_Y^2(\theta)\,,\,\,\mathcal{L}_Y^3(\theta)\,.
$$
\end{definition}
Note that, due to the complete non-integrability of the contact distribution, one cannot have vector field of type $1$. Note also that the type depends neither on the choice of $\theta$ nor on the length of $Y$, i.e. it is well defined the type of a line distribution contained in the contact distribution $\mathcal{C}$. By choosing a contact form $\theta$,
any $1$-form $\alpha$ on $\M$ determines a vector field $Y_\alpha$ lying in the contact distribution by the relations
$$
\mathcal{L}_{Y_\alpha}(\theta) = d\theta(Y_\alpha,\,\cdot\,) = \alpha - \alpha(Z)\theta\,,\quad \theta(Y_\alpha) = 0\,,
$$
where $Z$ is the Reeb vector field associated with  $\theta$. Although $Y_\alpha$ depends on the choice of $\theta$, its direction does not.
In the case $\alpha=df$ where $f\in C^\infty(\M)$, we simply write $Y_f$ instead of $Y_{df}$ and it will be called the \emph{Hamiltonian vector field associated with $f$}. Hamiltonian vector fields are a special kind of vector field of type 2.
 We quote the following propositions, whose proofs are contained in \cite{GianniDGA}.  We shall use them in Sections \ref{equivalent} and \ref{morecond} for a geometrical interpretation of some invariants of a marked contact Engel structure.
\begin{proposition}[\cite{GianniDGA}]\label{prop.equiv.type}
The following statements are equivalent.
\begin{enumerate}
\item The vector field $Y\in\mathcal{C}$ is of type $2$.
\item $Y$ is a characteristic symmetry of the distribution $Y^\perp$.
\item the derived distribution $(Y{^\perp})'$ has dimension 4.
\end{enumerate}
\end{proposition}

\begin{proposition}[\cite{GianniDGA}]\label{prop.Hamiltionian.4.dim}
$Y$ is a multiple of a hamiltonian field $Y_f$ if and only if $(Y^\perp)'$ is $4$--dimensional and integrable.
\end{proposition}

\subsection{Equivalent descriptions of Integrability}\label{equivalent}
In this section we shall use the notions introduced in Section \ref{type} to provide equivalent descriptions of integrable marked contact Engel structures.

First, using the coordinate description \eqref{oscfilt} of the osculating filtration $\ell^{\sigma}\subset\mathcal{D}^{\sigma}\subset \mathcal{H}^{\sigma}\subset\mathcal{C}\subset T\M$ it is straightforward to verify the following Proposition.
\begin{proposition}\label{prop.integrability}
\
\begin{enumerate}
\item We always have an inclusion
$(\mathcal{D}^{\sigma})'\subset\mathcal{H}^{\sigma}.$
\item There exists a well-defined invariant map
$$
\Phi_{J}:\Lambda^2\mathcal{D}^{\sigma}\to\mathcal{H}^{\sigma}/\mathcal{D}^{\sigma},\quad \xi_x\wedge\eta_x\mapsto [\xi,\eta]_x\quad\mathrm{mod}\,\mathcal{D}^{\sigma}.
$$
whose vanishing is equivalent to  integrability of the distribution $\mathcal{D}^{\sigma}$.
\item In the parametrization \eqref{section}, integrability of $\mathcal{D}^{\sigma}$ is equivalent to
\begin{equation}\label{J=0eq}
J=-\xi_4(t)=(x^1+3tx^2)t_{x^0}+t^3 t_{x^1} - t^2 t_{x^2} + t t_{x^3} - t_{x^4}=0.
\end{equation}
\end{enumerate}
\end{proposition}


\begin{proposition}\label{prop.integr.2.3.5}
The distribution $\D^\sigma$ is either integrable or of $(2,3,5)$--type.
\end{proposition}
\begin{proof}
Let us assume $\D^\sigma$ non-integrable. Then
\begin{equation}\label{eq.derivato.secondo}
{\D^\sigma}''=\langle  \xi_4\,, \xi_3\,, [\xi_4,\xi_3]\,, [\xi_4,[\xi_4,\xi_3]]  \,, [\xi_3,[\xi_4,\xi_3]] \rangle\,.
\end{equation}
The dimension of ${\D^\sigma}''$ is less than $5$ if and only if the determinant of the $5\times 5$ matrix formed by the components of vector fields of \eqref{eq.derivato.secondo} is zero. Such condition is $\xi_4(t)=0$, that in view of Proposition \ref{prop.integrability} implies the integrability of $\D^\sigma$, that contradicts our initial hypothesis.
\end{proof}

Recall that we denote by $\mathcal{H}^{\sigma}=(\ell^{\sigma})^{\perp}$ the symplectic orthogonal to $\ell^{\sigma}\subset\mathcal{C}$.

\begin{proposition}\label{prop.J}
The following statements are equivalent:
\begin{enumerate}
\item $J=0$.
\item $\D^\sigma$ is integrable.
\item $\dim\left({\mathcal{H}^{\sigma}}'\right)=4$.
\item Any vector field in $\ell^{\sigma}$ is of type $2$.
\item Any vector field in $\ell^{\sigma}$ is a characteristic symmetry of the distribution $\mathcal{H}^{\sigma}$.
\end{enumerate}
\end{proposition}
\begin{proof}
The equivalence between point $1$ and $2$ has already been proven in Proposition \ref{prop.integrability}.

\smallskip\noindent
$2$ implies $3$. In general, since $\mathcal{H}^{\sigma}=\langle \xi_4,\xi_3,\xi_2\rangle$, we have ${\mathcal{H}^{\sigma}}'=\langle \xi_4,\xi_3,\xi_2,[\xi_4,\xi_3],[\xi_4,\xi_2],[\xi_3,\xi_2] \rangle$.
If $\D^\sigma=\langle \xi_4,\xi_3 \rangle$ is integrable,  then  ${\mathcal{H}^{\sigma}}'$ is spanned by  $\xi_4,\xi_3,\xi_2,[\xi_4,\xi_2],[\xi_3,\xi_2]$, and a direct calculation shows that the condition that it has rank equal to $4$ is precisely $J=0$.

\smallskip\noindent
$3$ implies $2$. By contradiction, let us suppose that $\D^\sigma$ is not integrable. Then ${{\D^\sigma}'}^\perp=\ell^\sigma$ that implies ${{\D^\sigma}''}={\mathcal{H}^{\sigma}}'$, that in view of Proposition \ref{prop.integr.2.3.5} is $5$--dimensional, a contradiction.

\smallskip\noindent
$3$, $4$ and $5$ are equivalent because of Proposition \ref{prop.equiv.type}.
\end{proof}

\begin{remark}
Proposition \ref{prop.J} shows that for integrable marked contact Engel structures, the filtration $$\mathcal{D}^{\sigma}\subset\mathcal{H}^{\sigma}\subset{\mathcal{H}^{\sigma}}'\subset T\M$$ is preserved under the Lie derivative of any vector field contained in $\ell^{\sigma}$. In particular, it descends to a filtration on the local leaf space of the foliation determined by $\ell^{\sigma}$.
\end{remark}

%

\subsection{Two more  conditions on integrable marked contact Engel structures}\label{morecond}

Suppose that $J=0$. Then, by Proposition \ref{prop.J}, any vector field in $\ell^{\sigma}$ is a characteristic symmetry of the distribution $\mathcal{H}^{\sigma}$ and consequently also of ${\mathcal{H}^{\sigma}}'$. It follows that, if $J$ vanishes, the Lie bracket of vector fields induces a well defined map
$$\Phi_L:\mathcal{D}^{\sigma}/\ell^{\sigma}\otimes( {\mathcal{H}^{\sigma}}'/\mathcal{H}^{\sigma})\to T\M/{\mathcal{H}^{\sigma}}'.$$
With respect to the frame \eqref{frame1}, the map is determined by a single function. Vanishing of $\Phi_L$ is equivalent to  $L=0$.

\begin{proposition}\label{prop.J.L}
Suppose that $J=0$. The following statements are equivalent:
\begin{enumerate}
\item $L=0$.
\item  Any vector field contained in the distribution $\D^\sigma$ is an internal symmetry of ${\mathcal{H}^{\sigma}}'$.
\end{enumerate}
\end{proposition}

\begin{proof}
Since $\D^\sigma$ is integrable, in view of Proposition \ref{prop.J} the distribution ${\mathcal{H}^{\sigma}}'$ is $4$--dimensional. It is spanned by vectors $\xi_4$, $\xi_3$, $\xi_2$ (that are inside $\mathcal{H}^{\sigma}$) and by an extra vector
$$
[\xi_3,\xi_2]=-3\partial_{x^0} -3(3x^2t_{x^0} -3t^2t_{x^1} +t_{x^3})\partial_{x^1} -2(3tt_{x^1} - t_{x^2})\partial_{x^2}
$$
In view of the integrability of $\D^\sigma$ and in view of the fact that $\xi_4$ is a characteristic symmetry of ${\mathcal{H}^{\sigma}}$ (and then also of ${\mathcal{H}^{\sigma}}'$), see Proposition \ref{prop.equiv.type}, we have that the vector fields in $\D^\sigma$ are symmetries of ${\mathcal{H}^{\sigma}}'$ if and only if $[\xi_3,[\xi_3,\xi_2]]\in {\mathcal{H}^{\sigma}}'$. This is equivalent to $L=0$.
\end{proof}

By Proposition \ref{prop.J.L}, if $J=L=0$, then the Lie bracket of vector fields induces a well-defined map
$$\Phi_M: \Lambda^2{\mathcal{H}^{\sigma}}'/\mathcal{D}^{\sigma}\to T\M/{\mathcal{H}^{\sigma}}'.$$
With respect to the frame \eqref{frame1}, it corresponds to a single function. Vanishing of $\Phi_M$ means precisely that $M=0$.

\begin{proposition}\label{prop.J.L.M}
Suppose that $J=L=0$. The following statements are equivalent:
\begin{enumerate}
\item $M=0$.
\item  The distribution ${\mathcal{H}^{\sigma}}'$ is $4$-dimensional and integrable.
\item The direction $\ell^\sigma$ is a Hamiltonian direction.
\end{enumerate}
\end{proposition}
\begin{proof}
$1$ is equivalent to $2$. In fact, in view of the reasonings contained in the proof of Proposition \ref{prop.J.L}, under our assumption ${\mathcal{H}^{\sigma}}'$ is integrable if and only if $[\xi_2,[\xi_3,\xi_2]]\in {\mathcal{H}^{\sigma}}'$. Recalling that ${\mathcal{H}^{\sigma}}'=\langle \xi_4,\xi_3,\xi_2,[\xi_3,\xi_2] \rangle$ (see Proposition \ref{prop.J.L}), it is straightforward to realize that $[\xi_2,[\xi_3,\xi_2]]\in \langle \xi_4,\xi_3,\xi_2,[\xi_3,\xi_2] \rangle$ if and only if $M=0$.

\smallskip\noindent
$2$ is equivalent to $3$. It follows from Proposition \ref{prop.Hamiltionian.4.dim}.
\end{proof}


%

\section{A Kerr  theorem for contact Engel structures} \label{SecKerr}
In Section \ref{secKerr1} we show how to construct a general integrable marked contact Engel structure.
 We state  this result  in Theorem \ref{Kerr1} in analogy to Penrose's formulation of  Kerr's  theorem from relativity.  In Section \ref{secKerr2} we give a twistorial interpretation of the result.   We show that integrable marked contact Engel structures are in local 1-1 correspondence with generic  hypersurfaces in the twistor space $\mathrm{G}_2/\mathrm{P}_1$, see Corollary \ref{Kerr2}. Via this correspondence, highly symmetric integrable marked contact Engel structures correspond to highly symmetric hypersurfaces of $\mathrm{G}_2/\mathrm{P}_1$. We use this correspondence to give a description of the maximal and submaximal  models, having symmetry algebras $\mathfrak{p}_1$, $\mathfrak{sl}(3,\mathbb{R})$ and $\mathfrak{su}(1,2)$, respectively, in Section \ref{maxsubmax}. Moreover, we investigate the geometric structures hypersurfaces in $\mathrm{G}_2/\mathrm{P}_1$ inherit from the geometry of the ambient space.

\subsection{Local description of integrable marked contact Engel structures: the Kerr theorem} \label{secKerr1}
In this section we show how to find the general solution to the
non-linear PDE
\begin{equation}\label{eq_Jvanishes} J=(x^1+3tx^2)t_{x^0}+t^3 t_{x^1} - t^2 t_{x^2} + t t_{x^3} -t_{x^4}=0.
\end{equation}
This is analogous to a result from relativity attributed to Kerr, see e.g.
 \cite{penroserindler2, Tafel}. We thus refer to it as a Kerr theorem for Engel structures\footnote{We state our theorem in parallel to Penrose's formulation of the original Kerr theorem, as in \cite[Theorem 7.4.8]{penroserindler2}.}.

%

\begin{theorem}[Kerr theorem for contact Engel structures]\label{Kerr1}
The general smooth solution to the equation \eqref{eq_Jvanishes} is obtainable locally by choosing an arbitrary smooth function $F$ of five variables and solving the equation
$$F(x^0+x^1x^4+3tx^2x^4-t^3(x^4)^2, x^1+t^3x^4, x^2-t^2x^4, x^3+t x^4, t)=0$$
for $t$ in terms of $x^0, x^1, x^2, x^3, x^4$.
\end{theorem}

\begin{proof}
We introduce the following variables
\begin{equation}\label{new_var}\begin{aligned}
y^0=x^0+x^1x^4+3t x^2x^4-t^3(x^4)^2, \quad y^1=x^1+t^3x^4,\quad y^2=x^2-t^2x^4,\quad y^3=x^3+tx^4.
\end{aligned}
\end{equation}
As in the proof of Proposition \ref{prop.integrability} one sees that $\der\omega^0\wedge\omega^0\wedge\omega^1\wedge\omega^2=0,$  $ \der\omega^1\wedge\omega^0\wedge\omega^1\wedge\omega^2=0$ and in the new variables we have
\begin{equation*}
\der\omega^2\wedge\omega^0\wedge\omega^1\wedge\omega^2= -2 \; \der t\wedge\der y^0\wedge\der y^1\wedge\der y^2\wedge\der y^3.
\end{equation*}
The latter expression vanishes  if and only if there exists a smooth function $F$ of five variables such that $F(t, y^0,y^1,y^2,y^3)=0$.
 On the other hand, the proof of Proposition \ref{prop.integrability} shows that vanishing of $\der\omega^2\wedge\omega^0\wedge\omega^1\wedge\omega^2$
   is equivalent to $J=0$.
%
\end{proof}

\begin{example}
To give an example how Theorem \ref{Kerr1} works, we consider $F(t,y^0,y^1,y^2,y^3)=t-\frac{s y^3 -y^1}{y^2}$, where $s$ is an arbitrary constant.  Then we find $t$ as a function of $x^0,x^1,x^2,x^3,x^4$ from $$t=\frac{s y^3 -y^1}{y^2}=\frac{sx^3+tsx^4-x^1-t^3x^4}{x^2-t^2x^4}.$$
This gives
$$t=\frac{x^1-sx^3}{-x^2+sx^4},$$
and one can check by a direct calculation that it satisfies \eqref{eq_Jvanishes}.
\end{example}

\begin{remark}
An operational answer to how the variables \eqref{new_var} were obtained is  that we were rewriting the co-frame forms  from \eqref{coframet} as
\begin{align*}
\omega^4&=\der x^4\\
\omega^3&
=\der (x^3+t x^4)-x^4\der t\\
\omega^2&
=\der (x^2- t^2 x^4)+2 t \omega^3+2 t x^4 \der t\\
  \omega^1&
 =\der (x^1+t^3 x^4)-3t^2x^4\der t+3 t \omega^2-3 t^2 \omega^3\\
  \omega^0&
  =\der (x^0+3t x^2x^4+x^1x^4-t^3{x^4}^2)-x^4\omega^1 -3 (x^2-t^2x^4)\omega^3-3x^4(x^2-t^2x^4)\der t.
\end{align*}

\end{remark}
%

\subsection{Local coordinates adapted to the $\mathrm{G}_2$ double fibration}\label{double_coord}
In analogy with the classical Kerr Theorem, we also have a  geometrical interpretation of Theorem \ref{Kerr1} in terms of a  twistorial correspondence, which is given in Corollary \ref{Kerr2} in the next section.
Our proof of this correspondence uses local coordinates adapted to the double filtration for $\mathrm{G}_2$ depicted below.


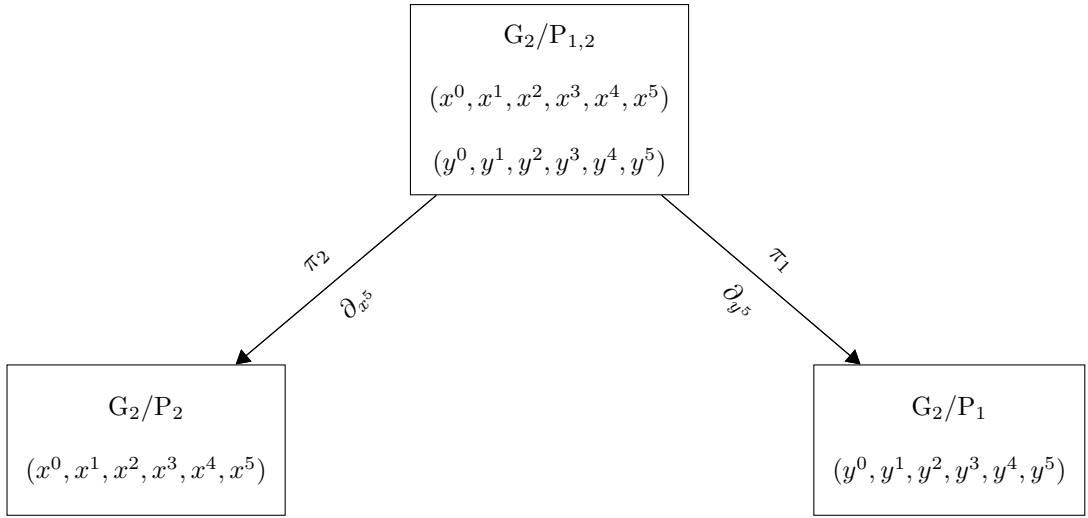
\begin{center}
\begin{tikzpicture}[sharp corners=2pt,inner sep=7pt,node distance=.7cm,every text node part/.style={align=center}]

\node[draw, minimum height = 2.5cm, minimum width = 3cm] (state0){$\mathrm{G}_2/\mathrm{P}_{1,2}$\\\\$(x^0,x^1,x^2,x^3,x^4, x^5)$ \\\\$(y^0,y^1,y^2,y^3, y^4, y^5)$};
\node[below=2cm of state0, minimum height = 2.5cm, minimum width = 3cm](state2){};
\node[draw,left=2cm of state2, minimum height = 2cm, minimum width = 3cm](state1){$\mathrm{G}_2/\mathrm{P}_2$\\\\$(x^0,x^1,x^2,x^3,x^4, x^5)$};
\node[draw,right=2cm of state2, minimum height = 2cm, minimum width = 3cm](state3){$\mathrm{G}_2/\mathrm{P}_1$\\\\$(y^0,y^1,y^2,y^3, y^4, y^5)$};

\draw[-triangle 60] (state0) -- (state1) node [midway, above, rotate = 45]{$\pi_2$};
\draw[-triangle 60] (state0) -- (state1) node [midway, below, rotate = 45]{$\partial_{x^5}$};
\draw[-triangle 60] (state0) -- (state3) node [midway, above, rotate = -45]{$\pi_1$};
\draw[-triangle 60] (state0) -- (state3) node [midway, below, rotate = -45]{$\partial_{y^5}$};
\end{tikzpicture}
\end{center}


Let $(\theta^0,\theta^1,\dots, \theta^{13})$ be the coframe  of left-invariant forms on $\mathrm{G}_2$ corresponding to a basis of $\mathfrak{g}$ as in \eqref{basis_g2}.
This coframe is adapted to the grading of the Lie algebra $\mathfrak{g}$
  in such a way that each leaf of the integrable distribution of the kernel of the eight left-invariant forms $\theta^5, \theta^6, \theta^8, \theta^9, \theta^{10}, \theta^{11}, \theta^{12}, \theta^{13}$ on $\mathrm{G}_2$ from \eqref{MaurerCartan}
 corresponds to a section of $\mathrm{G}_2\to\mathrm{G}_2/\mathrm{P}_{1,2}$. The pullbacks $\omega^0,\omega^1,\omega^2,\omega^3,\omega^4,\omega^7$ of the forms $\theta^0,\theta^1,\theta^2,\theta^3,\theta^4, \theta^7$  to a leaf satisfy
 $$\der\omega^0=\omega^1\dz\omega^4-3\omega^2\dz\omega^3, \quad \der\omega^1=3\omega^2\wedge\omega^7, \quad\der\omega^2=2\omega^3\wedge\omega^7, \quad \der\omega^3=\omega^4\wedge\omega^7, \quad \der\omega^4=0,\quad \der\omega^7=0.$$

We integrate this system in two ways. One yields local coordinates $(x^0,x^1,x^2,x^3,x^4, x^5)$ on $\mathrm{G}_2/\mathrm{P}_{1,2}$ such that
 \begin{equation}\label{omegas}
\begin{aligned}
& \omega^0=\der x^0+x^1\der x^4- 3 x^2\der x^3\\
 &\omega^1=\der x^1+3 x^5\der x^2+ 3 (x^5)^2 \der x^3+(x^5)^3\der x^4\\
 &  \omega^2=\der x^2 +2 {x^5} \der x^3 +(x^5)^2 \der x^4\\
  &  \omega^3=\der x^3+ x^5 \der x^4\\
  &  \omega^4=\der x^4,\\
  &\omega^7=-\der x^5,
 \end{aligned}
 \end{equation}
Denoting  by $\xi_0,\xi_1,\xi_2,\xi_3,\xi_4,\xi_7$ the dual frame, the vertical bundle for $\pi_1$ is  spanned by $$\xi_4=-(x^1+3x^5x^2)\partial_{x^0}-(x^5)^3\partial_{x^1}+(x^5)^2\partial_{x^2}-(x^5)\partial_{x^3}+\partial_{x^4},$$ the vertical bundle for $\pi_2$ is spanned by $$\xi_7=-\partial_{x^5}.$$ We can view $(x^0,x^1,x^2,x^3,x^4)$ as local coordinates on $\mathrm{G}_2/\mathrm{P}_2$, then
$$\pi_2:(x^0,x^1,x^2,x^3,x^4,x^5)\mapsto (x^0,x^1,x^2,x^3,x^4),$$
i.e., $x^5$ is the fibre coordinate for $\pi_2$.

The other way of integrating yields local coordinates  $(y^0,y^1,y^2,y^3,y^4,y^5)$ on $\mathrm{G}_2/\mathrm{P}_{1,2}$ such that
\begin{equation}\label{newomegas}
\begin{aligned}
& \omega^0=\der y^0-y^5\der y^1- 3 y^4 y^5\der y^2-3(y^2+y^5(y^4)^2)\der y^3\\
 &\omega^1=\der y^1+3 y^4\der y^2+ 3 (y^4)^2 \der y^3\\
 &  \omega^2=\der y^2 +2 {y^4} \der y^3 \\
  &  \omega^3=\der y^3- y^5 \der y^4\\
  &  \omega^4=\der y^5,\\
  &\omega^7=-\der y^4.
 \end{aligned}
 \end{equation}
In these coordinates the field $\xi_4$ spanning the vertical bundle for $\pi_1$, is rectified, i.e., we have
$$\xi_4=\partial_{y^5},$$
and
$$\xi_7=-3 y^5 y^2 \partial_{y^0}-3 (y^4)^2 y^5 \partial_{y^1}+2 y^4 y^5 \partial_{y^2}-y^5 \partial_{y^3}-\partial_{y^4}\, .$$
 We can view $(y^0,y^1,y^2,y^3,y^4)$ as coordinates on $\mathrm{G}_2/\mathrm{P}_1$. Then
 $$\pi_1: (y^0,y^1,y^2,y^3,y^4,y^5)\mapsto (y^0,y^1,y^2,y^3,y^4),$$
 i.e., $y^5$ is the fibre coordinate for $\pi_1$.

 A change of coordinates from $(x^0,x^1,x^2,x^3,x^4, x^5)$ to $(y^0,y^1,y^2,y^3, y^4, y^5)$ is given by
\begin{equation}\label{coordchange}
\begin{aligned}
&y^0=x^0+x^1x^4+3x^5 x^2x^4-(x^5)^3(x^4)^2, \quad y^1=x^1+(x^5)^3x^4,\\ & y^2=x^2-(x^5)^2x^4,
\quad y^3=x^3+x^5x^4,\quad y^4=x^5, \quad y^5=x^4.
\end{aligned}
\end{equation}
Similar coordinate transformations can be found e.g. in \cite{machida, ishi}.

\subsection{Geometrical interpretation of the Kerr theorem for contact Engel structures}\label{secKerr2}

Having set up the coordinate systems, the geometrical interpretation  of Theorem \ref{Kerr1}, given in Corollary \ref{Kerr2}, is now almost immediate.

\begin{corollary}\label{Kerr2} Consider the double fibration
\begin{align}\label{doubfibcor}
   \xymatrix{
        &\mathrm{G}_2/\mathrm{P}_{1,2}  \ar[dl]_{\pi_2}^{\xi_4} \ar[dr]^{\pi_1}_{\xi_7} & \\
                   \mathrm{G}_2/\mathrm{P}_2 &            & \mathrm{G}_2/\mathrm{P}_1 .}
\end{align}
There is a local bijective correspondence between  integrable sections of $\pi_2$ and hypersurfaces $\Sigma\subset\mathrm{G}_2/\mathrm{P}_1$ which are generic in the sense that their preimages ${\pi_1}^{-1}(\Sigma)$ intersect the fibres ${\pi_2}^{-1}(x)$
 transversally.

\end{corollary}

\begin{proof}
 Any local section $\sigma:\mathcal{U}\to\mathrm{G}_2/\mathrm{P}_{1,2}$, with\, $\mathcal{U}\subset \mathrm{G}_2/\mathrm{P}_2,$
 defines a hypersurface in  $\mathrm{G}_2/\mathrm{P}_{1,2}$ locally  given in terms of coordinates $(x^0,x^1,x^2,x^3,x^4,x^5)$ by its graph $$x^5=t(x^0,x^1,x^2,x^3,x^4).$$ By Proposition \ref{prop.integrability}, the integrability condition reads
 $$0=-(x^1+3tx^2)t_{x^0}-t^3 t_{x^1}+t^2 t_{x^2}-t t_{x^3}+ t_{x^4}=\xi_4(t)\vert_{\sigma (\mathcal{U})}.$$
 Since $\xi_4$ spans the vertical bundle of $\pi_1$, this means that $\sigma(\mathcal{U})$ is tangential to the fibres of $\pi_1$, which implies that $\sigma$ defines a hypersurface in $\mathrm{G}_2/\mathrm{P}_1$.

Conversely, let $\Sigma$ be a hypersurface in $\mathrm{G}_2/\mathrm{P}_1$ such that $\pi_1^{-1}(\Sigma)$ is transversal to the fibres of $\pi_2$.
  Because of this genericity assumption on $\Sigma$, we may apply the implicit function theorem and write $\pi_1^{-1}(\Sigma)$, locally, as the graph of a section $x^5=t(x^0,x^1,x^2,x^3,x^4)$. By construction $\xi_4\cdot t\vert_{\sigma (\mathcal{U})}=0$, that is,  the section is integrable.
\end{proof}

We conclude this section with a number of remarks, each of which deserves further investigations.  Recall that a  marked contact Engel structure can be viewed as a (local) foliation  of $\mathrm{G}_2/\mathrm{P}_2$ by unparametrized curves whose tangent directions are contained in  $\gamma\subset\mathbb{P}(\mathcal{C})$. We called such a foliation a $\gamma$-congruence in Proposition \ref{prop_punctured}. Note that $\Sigma$ appearing in the Corollary \ref{Kerr2} can be locally  identified with its leaf space.

 \begin{remark}{(\textbf{On geodesics for Weyl connections})} \
 For contact twisted cubic structures,   there exists a class of distinguished connections on the tangent bundle preserving the geometric structure, which are known as \emph{Weyl connections}.
 A choice of contact form uniquely determines a connection from the class of Weyl connections. It is an algebraic computation to determine how a Weyl connection transforms under a change of contact form, see  \cite[Proposition 5.1.6]{book}. In particular, using the transformation formula, it is straightforward to verify that if an unparametrised curve whose tangent directions are contained in  $\gamma\subset\mathbb{P}(\mathcal{C})$ is a geodesic for one Weyl connection, i.e., $\nabla_{c'}c'\propto c',$ then it is a geodesic for any other Weyl connection as well. We shall call these curves $\gamma$-geodesics.
 In the case of the flat model, i.e., the  contact Engel structure, the $\gamma$-geodesics are then just curves of the form $g\,\mathrm{exp}(tX)\cdot o\subset\mathrm{G}_2/\mathrm{P}_2$ with $X$ an element in the highest weight orbit of $G_0$ on $\mathfrak{g}_{-1}$.

Returning to the coordinate representation \eqref{section} of marked contact Engel structures, here the  Weyl connection $\nabla$ determined by the contact form $\alpha^0$ is such that it preserves the coframe $(\alpha^0,\alpha^1,\alpha^2,\alpha^3,\alpha^4)$ in all horizontal directions, i.e., $\nabla_{X}\alpha^i=0$ for all $X\in\Gamma(\mathcal{C})$.
In terms of this Weyl connection,
$\nabla_{\xi_4}\xi_4=-t_{\omega^4}\xi_3=J\xi_3,$ where $\ell^{\sigma}=\mathrm{Span}(\xi_4)$.
Hence,  the condition that  a  $\gamma$-congruence consists entirely of $\gamma$-geodesics  is precisely  the integrability condition $J=0$. (Note that this means  that the relative invariant $J$ is an obstruction against the existence of a Weyl connection that preserves the marked contact Engel structure.)

\end{remark}
There are further viewpoints on the $\mathrm{G}_2$-correspondence discussed here and results that should be useful in this context, we refer e.g.  to \cite{BryantCartan, BryantHsu, DoubZel,  machida, ishi, Ishiusw2}.

Our next remarks concern the geometric structures that a hypersurface $\Sigma\subset\mathrm{G}_2/\mathrm{P}_1$ inherits from the ambient geometry on $\mathrm{G}_2/\mathrm{P}_1$.
The $\mathrm{G}_2$-homogeneous space $\mathrm{G}_2/\mathrm{P}_1$ is equipped with a  $\mathrm{G}_2$-invariant $(2,3,5)$ distribution $\mathcal{D}^{(2,3,5)}$ (see Definition \ref{235}), first discovered  by Cartan and Engel \cite{cartan, engel}. Taking the pullback of the $1$-forms $\omega^0, \omega^1, \omega^2, \omega^3, \omega^7$ on $\mathrm{G}_2/\mathrm{P}_{1,2}$  as in \eqref{newomegas} by any section of $\pi_1:\mathrm{G}_2/\mathrm{P}_{1,2}\to\mathrm{G}_2/\mathrm{P}_1$  defines a co-frame on $\mathrm{G}_2/\mathrm{P}_1$. This coframe is adapted to the $\mathrm{G}_2$-invariant  $(2,3,5)$-distribution $\mathcal{D}^{(2,3,5)}$ in the sense that
$$\mathcal{D}^{(2,3,5)}=\mathrm{ker}(\omega^0,\omega^1,\omega^2),$$
with derived rank $3$ distribution
$$(\mathcal{D}^{(2,3,5)})'=[\mathcal{D}^{(2,3,5)},\mathcal{D}^{(2,3,5)}]=\mathrm{ker}(\omega^0,\omega^1).$$

\begin{remark}{\textbf{(On 3rd order ODEs 1)}}\
Consider the section of $\,\pi_2:\mathrm{G}_2/\mathrm{P}_{1,2}\to\mathrm{G}_2/\mathrm{P}_1$ corresponding to $y^5=0$, rename the coordinates as usual jet coordinates as follows
$$y^0=y,\, y^1=z,\, y^2=y',\,  3y^3=x,\, -\tfrac{2}{3}y^4=y'',$$ and change the co-frame by an admissible transformation
 (in other words, we are putting it  into Goursat normal form):
\begin{equation*}
\begin{aligned}
&\hat{ \omega}^0=\omega^0=\der y-y'\der x\\
 &\hat{\omega}^1=\omega^1-3y^4\omega^2=\der z -\tfrac{9}{4} (y'')^2\der x \\
 &  \hat{\omega}^2   =\der y'-y'' \der x\\
  & \hat{ \omega}^3= 3\omega^3=\der x\\
   &\hat{\omega}^7=\tfrac{3}{2}\omega^7=\der y''.
 \end{aligned}
 \end{equation*}
This shows that integral curves $c(x)=(x,y(x),y'(x),y''(x),z(x))$ of the distribution $\mathrm{ker}(\omega^0,\omega^1,\omega^2)$ are solutions to the Hilbert-Cartan equations $z'=\frac{9}{4}{y''}^2$.

Now consider a hypersurface $\Sigma\subset \mathrm{G}_2/\mathrm{P}_1$ given as as $H(x,y,y',y'',z)=0$. Differentiating and inserting the Hilbert-Cartan equation,  we get an explicit third order ODE on $y=y(x)$,
$$y'''=-\frac{1}{H_{y''}}(\tfrac{9}{4}{y''}^2+H_x-H_yy'-H_{y'}y'').$$
\end{remark}

\begin{remark}{\textbf{(On 3rd order ODEs 2)}} \
Here we take another viewpoint.
 Recall that a distribution with growth vector $(2,3,4)$ is called an \emph{Engel distribution} (see e.g. \cite{BryantHsu, BryGovEasNeu}). It is well known that the derived rank $3$ distribution of an  Engel distribution admits a unique line field spanned by a  characteristic symmetry contained in the Engel distribution. We refer to it as the \emph{characteristic line field}.   More precisely, there exist local coordinates $(x,y,y',y'')$ such that the Engel distribution is generated by
$$\tfrac{\mathrm{d}\,}{\mathrm{d} x}=\partial_x+y'\partial_y+y''\partial_{y'}, \quad \partial_{y''},$$
where $\partial_{y''}$ spans the characteristic line field. Any line field transversal to $\partial_{y''}$ is generated by $\mathrm{D}=\tfrac{\mathrm{d}\,}{\mathrm{d} x}+F\,\partial_{y''}$, for some smooth function $F$, to which is associated the third order ODE
$$y'''=F(x,y,y',y'').$$ The geometry consisting of an Engel disribution together with a transversal line field is itself a parabolic geometry, modeled on  $\mathrm{Sp}(4,\mathbb{R})/P,$ where $P$ is the Borel subgroup \cite{thirdorder, BryGovEasNeu}.

Now let $\Sigma$ be a hypersurface in $\mathrm{G}_2/\mathrm{P}_1$. One verifies that in terms of the geometry on $\mathrm{G}_2/\mathrm{P}_1$, the genericity condition of Corollary \ref{Kerr2}, namely,  that $\pi_1^{-1}(\Sigma)$ be transversal to fibres of $\pi_2$,
can  be rephrased as the condition that  at each point $p\in\Sigma$ the tangent space of $\Sigma$ and the $(2,3,5)$-distribution $\mathcal{D}^{(2,3,5)}$ intersect  in a line. In particular, this yields a line distribution $\mathcal{L}^{\Sigma}\subset T\Sigma$ on $\Sigma$ (and $\Sigma$ is  thus foliated by integral curves).
Likewise, the rank three distribution $(\mathcal{D}^{(2,3,5)})'$ on $\mathrm{G}_1/\mathrm{P}_1$ gives rise to a rank two distribution $\mathcal{H}^{\Sigma}\subset T \Sigma$.

It turns out that  distribution $\mathcal{H}^{\Sigma}$ is maximally non-integrable, i.e., it is an Engel distribution,
if and only if an additional genericity condition on the hypersurface $\Sigma$ is satisfied. Computing shows that this condition is equivalent  to  $L\neq 0$ as in Theorem \ref{main}. Suppose that $L\neq 0$ and
let $\mathcal{K}^{\Sigma}\subset\mathcal{H}^{\Sigma}$ be the characteristic line field of the Engel distribution $\mathcal{H}^{\Sigma}$.
Then one further verifies that the fields $\mathcal{K}^{\Sigma}$ and  $\mathcal{L}^{\Sigma}$ are linearly independent, and  thus one has a direct sum decomposition $\mathcal{H}^{\Sigma}=\mathcal{K}^{\Sigma}\oplus\mathcal{L}^{\Sigma}$. By the above discussion, this equips $\Sigma$ with the structure of a third order ODE (considered modulo contact transformations), or equivalently, a parabolic geometry modeled on $\mathrm{Sp}(4,\mathbb{R})/P$.

\end{remark}

\begin{remark}{\textbf{(On the induced conformal structures)}}\
For our final remark, we recall that $\mathrm{G}_2/\mathrm{P}_1$ carries a $\mathrm{G}_2$-invariant conformal class of metrics $[g]$ of signature $(2,3)$, with respect to which $\mathcal{D}^{(2,3,5)}$ is totally null, see \cite{nurowskiconf}. When $\mathrm{G}_2/\mathrm{P}_1$ is identified with the projectivized null cone $\mathbb{P}(\mathcal{N})=\{[X]\in\mathbb{R}^{3,4}: h(X,X)=0\}$, then this conformal structure is induced from the $\mathrm{G}_2$-invariant metric $h$ on $\mathbb{R}^{3,4}$.

One can   pullback the $\mathrm{G}_2$-invariant conformal class $[g]$ to the hypersurface $\Sigma\subset \mathrm{G}_2/\mathrm{P}_1$,
which yields
an induced  \emph{non-degenerate} conformal structure on $\Sigma$ if and only if the relative invariant $M-P$ as in Proposition \ref{prop_seconddi}  is non-vanishing.
\end{remark}

\subsection{Maximal and  submaximal models for marked contact Engel structures revisited}\label{maxsubmax}

We shall use the correspondence between integrable marked contact Engel structures and hypersurfaces in the twistor space to describe  the maximal and submaximal models derived in Section \ref{CartanEquiv}.

 Let $\Phi\in\Lambda^3(\mathbb{R}^{3,4})^*$ be the defining three form of the group $\mathrm{G}_2$ and let  $h\in\bigodot^2(\mathbb{R}^{3,4})^*$ be the $\mathrm{G}_2$-invariant bilinear form of signature $(3,4)$. Then homogeneous spaces occurring in the double fibration \eqref{doublefib} admit the following descriptions (see e.g. \cite{BryantCartan, machida, SplitOct}):
 \begin{itemize}
 \item $\mathrm{G}_2/\mathrm{P}_1$ can be identified with the projectivized null cone $\mathbb{P}(\mathcal{N})$ of all $1$-dimensional subspaces $\mathbb{L}\subset\mathbb{R}^{3,4}$ that are null with respect to $h$,
 \item $\mathrm{G}_2/\mathrm{P}_2$ can be identified with the set of $2$-dimensional totally null subspaces $\Pi\subset\mathbb{R}^{3,4}$ that insert trivially into the defining $3$-form $\Phi$,
 \item $\mathrm{G}_2/\mathrm{P}_{1,2}$ can be identified with the correspondence space of all  pairs $(\mathbb{L}, \Pi)\in\mathrm{G}_2/\mathrm{P}_1\times \mathrm{G}_2/\mathrm{P}_2$, where  $\mathbb{L}\subset \Pi$.
\end{itemize}
 A fibre ${\pi_2}^{-1}(\Pi)$ can be identified with the set of all $1$-dimensional subspaces contained in $\Pi$ and is thus isomorphic to $\mathbb{R}\mathbb{P}^1$. A fibre ${\pi_1}^{-1}(\mathbb{L})$ can be identified with the set of all totally null $2$-dimensional subspaces $\Pi$  that insert trivially into $\Phi$ and contain $\mathbb{L}$; this is the set of $2$-dimensional subspaces of the $3$-dimensional null subspace
\begin{equation*}
\mathrm{Ann}_{\Phi}(\mathbb{L})=\{ X\in \mathbb{R}^{3,4}\mid \Phi (\mathbb{L}, X, \cdot)=0\}\subset  \mathbb{R}^{3,4} ,
\end{equation*}
and hence also isomorphic to $\mathbb{R}\mathbb{P}^1$.

Viewing $\mathrm{G}_2/\mathrm{P}_1=\mathbb{P}(\mathcal{N})$ as a projectivized null cone, the simplest kinds of hypersurfaces in $\mathrm{G}_2/\mathrm{P}_1$ are obtained by intersecting the null cone with a $6$-dimensional subspace $\mathbb{W}\subset\mathbb{R}^{3,4}$ and projectivizing. Such hyperplanes $\mathbb{W}=\mathbb{L}^{\perp}$ split into three classes according to whether its annihilator $\mathbb{L}$ is a  lightlike, timelike or spacelike line.
  It is further known that the group $\mathrm{G}_2$ acts transitively on the set of, respectively,   lightlike, timelike,  spacelike lines $\mathbb{L}\subset \mathbb{R}^{3,4}$ and that
\begin{itemize}
\item $\mathrm{Stab}_{\mathrm{G}_2}(\mathbb{L})=\mathrm{P}_1$ iff $\left\langle\mathbb{L},\mathbb{L}\right\rangle=0$,
\item $\mathrm{Stab}_{\mathrm{G}_2}(\mathbb{L})=\mathrm{SU}(1,2)$ iff $\left\langle\mathbb{L},\mathbb{L}\right\rangle >0$,
\item $\mathrm{Stab}_{\mathrm{G}_2}(\mathbb{L})=\mathrm{SL}(3,\mathbb{R})$ iff $\left\langle\mathbb{L},\mathbb{L}\right\rangle<0$.
\end{itemize}
Each of these groups has a unique open orbit in $\mathbb{P}(\mathcal{N}),$  which is contained in    the space $\mathbb{P}(\mathcal{N}\cap\mathbb{L}^\perp)$,
 see e.g. \cite{meTravis}.

According to Theorem \ref{Kerr2}, there are corresponding marked contact Engel structures, which we can easily describe explicitly:


\begin{proposition} \label{submaxmod}
The subset
\begin{equation}\label{eqSubMaxMod}
\M_{\mathbb{L}}:=\{ \Pi\in \mathrm{G}_2/\mathrm{P}_2\mid   \dim(\Pi\cap\mathbb{L}^\perp)=1\}\subset \mathrm{G}_2/\mathrm{P}_2
\end{equation}
is  equipped with a canonical  $\mathrm{Stab}_{\mathrm{G}_2}(\mathbb{L})$-invariant marked contact Engel structure
\begin{equation}\label{def_sig}
\sigma (\Pi):=( \Pi,\Pi\cap\mathbb{L}^\perp)\in \mathrm{G}_2/\mathrm{P}_{1,2}\, .
\end{equation}

\end{proposition}
Clearly, if we
 fit  \eqref{eqSubMaxMod}  into the double fibration \eqref{doublefib},
then for $\sigma:\M_{\mathbb{L}}\to \pi_2^{-1}(\M_{\mathbb{L}})\subset\mathrm{G}_2/\mathrm{P}_{1,2}$ defined as in \eqref{def_sig}, the corresponding hypersurface $\Sigma_{\mathbb{L}}:=\pi_1(\sigma(\M_{\mathbb{L}}))$ is contained in $\mathbb{P}(\mathcal{N}\cap \mathbb{L}^\perp)$.
\begin{remark}
By looking at the three cases individually we can see that $\Sigma_{\mathbb{L}}$  indeed coincides with the open $\mathrm{Stab}_{\mathrm{G}_2}(\mathbb{L})$-orbit in $\mathbb{P}(\mathcal{N}\cap \mathbb{L}^\perp)$.

If $\left\langle \mathbb{L},\mathbb{L}\right\rangle=0$, the $\mathrm{Stab}_{\mathrm{G}_2}(\mathbb{L})\cong \mathrm{P}_1$ preserves a filtration
$$\mathbb{L}\subset\mathbb{D}\subset\mathbb{D}^{\perp}\subset\mathbb{L}^{\perp}\subset\mathbb{V},$$ where
$\mathbb{D}:=\mathrm{Ann}_{\Phi}(\mathbb{L})=\{ X\in \mathbb{R}^{3,4}\mid \Phi (\mathbb{L}, X, \cdot)=0\}\subset  \mathbb{R}^{3,4}.$
The open $\mathrm{Stab}_{\mathrm{G}_2}(\mathbb{L})$-orbit
consists of all null lines contained in $\mathbb{L}^\perp$ but transversal to $\mathbb{D}^\perp$.
Now suppose that a $2$-plane $\Pi\in \mathrm{G}_2/\mathrm{P}_2$  has non-trivial intersection with $\mathbb{D}^{\perp}$. Then, since $\mathbb{D}$ is maximally isotropic, a null line contained in the intersection has to be already contained in $\mathbb{D}$.
Using the terminology from \cite{BaezHuerta}, this implies that any element $X\in\mathbb{L}$ and any element $Y\in\Pi$ are two rolls away from each other and then Theorem 10 in \cite{BaezHuerta} shows that $\left\langle X,Y\right\rangle=0$, hence $\Pi\subset\mathbb{L}^{\perp}$. This shows that $\Sigma_{\mathbb{L}}$ is contained in the open $\mathrm{Stab}_{\mathrm{G}_2}(\mathbb{L})$-orbit and equality follows from the fact that $\Sigma_{\mathbb{L}}$ is also invariant under the $\mathrm{Stab}_{\mathrm{G}_2}(\mathbb{L})$-action.

If $\left\langle \mathbb{L},\mathbb{L}\right\rangle<0$, we have $\mathrm{Stab}_{\mathrm{G}_2}(\mathbb{L})\cong\mathrm{SL}(3,\mathbb{R})$ which preserves the following decomposition
$$\mathbb{R}^7=\mathbb{L}\oplus\mathbb{L}^{\perp}=\mathbb{L}\oplus\mathbb{U}\oplus\mathbb{U}^*.$$ The group $\mathrm{SL}(3,\mathbb{R})$ acts transitively on $\mathbb{P}\mathbb{U}$, $\mathbb{P}\mathbb{U}^*$ and the open orbit of all null lines in $\mathbb{L}^{\perp}$ that are neither contained in $\mathbb{U}$ nor $\mathbb{U}^*$,
respectively, see \cite{meTravis}.
The open orbit is $\Sigma_{\mathbb{L}}$; this follows from the fact that if a null line $\mathbb{L'}$  is contained in one of the spaces $\mathbb{U}$ or $\mathbb{U}^*$, then its $\Phi$-annihilator $\mathrm{Ann}_{\Phi}(\mathbb{L}')$   is contained in $\mathbb{L}^{\perp}$.

If $\left\langle \mathbb{L},\mathbb{L}\right\rangle>0,$ the group $\mathrm{Stab}_{\mathrm{G}_2}(\mathbb{L})\cong\mathrm{SU}(1,2)$ acts transitively on $\mathbb{P}(\mathcal{N}\cap \mathbb{L}^\perp)$.
\end{remark}

\begin{proposition}\label{submaxprop}
The structures from Proposition \ref{submaxmod} realize maximally symmetric and submaximally symmetric models of marked contact twisted cubic structures.
 Their    infinitesimal symmetry algebras  are $\mathfrak{p}_1$, $\mathfrak{sl}(3,\mathbb{R})$, and $\mathfrak{su}(1,2)$, respectively.
\end{proposition}

\begin{proof}
It is known that   the infinitesimal symmetry algebra of a contact twisted cubic structure is either of dimension  $14$, in which case it is the Lie algebra $\mathfrak{g}$ of $\mathrm{G}_2$,  or else the dimension is $\leq 7$, see \cite{gap}. This implies that if the infinitesimal symmetry algebra of a marked contact twisted cubic structure has dimension $8$ or $9$, then it is a subalgebra of the Lie algebra $\mathfrak{g}$ of $\mathrm{G}_2$ and the underlying contact twisted cubic structure is a contact Engel structure.

By construction, the marked contact Engel structures from Proposition \ref{submaxmod} are invariant under  $\mathfrak{p}_1$, $\mathfrak{sl}(3,\mathbb{R})$ and $\mathfrak{su}(1,2)$, respectively.
It remains to show that the infinitesimal symmetry algebras of these structures are not bigger, but this follows from the fact that $\mathfrak{p}_1$, $\mathfrak{sl}(3,\mathbb{R})$ and $\mathfrak{su}(1,2)$ are maximal subalgebras of $\mathfrak{g}$ \cite{subalg_g2}.
\end{proof}

\begin{remark}
Of course, it follows from the analysis  in Section \ref{CartanEquiv} that, up to local equivalence, the structures from Proposition \ref{submaxmod} are the  unique homogeneous marked contact Engel structures having  infinitesimal symmetry algebras of dimension eight or nine. Alternatively, with a little more work, we could recover this fact from purely algebraic considerations at this point using that we know the subalgebras of   $\mathfrak{g}$.
\end{remark}

\section{Considerations about general marked contact twisted cubic structures}\label{sec_general}
The discussion of this section applies to general marked contact twisted cubic structures, i.e., here we shall \emph{not} restrict our considerations to marked contact Engel structures. We will regard marked contact twisted cubic structures as particular types of filtered $G_0$-structures in this section. For references on the general material used in this section see \cite{tanaka,  Morimoto,  yamaguchi, zelenko, book, CapCartan}.

In Section \ref{sec_tanaka} we review the  (algebraic) Tanaka prolongation and some of its implications. The computation of the Tanaka prolongation implies the existence of a canonical coframe on a $9$-dimensional bundle associated with any marked contact twisted cubic structure in a natural manner.

 In Section \ref{Cartan}, we briefly address the existence question of a canonical Cartan connection for marked contact twisted cubic structures, that is, of a canonical coframe with particularly nice  properties. We show that, for algebraic reasons, the constructions   of canonical Cartan connections from \cite{Morimoto} or \cite{CapCartan} are not applicable to our case.  
 In particular, for the filtered $G_0$-structures we are considering, a normalization condition in the sense of \cite{CapCartan} does not exist.

\subsection{Tanaka prolongation and applications}
 \label{sec_tanaka}
Recall, see Proposition \ref{prop_twisted}, that a contact twisted cubic structure can be equivalently regarded as a contact structure $\mathcal{C}\subset T\M$ together with a reduction of the graded frame bundle $\mathcal{F}\to \M$ with respect to an irreducible representation $\rho:\mathrm{GL}(2,\mathbb{R})\to\mathrm{CSp}(2,\mathbb{R})$. A marked contact twisted cubic structure, see Proposition \ref{prop_punctured}, can be seen as a further reduction of $\mathcal{F}\to \M$ with respect to the restriction $\rho:B\to\mathrm{CSp}(2,\mathbb{R})$ of $\rho$ to the Borel subgroup $B\subset\mathrm{GL}(2,\mathbb{R})$.
In the terminology of \cite{Morimoto, CapCartan}, this means that
\begin{itemize}
\item  a contact twisted cubic structure is a  filtered  $G_0$-structures    of type $\mathfrak{m}$, where $G_0$ is the irreducible $\mathrm{GL}(2,\mathbb{R})$, and
\item a marked contact twisted cubic structure is a  filtered  $Q_0$-structures    of type $\mathfrak{m}$, where $Q_0$ is the Borel subgroup $B\subset \mathrm{GL}(2,\mathbb{R}).$
\end{itemize}
In both cases $\mathfrak{m}=\mathfrak{m}_{-2}\oplus\mathfrak{m}_{-1}$ is the $5$-dimensional Heisenberg Lie algebra.

Now suppose $\mathfrak{m}=\mathfrak{m}_{-k}\oplus\cdots\oplus\mathfrak{m}_{-1}$ is any fundamental graded Lie algebra, where fundamental means that it is generated as a Lie algebra by $\mathfrak{m}_{-1}$.  Let $\mathfrak{g}_0\subset\mathrm{Der}_{gr}(\mathfrak{m})$ be a subalgebra of  the Lie algebra $\mathrm{Der}_{gr}(\mathfrak{m})$ of $\mathrm{Aut}_{gr}(\mathfrak{m})$.
Tanaka introduced the following  algebraic object, which plays a fundamental role in his approach to the equivalence problem of filtered $G_0$-structures.
\begin{proposition}\label{tanakaprop} (\cite{tanaka}) There exists a unique, up to isomorphism, graded Lie algebra $\mathfrak{g}(\mathfrak{m},\mathfrak{g}_0)$,  called the (algebraic) \emph{Tanaka prolongation} of the pair $(\mathfrak{m},\mathfrak{g}_0)$, satisfying the following conditions:
\begin{enumerate}
\item The non-positive part is  $\mathfrak{m}\oplus\mathfrak{g}_0$, i.e., $\mathfrak{g}(\mathfrak{m},\mathfrak{s})_i=\mathfrak{m}_i$ for $i<0$ and $\mathfrak{g}(\mathfrak{m},\mathfrak{g}_0)_0=\mathfrak{g}_0$.
\item If $X\in\mathfrak{g}(\mathfrak{m},\mathfrak{g}_0)_i$ for some $i>0$ satisfies $[X,\mathfrak{m}_{-1}]=\{0\}$, then $X=0$.
\item $\mathfrak{g}(\mathfrak{m},\mathfrak{g}_0)$ is maximal among the graded Lie algebras satisfying (1) and (2).
\end{enumerate}
\end{proposition}
Let $\mathfrak{g}=\bigoplus_{i\in\mathbb{Z}}\mathfrak{g}_i$ be a graded Lie algebra
satisfying (1) and (2) from Proposition \ref{tanakaprop}. The condition that $\mathfrak{g}$ be
the Tanaka prolongation of $(\mathfrak{m},\mathfrak{g}_0)$
can be expressed in terms of the Lie algebra cohomology $H^*(\mathfrak{m},\mathfrak{g})$ with respect to the representation $\mathrm{ad}:\mathfrak{m}\to\mathfrak{gl}(\mathfrak{g})$; this is the cohomology of the cochain complex $(C(\mathfrak{m},\mathfrak{g}),\partial)$ where $C^q(\mathfrak{m},\mathfrak{g}):=\Lambda^q\mathfrak{m}^*\otimes\mathfrak{g}$ and $\partial:C^q(\mathfrak{m},\mathfrak{g})\to C^{q+1}(\mathfrak{m},\mathfrak{g})$ is the standard differential. Note that since $\mathfrak{m}$ and $\mathfrak{g}$ are graded Lie algebras, also the cochain spaces are naturally graded, and  since $\partial$ preserves the homogeneous degree of maps, we have an induced grading on the cohomology spaces.  We shall denote the $l$th grading component by a subscript $l$.  Then (see e.g. \cite{yamaguchi}) the graded Lie algebra    $\mathfrak{g}$ is the prolongation of $(\mathfrak{m},\mathfrak{g}_0)$ if and only if $H^1(\mathfrak{m},\mathfrak{g})_l=0$ for all $l>0$. If $\mathfrak{g}$ is simple, the Lie algebra cohomologies can be computed using Kostant's theorem (see e.g. \cite{book} for an account of Kostant's theorem).

\begin{example}
Let $\mathfrak{g}$ be the Lie algebra of  $\mathrm{G}_2$ equipped with its contact grading
$$\mathfrak{g}=\mathfrak{g}_{-2}\oplus\mathfrak{g}_{-1}\oplus\mathfrak{g}_0\oplus\mathfrak{g}_{1}\oplus\mathfrak{g}_2$$
as discussed in Section \ref{LieTheory}. Then $\mathfrak{m}=\mathfrak{g}_{-2}\oplus\mathfrak{g}_{-1}$ is the $5$-dimensional Lie Heisenberg algebra and,
via the restriction of the adjoint representation, $\mathfrak{g}_0$ is a subalgebra of $\mathrm{Der}_{gr}(\mathfrak{m})$.
 Utilizing Kostant's theorem, one shows that $H^1(\mathfrak{m},\mathfrak{g})_l=0$ for all $l>0$, see \cite{yamaguchi}, and therefore  $\mathfrak{g}$ is the Tanaka prolongation of $(\mathfrak{m},\mathfrak{g}_0)$.
\end{example}

 Let  $\mathfrak{q}_0\subset\mathfrak{g}_0\subset \mathrm{Der}_{gr}(\mathfrak{m})$ be a subalgebra, then the  Tanaka prolongation $\mathfrak{q}=\mathfrak{g}(\mathfrak{m},\mathfrak{q}_0)$ of the pair $(\mathfrak{m},\mathfrak{q}_0)$   is a graded subalgebra of $\mathfrak{g}=\mathfrak{g}(\mathfrak{m},\mathfrak{g}_0)$, where, for positive $i$,
$$\mathfrak{q}_{i}=\{X\in\mathfrak{g}_{i}:\, [X,\mathfrak{g}_{-1}]\subset\mathfrak{q}_{i-1} \}.$$
This immediately leads to the following:
\begin{proposition}
\label{p_1}
Let $\mathfrak{g}=\mathfrak{g}_{-2}\oplus\mathfrak{g}_{-1}\oplus\mathfrak{g}_0\oplus\mathfrak{g}_{1}\oplus\mathfrak{g}_2$  be the Lie algebra of  $\mathrm{G}_2$ equipped with its contact grading, $\mathfrak{m}=\mathfrak{g}_{-2}\oplus\mathfrak{g}_{-1}$ the $5$-dimensional Heisenberg Lie algebra,  and let $\mathfrak{q}_0\subset\mathfrak{g}_0\cong\mathfrak{gl}(2,\mathbb{R})$ be the Borel subalgebra. Then the Tanaka prolongation  $\mathfrak{q}$ of $(\mathfrak{m},\mathfrak{q}_0)$ is a $9$-dimensional   Lie algebra isomorphic to the parabolic subalgebra $\mathfrak{p}_1\subset\mathfrak{g}$.
\end{proposition}

\begin{proof} Let $\mathfrak{q}=\mathfrak{q}_{-2}\oplus\mathfrak{q}_{-1}\oplus\mathfrak{q}_0\oplus\mathfrak{q}_1$ be the subalgebra of $\mathfrak{g}$
spanned by  the Cartan subalgebra and all root spaces corresponding to black nodes in the following root diagram of $\mathrm{G}_2$ :
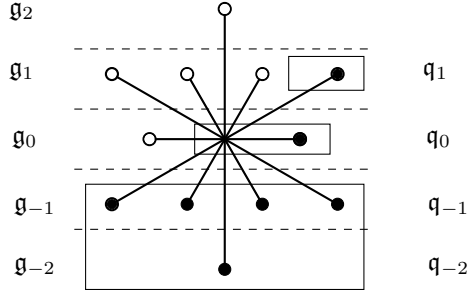
\begin{center}
\begin{tikzpicture}[scale=1]
     \draw (-1.85,  -0.60) rectangle (1.85,-2);
  \draw (-0.4,  -0.2) rectangle (1.4,0.2);
\draw (0.85,  0.65) rectangle (1.85,1.1);

 \draw[thick] (0,-1.732) -- (0,1.732);
 \draw[thick](0,1.732)circle(0.08);
  \filldraw[white](0,1.732)circle(0.06);
 \filldraw[black](0,-1.732)circle(0.08);

 \draw[thick](-1.5,-0.866)--(1.5,0.866);
\draw[thick](-1.5,-0.866)circle(0.08);
\filldraw[black](-1.5,-0.866)circle(0.06);
 \draw[thick](1.5,0.866)circle(0.08);
  \filldraw[black](1.5,0.866)circle(0.06);

 \draw[thick](1.5,-0.866)--(-1.5,0.866);
\filldraw[white](-1.5,0.866)circle(0.06);
\draw[thick](-1.5,0.866)circle(0.08);
 \filldraw(1.5,-0.866)circle(0.08);

\draw[thick](-1,0)--(1,0);
\draw[thick](-1,0)circle(0.08);
\filldraw[white](-1,0)circle(0.06);
 \draw[thick](1,0)circle(0.08);
  \filldraw[black](1,0)circle(0.06);
\draw[thick](1,0)circle(0.08);

 \filldraw[thick](0.5,-0.866)--(-0.5,0.866);
  \filldraw[black](0.5,-0.866)circle(0.08);
 \draw[thick](-0.5,0.866)circle(0.08);
  \filldraw[white](-0.5,0.866)circle(0.06);

  \draw[thick](-0.5,-0.866)--(0.5,0.866);
 \filldraw[black](-0.5,-0.866)circle(0.08);
  \draw[thick](0.5,0.866)circle(0.08);
 \filldraw[white](0.5,0.866)circle(0.06);

  \draw[dashed](-  2, 1.2)--(2, 1.2);
   \draw[dashed](-  2, 0.4)--(2, 0.4);
   \draw[dashed](-  2, -1.2)--(2, -1.2);
   \draw[dashed](-  2, -0.4)--(2, -0.4);

\draw (-2.7,1.7) node {$\mathfrak{g}_{2}$};
\draw (-2.5,-1.7) node {$\mathfrak{g}_{-2\,}$};
\draw (-2.7,0.9) node {$\mathfrak{g}_{1}$};
\draw (-2.5,-0.9) node {$\mathfrak{g}_{-1\,}$};
\draw (-2.6,0) node {$\mathfrak{g}_{0}\ $};

\draw (3,-1.7) node {$\mathfrak{q}_{-2\,}$};
\draw (2.8,0.9) node {$\mathfrak{q}_{1}$};
\draw (3,-0.9) node {$\mathfrak{q}_{-1\,}$};
\draw (2.9,0) node{$\mathfrak{q}_{0}\ $};

 \draw (2.5,1.2) node {};
 \draw (-2.5,-1.2) node {};
 \draw (0.9,1.2) node {};
  \draw (-0.8,-1.2) node {};
 \draw (1.55,0) node {};
  \draw (-1.55,0) node {};
\draw(2.1,-1.2) node {} ;
\draw(-2.1,1.2) node {} ;
 \draw(0.8,-1.2) node {};
  \draw(-0.8,1.2) node {};
 \draw (0,-2.1) node {};
 \draw (0,2.1) node {};
%
%

\end{tikzpicture}
\end{center}
Then $\mathfrak{q}$ is a graded Lie algebra satisfying properties (1) and (2) from Proposition \ref{tanakaprop}. Moreover, there is no proper  subalgebra $\mathfrak{q}'\subset\mathfrak{g}$ containing $\mathfrak{q}$. This can be either deduced  from the above root diagram, by observing that any subalgebra $\mathfrak{q}'$ containing $\mathfrak{q}$ and in addition a root space corresponding to a white root has to be  all of $\mathfrak{g}$. Alternatively, it immediately follows from the fact that a Lie algebra of root type $\mathrm{G}_2$  has no subalgebra of dimension bigger than $9$. Hence property (3) of Proposition \ref{tanakaprop} is satisfied as well.
\end{proof}
\begin{remark}
Identifying $\mathfrak{g}_{-1}\cong S^3\mathbb{R}^2$, the Borel subalgebra $\mathfrak{q}_0\subset\mathfrak{g}_0$ is the stabilizer of a line $\mathrm{Span}(l)\subset\mathbb{R}^2$, equivalently, of a line $\mathrm{Span}(l\odot l\odot l)\subset \smash{\bigodot^3\mathbb{R}^2}$. Recall that $\mathfrak{g}_1=(\mathfrak{g}_{-1})^*$ via the Killing form, and then $\mathfrak{q}_1$ can be  viewed as the annihilator of the $3$-dimensional subspace $\mathrm{Span}(\{X\odot Y\odot l: X,Y\in\mathbb{R}^2\})$ of $\smash{\bigodot^3\mathbb{R}^2}=\mathfrak{g}_{-1}.$
\end{remark}




Given a filtered $G_0$-structure of type $\mathfrak{m}$ such that the Tanaka prolongation of the pair $(\mathfrak{m},\mathfrak{g}_0)$ is finite-dimensional,  Tanaka theory
\begin{itemize}
\item
 provides a procedure to construct, in a natural manner,
 a bundle  $\mathcal{G}\to \M$ of dimension $\mathrm{dim}(\mathfrak{g}(\mathfrak{m},\mathfrak{g}_0))$
 together with a   coframe $\omega$ (an \emph{absolute parallelism}) on  $\mathcal{G}$ (and it predicts the number of prolongation steps to be done to arrive there),
\item
and it establishes $\mathrm{dim}(\mathfrak{g}(\mathfrak{m},\mathfrak{g}_0))$ as a sharp upper bound for the dimension of the  infinitesimal symmetry algebra of the filtered $G_0$-structure.
\end{itemize}

Applied to marked contact twisted cubic structures, as a Corollary to Proposition \ref{p_1}, this yields the following:
\begin{corollary}\label{maximal_sym}

\

\begin{itemize}
\item To any marked contact twisted cubic structure there is a naturally associated $9$-dimensional bundle equipped with a canonical coframe.
\item The dimension of the Lie algebra of infinitesimal symmetries of a marked contact twisted cubic structure is
$\leq 9$.
\end{itemize}
\end{corollary}

\subsection{Canonical Cartan connections and the problem of finding a normalization condition} \label{Cartan}
Given a filtered $G_0$-structure of type $\mathfrak{m}$ with algebraic Tanaka prolongation $\mathfrak{g}=\mathfrak{g}(\mathfrak{m},\mathfrak{g}_0)$, it is a natural question to ask whether   there exists a canonical Cartan connection associated with the structure. This question has been studied in \cite{Morimoto}, where a general criterion  (the ``condition (C)'') ensuring the existence of a canonical Cartan connection is given, and more recently in \cite{CapCartan}, where the essential step to obtaining a canonical Cartan connection is to find a normalization condition with certain algebraic properties.
%

\subsubsection{Cartan geometries}
For a comprehensive introduction to Cartan geometries and applications of the concept see \cite{book}.

Let $G/P$ be a homogeneous space, let $\mathfrak{g}$ be the Lie algebra of $G$ and $\mathfrak{p}$ the Lie algebra of $P$.
A Cartan geometry of type $(\mathfrak{g},P)$  on a manifold $\M$ is a pair $(\mathcal{G} \to \M, \omega)$, where $\mathcal{G} \to \M$ is a $P$-principal bundle and $\omega\in\Omega^1(\mathcal{G},\mathfrak{g})$  a Cartan connection, i.e., a Lie algebra valued $1$-form  satisfying
\begin{enumerate}
\item  $\omega_u : T_u \mathcal{G} \to \mathfrak{g}$ is an isomorphism for all $u \in \mathcal{G}$,
      \item  $\omega(\zeta_X) = X$ for all $X \in \mathfrak{p}$,
	\item $(r^p)^*\omega = \mathrm{Ad}(p^{-1})\omega$,
	\end{enumerate}
where $r^p$ denotes the  right action of $P$ on $\mathcal{G}$ and $\zeta_{X}$ the fundamental vector field generated by $X\in\mathfrak{p}$. The homogeneous (flat) model of a Cartan geometry of type $(\mathfrak{g},P)$ is the principal bundle $G\to G/P$ together with the Maurer-Cartan form $\omega_{MC}$ on $G$. The curvature of a Cartan geometry is the $2$-form $K = d\omega + \frac{1}{2} [\omega, \omega] \in \Omega^2(\mathcal{G},\mathfrak{g}).$ It is equivariant for the principal $P$-action and horizontal, i.e. $K(\zeta_{X},\cdot)=0$ for any $X\in\mathfrak{p}$, which implies that it can be equivalently viewed as an equivariant function $K:\mathcal{G}\to\Lambda^2(\mathfrak{g}/\mathfrak{p})^*\otimes\mathfrak{g}$. The curvature vanishes if and only if the Cartan geometry is locally isomorphic to the homogeneous model; in this  case the Cartan geometry is called flat.


\subsubsection{Normalization conditions}
Given a filtered $G_0$-structure of type $\mathfrak{m}$, let $\mathfrak{g}=\mathfrak{g}(\mathfrak{m},\mathfrak{g}_0)$ be the algebraic Tanaka prolongation. Let $P$ be a Lie group with Lie algebra the non-negative part $\mathfrak{g}^0$ of $\mathfrak{g}$.
Then the curvature function of  any  Cartan connection of type $(\mathfrak{g}, P)$ takes values in $\Lambda^2(\mathfrak{g}/\mathfrak{g}^0)^*\otimes\mathfrak{g}$, which is naturally filtered, and the associated graded space  $\mathrm{gr}(\Lambda^2(\mathfrak{g}/\mathfrak{g}^0)^*\otimes\mathfrak{g})$ can be identified with $\Lambda^2\mathfrak{m}^*\otimes\mathfrak{g}$. The latter space is the space of $2$-cochains in the standard complex computing the Lie algebra cohomology $H^*(\mathfrak{m},\mathfrak{g})$. As before we denote by $\partial:\Lambda^k\mathfrak{m}^*\otimes\mathfrak{g}\to\Lambda^{k+1}\mathfrak{m}^*\otimes\mathfrak{g}$ the coboundary operators in that complex and we denote the $i$th grading component by  a subscript $i$.

\begin{definition}\label{normcond}\cite[Definition 3.3]{CapCartan}
A normalization  condition for Cartan geometries of type  $(\mathfrak{g},P)$  is  a  $P$-invariant linear subspace  $\mathcal{N}\subset\Lambda^2(\mathfrak{g}/\mathfrak{g}^0)^*\otimes\mathfrak{g}$ such that  for each $i>0$ the subspace $\mathrm{gr}(\mathcal{N})_i\subset (\Lambda^2\mathfrak{m}^*\otimes\mathfrak{g})_i$ is complementary to the image of
  $\partial:(\mathfrak{m}^*\otimes\mathfrak{g})_i\to (\Lambda^2\mathfrak{m}^*\otimes\mathfrak{g})_i$.
\end{definition}

\subsubsection{Analysis for marked contact twisted cubic structures}
 Recall the algebraic setup: Let $\mathfrak{g}$ be the Lie algebra of $\mathrm{G}_2$  endowed with its contact grading
$\mathfrak{g}=\bigoplus_{i=-2}^{2}\mathfrak{g}_i$
and
$\mathfrak{q}=\bigoplus_{i=-2}^{1}\mathfrak{q}_i$
 the  graded  subalgebra from Proposition \ref{p_1}.
In particular,  $\mathfrak{m}=\mathfrak{g}_{-2}\oplus\mathfrak{g}_{-1}=\mathfrak{q}_{-2}\oplus\mathfrak{q}_{-1}$ is the $5$-dimensional Heisenberg algebra.
We ask whether we can find a normalization condition for Cartan geometries of type $(\mathfrak{q},Q^0),$ where $Q^0$ is a Lie group with Lie algebra $\mathfrak{q}^0=\mathfrak{q}_0\oplus\mathfrak{q}_1$.

The  inclusion $\mathfrak{q}\hookrightarrow\mathfrak{g}$ induces inclusions of the corresponding cochain spaces and we obtain the following  commuting diagram
$$\begin{array}{cccccccccccc}
0&\longrightarrow&
\mathfrak{g}&
\stackrel{\tilde{\partial}}{\longrightarrow}&\mathfrak{m}^*\otimes\mathfrak{g}&\stackrel{\tilde{\partial}}{\longrightarrow}&\Lambda^2\mathfrak{m}^*\otimes\mathfrak{g}&\stackrel{\tilde{\partial}}{\longrightarrow}
 &\Lambda^3\mathfrak{m}^*\otimes\mathfrak{g}&\stackrel{\tilde{\partial}}{\longrightarrow}&\dots\\
& &\uparrow&  &\uparrow &  &\uparrow& &\uparrow& \\
0&\longrightarrow&\mathfrak{q}&\stackrel{\partial}{\longrightarrow}&\mathfrak{m}^*\otimes\mathfrak{q}&\stackrel{\partial}{\longrightarrow}&\Lambda^2\mathfrak{m}^*\otimes\mathfrak{q}&\stackrel{\partial}{\longrightarrow}&\Lambda^3\mathfrak{m}^*\otimes\mathfrak{g}&\stackrel{\partial}{\longrightarrow}&\dots
\end{array}$$
We know that  $H^1(\mathfrak{m},\mathfrak{g})_l=0$ and $H^1(\mathfrak{m},\mathfrak{q})_l=0$ for all $l>0$, since this is implied by the fact that $\mathfrak{g}$  and $\mathfrak{q}$ are the Tanaka prolongations  of $(\mathfrak{m},\mathfrak{g}_0)$ and  $(\mathfrak{m},\mathfrak{q}_0)$, respectively.

 The space of $2$-cochains of homogeneity one
 \begin{equation}\label{2cochain}(\Lambda^2\mathfrak{m}^*\otimes\mathfrak{g})_1=\Lambda^2\mathfrak{g}_{-1}^*\otimes\mathfrak{g}_{-1}\oplus\mathfrak{g}_{-2}^*\otimes\mathfrak{g}_{-1}\otimes\mathfrak{g}_{-2}\end{equation}
  is a completely reducible $\mathfrak{g}_0\cong\mathfrak{gl}(2,\mathbb{R})$ representation
  isomorphic, as a representation of the semisimple part ${\mathfrak{g}_0}^{ss}$, to
\begin{align}\label{decomp}
\overbrace{\underbrace{\smash{\bigodot}^5\mathbb{R}^2\oplus\smash{\bigodot}^3\mathbb{R}^2\oplus\smash{\bigodot}^3\mathbb{R}^2\oplus\mathbb{R}^2}_{\mathrm{Im}(\tilde{\partial})}\oplus\smash{\bigodot}^7\mathbb{R}^2}^{\mathrm{ker}(\tilde{\partial})}\oplus\smash{\bigodot}^3\mathbb{R}^2\,.
\end{align}
Hence $$H^2(\mathfrak{m},\mathfrak{g})_1\cong \smash{\bigodot}^7\mathbb{R}^2.$$
This fact can also be derived using Kostant's theorem (see \cite{yamaguchi, book}).

%

Next, it is visible from the decomposition \eqref{2cochain} that the inclusion $\mathfrak{q}\hookrightarrow\mathfrak{g}$ induces an identification $(\Lambda^2\mathfrak{m}^*\otimes\mathfrak{q})_1 =(\Lambda^2\mathfrak{m}^*\otimes\mathfrak{g})_1$. Likewise $(\Lambda^3\mathfrak{m}^*\otimes\mathfrak{q})_1= (\Lambda^3\mathfrak{m}^*\otimes\mathfrak{g})_1,$ and thus  $\mathrm{ker}(\partial)=\mathrm{ker}(\tilde{\partial})\subset(\Lambda^2\mathfrak{m}^*\otimes\mathfrak{g})_1$. We can see from \eqref{decomp} that $\mathrm{Im}(\tilde{\partial})$ has dimension $16$, and that it has an invariant complement isomorphic to $\smash{\bigodot}^7\mathbb{R}^2$ in $\mathrm{ker}(\tilde{\partial})$.
The image of $\partial:(\mathfrak{m}^*\otimes\mathfrak{q})_1\to(\Lambda^2\mathfrak{m}^*\otimes\mathfrak{q})_1$ is a $\mathfrak{q}_0$-submodule  $\mathrm{Im}(\partial)\subset\mathrm{Im}(\tilde{\partial})$ of dimension $\mathrm{dim}((\mathfrak{m}^*\otimes\mathfrak{q})_1)-\mathrm{dim}(\mathfrak{q}_1)=15$, where $(\mathfrak{m}^*\otimes\mathfrak{q})_1=\mathfrak{q}_{-1}^*\otimes\mathfrak{q}_0\oplus\mathfrak{q}_{-2}^*\otimes\mathfrak{q}_{-1}$,  and hence of codimension $1$ in $\mathrm{Im}(\tilde{\partial})$. In particular,
$$H^2(\mathfrak{m},\mathfrak{q})_1\cong H^2(\mathfrak{m},\mathfrak{g})_1\oplus\mathbb{R}\cong \smash{\bigodot}^7\mathbb{R}^2\oplus\mathbb{R}.$$

On the other hand, we have the following:
 \begin{proposition}\label{no_norm}
There is no $\mathfrak{q}_0$-invariant subspace complementary to the image of
$$\partial: (\mathfrak{m}^*\otimes\mathfrak{q})_1\to (\Lambda^2\mathfrak{m}^*\otimes\mathfrak{q})_1\,.$$
In particular, there exists no normalization condition in the sense of Definition \ref{normcond} for Cartan geometries of type $(\mathfrak{q}, Q^0)$.
\end{proposition}
\begin{proof}
Suppose such a $\mathfrak{q}_0$-invariant complement $\mathbb{W}$ exists, i.e., we have a $\mathfrak{q}_0$-invariant decomposition
$$\mathbb{W}\oplus\mathrm{Im}(\partial)=   (\Lambda^2\mathfrak{m}^*\otimes\mathfrak{q})_1.$$
To simplify the discussion, recall that $\mathrm{Im}(\partial)$ is a codimension one subspace of $\mathrm{Im}(\tilde{\partial})$, and consider
 $\mathbb{U}:=\mathbb{W}\cap\mathrm{Im}(\tilde{\partial})$; this is now a $1$-dimensional $\mathfrak{q}_0$-subrepresentation  of the $\mathfrak{g}_0$-representation $\mathrm{Im}(\tilde{\partial})$ such that
 $$\mathbb{U}\oplus\mathrm{Im}(\partial)=\mathrm{Im}(\tilde{\partial}).$$
 Now let $\tilde{\mathbb{U}}$ be the irreducible $\mathfrak{g}_0$-subrepresentation of $\mathrm{Im}(\tilde{\partial})$ generated by $\mathbb{U}$. The dimension of $\tilde{\mathbb{U}}$ is $>1$, since  \eqref{decomp} shows that there is no $1$-dimensional $\mathfrak{g}_0$-subrepresentation in $\mathrm{Im}(\tilde{\partial})$. In particular,  $\tilde{\mathbb{U}}$  has non-zero intersection with $\mathrm{Im}(\partial)$. So we now have a non-trivial $\mathfrak{q}_0$-invariant decomposition
 $$\mathbb{U}\oplus(\tilde{\mathbb{U}}\cap\mathrm{Im}(\tilde{\partial}))=\tilde{\mathbb{U}},$$
 where $\tilde{\mathbb{U}}$ is now a finite-dimensional \emph{irreducible} $\mathfrak{g}_0$-representation. But this is impossible.
\end{proof}

Proposition \ref{no_norm} also shows that there exists no Lie group $Q_0$ with Lie algebra $\mathfrak{q}_0$ such that Morimoto's ``Condition C'' (see \cite[Definition 3.10.1]{Morimoto}) is satisfied.

\section{Appendix}
\label{appendixg2}
For explicit computations we use the following basis of the Lie algebra $\mathfrak{g}$  of  $\mathrm{G}_2$. Consider the $7\times 7$ matrices
\begin{equation*}
\begin{aligned}
&A=\tiny{
\bma 0&\tfrac43\al_2&\tfrac43\al_0&\tfrac49\al_1-\al_3&-\tfrac49\al_1-\al_3&-\tfrac43\al_0&0\\
-\tfrac43\al_2&0&2\al_3&0&0&-2\al_3&-\tfrac43\al_2\\
-\tfrac43\al_0&-2\al_3&0&-\tfrac23\al_2+\tfrac32\al_4&\tfrac23\al_2+\tfrac32\al_4&0&-\tfrac43\al_0\\
-\tfrac49\al_1+\al_3&0&\tfrac23\al_2-\tfrac32\al_4&0&0&-\tfrac23\al_2+\tfrac32\al_4&-\tfrac49\al_1+\al_3\\
-\tfrac49\al_1-\al_3&0&\tfrac23\al_2+\tfrac32\al_4&0&0&-\tfrac23\al_2-\tfrac32\al_4&-\tfrac49\al_1-\al_3\\
-\tfrac43\al_0&-2\al_3&0&-\tfrac23\al_2+\tfrac32\al_4&\tfrac23\al_2+\tfrac32\al_4&0&-\tfrac43\al_0\\
0&-\tfrac43\al_2&-\tfrac43\al_0&-\tfrac49\al_1+\al_3&\tfrac49\al_1+\al_3&\tfrac43\al_0&0
\ema},\\
&B=\tiny{
\bma
0&0&\tfrac34\bet_2-\tfrac13\bet_3&0&0&-\tfrac34\bet_2-\tfrac13\bet_3&3\bet_1+\bet_4\\
0&0&0&\tfrac32\bet_2-\tfrac23\bet_3&\tfrac32\bet_2+\tfrac23\bet_3&0&0\\
-\tfrac34\bet_2+\tfrac13\bet_3&0&0&0&0&-3\bet_1+\bet_4&\tfrac34\bet_2+\tfrac13\bet_3\\
0&-\tfrac32\bet_2+\tfrac23\bet_3&0&0&-2\bet_4&0&0\\
0&\tfrac32\bet_2+\tfrac23\bet_3&0&-2\bet_4&0&0&0\\
-\tfrac34\bet_2-\tfrac13\bet_3&0&-3\bet_1+\bet_4&0&0&0&\tfrac34\bet_2-\tfrac13\bet_3\\
3\bet_1+\bet_4&0&\tfrac34\bet_2+\tfrac13\bet_3&0&0&-\tfrac34\bet_2+\tfrac13\bet_3&0
\ema},\\
&C=\tiny{
\bma 0&-\tfrac32\gam_3&-\tfrac98\gam_0&-\tfrac12\gam_2+\tfrac{27}{8}\gam_4&\tfrac12\gam_2+\tfrac{27}{8}\gam_4&-\tfrac98\gam_0&0\\
\tfrac32\gam_3&0&\gam_2&0&0&\gam_2&-\tfrac32\gam_3\\
\tfrac98\gam_0&-\gam_2&0&-\gam_1+\tfrac34\gam_3&\gam_1+\tfrac34\gam_3&0&-\tfrac98\gam_0\\
\tfrac12\gam_2-\tfrac{27}{8}\gam_4&0&\gam_1-\tfrac34\gam_3&0&0&\gam_1-\tfrac34\gam_3&-\tfrac12\gam_2+\tfrac{27}{8}\gam_4\\
\tfrac12\gam_2+\tfrac{27}{8}\gam_4&0&\gam_1+\tfrac34\gam_3&0&0&\gam_1+\tfrac34\gam_3&-\tfrac12\gam_2-\tfrac{27}{8}\gam_4\\
-\tfrac98\gam_0&\gam_2&0&\gam_1-\tfrac34\gam_3&-\gam_1-\tfrac34\gam_3&0&\tfrac98\gam_0\\
0&-\tfrac32\gam_3&-\tfrac98\gam_0&-\tfrac12\gam_2+\tfrac{27}{8}\gam_4&\tfrac12\gam_2+\tfrac{27}{8}\gam_4&-\tfrac98\gam_0&0
\ema},
\end{aligned}
\end{equation*}
where $\al_0,\al_1,\al_2,\al_3,\al_4,\bet_1,\bet_2,\bet_3,\bet_4,\gam_0,\gam_1,\gam_2,\gam_3,\gam_4$ are real constants. Then
\be \label{basis_g2}
E_0=\tfrac{\der A}{\der\al_0},\quad E_i=\tfrac{\der A}{\der\al_i},\quad E_{4+I}=\tfrac{\der B}{\der\bet_I},\quad E_{8+i}=\tfrac{\der C}{\der\gam_i},\quad E_{13}=\tfrac{\der C}{\der\gam_0},\ee
where $i=1,2,3,4$, and $I=1,2,3,4,$  define (as one can verify) a basis for $\mathfrak{g}_2$.
The basis is adapted to the contact grading   of $\mathfrak{g}=\mathfrak{g}_2$ in the sense that $E_0$ is contained in the grading component $\mathfrak{g}_{-2}$, $E_1, E_2, E_3, E_4$ are contained in $\mathfrak{g}_{-1}$, $E_5, E_6, E_7, E_8$ are contained in $\mathfrak{g}_0$, $E_9, E_{10}, E_{11},E_{12}$ in $\mathfrak{g}_1$, and $E_{13}$ in $\mathfrak{g}_{2}$.

Writing the Maurer-Cartan form $\Omega_{G_2}$ as $\Omega_{G_2}=\theta^i E_i$, where the $\theta^i$ are now left-invariant $\mathbb{R}$-valued $1$-forms,
  the Maurer-Cartan equations for $\mathrm{G}_2$ are of the following form:
  \begin{equation}\begin{aligned}\label{MaurerCartan}
&\der\theta^0=-6\theta^0\dz\theta^5+\theta^1\dz\theta^4-3\theta^2\dz\theta^3\\
&\\
&\der\theta^1=6\theta^0\dz\theta^9-3\theta^1\dz\theta^5-3\theta^1\dz\theta^8+3\theta^2\dz\theta^7\\
&\der\theta^2=2\theta^0\dz\theta^{10}+\theta^1\dz\theta^6-3\theta^2\dz\theta^5-\theta^2\dz\theta^8+2\theta^3\dz\theta^7\\
&\der\theta^3=2\theta^0\dz\theta^{11}+2\theta^2\dz\theta^6-3\theta^3\dz\theta^5+\theta^3\dz\theta^8+\theta^4\dz\theta^7\\
&\der\theta^4=6\theta^0\dz\theta^{12}+3\theta^3\dz\theta^6-3\theta^4\dz\theta^5+3\theta^4\dz\theta^8\\
&\\
&\der\theta^5=2\theta^0\dz\theta^{13}-\theta^1\dz\theta^{12}+\theta^2\dz\theta^{11}-\theta^3\dz\theta^{10}+\theta^4\dz\theta^9\\
&\der\theta^6=6\theta^2\dz\theta^{12}-4\theta^3\dz\theta^{11}+2\theta^4\dz\theta^{10}+2\theta^6\dz\theta^8\\
&\der\theta^7=-2\theta^1\dz\theta^{11}+4\theta^2\dz\theta^{10}-6\theta^3\dz\theta^9-2\theta^7\dz\theta^8\\
&\der\theta^8=-3\theta^1\dz\theta^{12}+\theta^2\dz\theta^{11}+\theta^3\dz\theta^{10}-3\theta^4\dz\theta^9-\theta^6\dz\theta^7\\
&\\
&\der\theta^9=-\theta^1\dz\theta^{13}-3\theta^5\dz\theta^9-\theta^7\dz\theta^{10}+3\theta^8\dz\theta^9\\
&\der\theta^{10}=-3\theta^2\dz\theta^{13}-3\theta^5\dz\theta^{10}-3\theta^6\dz\theta^9-2\theta^7\dz\theta^{11}+\theta^8\dz\theta^{10}\\
&\der\theta^{11}=-3\theta^3\dz\theta^{13}-3\theta^5\dz\theta^{11}-2\theta^6\dz\theta^{10}-3\theta^7\dz\theta^{12}-\theta^8\dz\theta^{11}\\
&\der\theta^{12}=-\theta^4\dz\theta^{13}-3\theta^5\dz\theta^{12}-\theta^6\dz\theta^{11}-3\theta^8\dz\theta^{12}\\
&\\
&\der\theta^{13}=-6\theta^5\dz\theta^{13}-6\theta^9\dz\theta^{12}+2\theta^{10}\dz\theta^{11}.
 \end{aligned}\end{equation}

The nine generators $(E_0, E_1, E_2, E_3, E_4, E_5, E_6, E_8, E_{12})$ marked by black dots below are a basis for a subalgebra $\mathfrak{q}\cong{\mathfrak{p}_1}$ having minimal intersection with $\mathfrak{p}_2=\mathfrak{g}_0\oplus\mathfrak{g}_1\oplus\mathfrak{g}_2$.

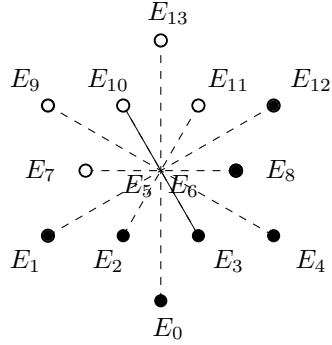
\begin{center}
\begin{tikzpicture}[scale=1]


\draw[dashed] (0,-1.732) -- (0,1.732);
 \draw[thick](0,1.732)circle(0.08);
  \filldraw[white](0,1.732)circle(0.06);
 \filldraw[black](0,-1.732)circle(0.08);

 \draw[dashed](-1.5,-0.866)--(1.5,0.866);
\draw[thick](-1.5,-0.866)circle(0.08);
\filldraw[black](-1.5,-0.866)circle(0.06);
 \draw[thick](1.5,0.866)circle(0.08);
  \filldraw[black](1.5,0.866)circle(0.06);

 \draw[dashed](1.5,-0.866)--(-1.5,0.866);
\filldraw[white](-1.5,0.866)circle(0.06);
\draw[thick](-1.5,0.866)circle(0.08);
 \filldraw(1.5,-0.866)circle(0.08);

\draw[dashed](-1,0)--(1,0);
\draw[thick](-1,0)circle(0.08);
\filldraw[white](-1,0)circle(0.06);
 \draw[thick](1,0)circle(0.08);
  \filldraw[black](1,0)circle(0.06);
\draw[thick](1,0)circle(0.08);

 \filldraw[dashed](0.5,-0.866)--(-0.5,0.866);
  \filldraw[black](0.5,-0.866)circle(0.08);
 \draw[thick](-0.5,0.866)circle(0.08);
  \filldraw[white](-0.5,0.866)circle(0.06);

  \draw[dashed](-0.5,-0.866)--(0.5,0.866);
 \filldraw[black](-0.5,-0.866)circle(0.08);
  \draw[thick](0.5,0.866)circle(0.08);
 \filldraw[white](0.5,0.866)circle(0.06);

\draw(0.1,2.1) node {$E_{13}$};
 \draw(0.9,1.2) node {$E_{11}$};
   \draw (-0.7,1.2) node {$E_{10}$};
    \draw (-1.8,1.2) node {$E_{9}$};
\draw (-1.6,0) node {$E_7$};
 \draw (2,1.2) node {$E_{12}$};
 \draw (1.6,0) node {$E_8$};
 \draw (2,-1.2) node {$E_{4}$};
  \draw (0.9,-1.2) node {$E_{3}$};
   \draw (-0.7,-1.2) node {$E_{2}$};
    \draw (-1.8,-1.2) node {$E_{1}$};
 \draw (-0.3,-0.2) node {$E_5$};
  \draw (0.3,-0.2) node {$E_6$};
 \draw (0.1,-2.1) node {$E_0$};
\end{tikzpicture}
\end{center}
The kernel of the forms $\theta^7, \theta^9, \theta^{10}, \theta^{11}, \theta^{13}$ is an integrable distribution. On each leaf of the foliation defined by  this distribution the forms $\theta^7, \theta^9, \theta^{10}, \theta^{11}, \theta^{13}$ vanish and the system \eqref{MaurerCartan} reduces to the Maurer-Cartan equations for $Q\cong \mathrm{P}_1$:
  \begin{equation}\begin{aligned}\label{MaurerCartan_Q}
&\der\theta^0 = -6 \theta^0\wedge\theta^5 + \theta^1\wedge \theta^4 - 3\theta^2\wedge\theta^3\\
&\der\theta^1 =  - 3\theta^1\wedge\theta^5 - 3\theta^1\wedge\theta^8 \\
&\der\theta^2 =   \theta^1\wedge\theta^6 - 3\theta^2\wedge\theta^5 - \theta^2\wedge\theta^8 \\
&\der\theta^3 =  2\theta^2\wedge\theta^6 -3\theta^3\wedge\theta^5 + \theta^3\wedge\theta^8 \\
&\der\theta^4=6 \theta^0\wedge\theta^{12}+3\theta^3\wedge\theta^6-3\theta^4\wedge\theta^5+3\theta^4\wedge\theta^8\\
&\der\theta^5=-\theta^1\wedge\theta^{12}\\
&\der\theta^6=6\theta^2\wedge\theta^{12}+2\theta^6\wedge\theta^8\\
&\der\theta^8=-3\theta^1\wedge\theta^{12}\\
&\der\theta^{12}= -3\theta^5\wedge\theta^{12}- 3\theta^8\wedge\theta^{12}
 \end{aligned}\end{equation}

 \end{document}